\documentclass[a4paper,reqno, 10pt]{amsart}
\usepackage{amsmath,amsfonts,amssymb,amsthm}
\usepackage[usenames,dvipsnames]{color}
\usepackage[utf8]{inputenc}    % utf8 support       %!!!!!!!!!!!!!!!!!!!!
\usepackage[T1]{fontenc}       % code for pdf file  %!!!!!!!!!!!!!!!!!!!!
\usepackage[autostyle]{csquotes}
\usepackage[square,sort,numbers]{natbib}
\usepackage[margin = 1in]{geometry}
\usepackage{mathabx}
\usepackage{mathtools, xcolor}
\usepackage{bbm}
\usepackage{enumerate}
\usepackage{lipsum}
\usepackage{longtable}
\newcommand\blfootnote[1]{%
  \begingroup
  \renewcommand\thefootnote{}\footnote{#1}%
  \addtocounter{footnote}{-1}%
  \endgroup
}

\theoremstyle{plain}
\begingroup
\newtheorem{theorem}{Theorem}[section]
\newtheorem{lemma}[theorem]{Lemma}
\newtheorem{proposition}[theorem]{Proposition}
\newtheorem{corollary}[theorem]{Corollary}

\endgroup
\theoremstyle{definition}
\begingroup

 \newtheorem{definition}[theorem]{Definition}
\endgroup
\theoremstyle{remark}
\begingroup
\newtheorem{rmk}[theorem]{Remark}

\endgroup
\numberwithin{equation}{section}
\theoremstyle{remark}
\begingroup
\endgroup

\mathchardef\emptyset="001F

\newcommand{\op}[1]{{\rm{#1}}}
\newcommand{\R}{\mathbb{R}}
\newcommand{\N}{\mathbb{N}}

\newcommand*{\medcup}{\mathbin{\scalebox{1.5}{\ensuremath{\cup}}}}
\newcommand{\verti}[1]{\ensuremath{\left\lvert #1 \right\rvert}}

\newcommand{\parenth}[1]{\ensuremath{\left( #1 \right)}}

\begin{document}
\title[Unconstrained polarization (Chebyshev) problems]{Unconstrained polarization (Chebyshev) problems: basic properties and Riesz kernel asymptotics}

\author{Douglas Hardin}
\address{Vanderbilt University}
\email{douglas.hardin@Vanderbilt.edu}

\author{Mircea Petrache}
\address{Pontificia Universidad Catolica de Chile%, Facultad de Matematicas, Av. Vicuna Mackenna 4860, Santiago, 6904441, Chile.
}
\email{mpetrache@mat.uc.cl}

\author{Edward B. Saff}
\address{Vanderbilt University}
\email{Edward.B.Saff@Vanderbilt.edu}

\date{\today}
\maketitle \blfootnote{\textbf{Acknowledgements:} The research of the authors was supported, in part, by the US National Science Foundation grant DMS-1516400. The authors are grateful for the stimulating work environment provided by ICERM (Brown University), during the Semester Program on ``Point Configurations in Geometry, Physics and Computer Science'' in spring 2018 supported by the National Science Foundation under Grant No. DMS-1439786.}

\begin{abstract}
 We introduce and study the unconstrained polarization (or Chebyshev) problem which requires to find an $N$-point configuration that maximizes the minimum value of its potential over a set $A$ in $p$-dimensional Euclidean space. This problem is compared to the constrained problem in which the points are required to belong to the set $A$. We find that for Riesz kernels $1/|x-y|^s$ with $s>p-2$ the optimum unconstrained configurations concentrate close to the set $A$ and based on this fundamental fact we recover the same asymptotic value of the polarization as for the more classical constrained problem on a class of $d$-rectifiable sets. We also investigate the new unconstrained problem in special cases such as for spheres and balls. In the last section we formulate some natural open problems and conjectures.
\end{abstract}

\medskip


\medskip

\tableofcontents
%%%%%%%%%%%%%%%%%%%%%%%%%%%%%%%%%%%%%%%%%%%%%%%%%%%%%%%%%%%%%%%%%%%%%%%%%
\section{Introduction and statement of main results}
%%%%%%%%%%%%%%%%%%%%%%%%%%%%%%%%%%%%%%%%%%%%%%%%%%%%%%%%%%%%%%%%%%%%%%%%%
%
Let $A, B$ be two non-empty sets, and $K:B\times A\to (-\infty,+\infty]$ be a kernel (or pairwise potential). For $N\in \N$ we consider the max-min optimization problem
\begin{equation}\label{2plate_polarization}
P_K(A,\omega_N):=\inf_{y\in A}\sum_{i=1}^NK(x_i,y),\quad \mathcal P_K(A,B,N):=\sup_{\omega_N\subset B}P_K(A,\omega_N),
\end{equation}
where the maximum is taken over $N$-point multisets $\omega_N=\{x_1,\ldots,x_N\}\subset B$. (Note that a multiset is a list where elements can be repeated.) The determination of \eqref{2plate_polarization} is called the \emph{two-plate polarization (or Chebyshev) problem} (see Proposition~\ref{prop:cheby} below for the link to the theory of Chebyshev polynomials, justifying this name). For background and motivation of the study of polarization problems,
see \cite[Chapter 14]{book}. If $A^\prime\subset A$ and $B^\prime\subset B$ we note the basic monotonicity properties
\begin{equation}\label{ineq_pol}
\mathcal P_K(A^\prime,B,N)\ge \mathcal P_K(A,B,N), \qquad \mathcal P_K(A,B^\prime,N)\le \mathcal P_K(A,B,N).
\end{equation}
The case $A=B$ of \eqref{2plate_polarization}, also known as \emph{the single-plate polarization (or Chebyshev) problem} for $A$, has been the more studied so far (see \cite{2013erdelyisaff, Kershner1939, book}); and for it we introduce the notation
\begin{equation}\label{ps}
\mathcal P_K(A,N):=\mathcal P_K(A,A,N).
\end{equation}
A related quantity is the value of the \emph{minimum $N$-point $K$-energy}\footnote{See, e.g., the recent book \cite{book} and recent articles \cite{BetSan2018}, \cite{PetNod2018}}, given by
\begin{equation}\label{min_en}
\mathcal E_K(A,N):=\inf_{\omega_N\subset A}\sum_{i=1}^N\sum_{\substack{j:j\neq i\\ j=1}}^N K(x_i,x_j).
\end{equation}
If $N\ge 2$, $A\subset B$ are compact sets and $K: B\times B\to (-\infty,+\infty]$ is a symmetric function, we have the following relation between the above quantities (see \cite[Prop.\,14.1.1]{book}, \cite[Thm.\,2.3]{2013erdelyisaff})
\begin{equation}\label{bounds}
\mathcal P_K(A,B,N)\ge \mathcal P_K(A,N)\ge \frac{\mathcal E_K(A,N+1)}{N+1}\ge\frac{\mathcal E_K(A,N)}{N-1}.
\end{equation}
The goal of this article is to study the case $A\subset B=\R^p$ of \eqref{2plate_polarization}, in which the configurations $\omega_N$ are \emph{unconstrained}, and we use the notation
\begin{equation}\label{psstar}
\mathcal P_K^*(A,N):=\mathcal P_K(A,\R^p,N)=\sup_{\omega_N\subset\R^p}P_K(A,\omega_N).
\end{equation}
Directly from \eqref{ineq_pol} and from the definitions \eqref{ps} and \eqref{psstar}, we find that
\begin{equation}\label{inclusion}
\mathcal P_K^*(A^\prime,N)\ge \mathcal P_K^*(A,N)\quad\mbox{whenever}\quad A^\prime\subset A,
\end{equation}
and, for all $A\subset \R^p$,
\begin{equation}\label{inequality}
\mathcal P_K(A,N)\le \mathcal P_K^*(A,N).
\end{equation}
Our results are motivated by the study of the important class of kernels called \emph{Riesz $s$-potentials}:
\begin{equation}\label{def_Ks}
K_s(x,y):=\left\{\begin{array}{cl}
                 |x-y|^{-s}&\mbox{ if }s>0\ ,\\[3mm]
                 -\log|x-y|&\mbox{ if }s=0\ ,\\[3mm]
                 -|x-y|^{-s}&\mbox{ if }s<0\ .
                 \end{array}
\right.
\end{equation}
 In \eqref{def_Ks}, we define $K_s(x,x)=+\infty$ if $s\ge 0$.  For brevity we set
\begin{equation}\label{def_ps}
P_s(A,\omega_N):=P_{K_s}(A,\omega_N),\quad \mathcal P_s(A,N):=\mathcal P_{K_s}(A,N),\quad \mathcal P_s^*(A,N):=\mathcal P_{K_s}^*(A,N).
\end{equation}
Note that the monotonicity property \eqref{inclusion} is not true for $\mathcal P_s(A,N)$ (see \cite[Sec. 14.2]{book}) which, in some cases, may make the problem $\mathcal P_s^*(A,N)$ more tractable than $\mathcal P_s(A,N)$. We shall refer to \eqref{ps} as the \emph{constrained polarization} problem and to \eqref{psstar} as the \emph{unconstrained} problem. 

\medskip

\noindent The above definition \eqref{def_Ks} for $s=0$ is justified by the results of Propositions \ref{prop:bosuwan} and \ref{prop:bosuwan*}, which say that optimal configurations for $\mathcal P_0$ are the limits as $s\downarrow 0$ of optimal configurations for the problems $\mathcal P_s$. The study of $s$-polarization for large values of $s$ is related to \emph{best-covering problems}; the limits of \eqref{def_ps} as $s\to\infty$ yield best-covering constants, also treated in Propositions \ref{prop:bosuwan} and \ref{prop:bosuwan*}. In preparation for these propositions, we give the following definitions:
\begin{definition}\label{covering}
If $A\subset \R^p$ is a non-empty set, then the \emph{covering radius} of a configuration $\omega_N=\{x_1,\ldots,x_N\}$ with respect to the set $A$ is
\begin{equation}\label{coveringrad}
\eta(\omega_N, A)=\sup_{x\in A}\min_{1\le i \le N}|x-x_i|\ .
\end{equation}
The \emph{minimal $N$-point covering radius of a set $A$ relative to the set $B$} is defined as
\begin{equation}\label{relcoverrad}
\eta_N(A,B):=\inf\left\{\eta(\omega_N,A):\ \omega_N\subset B\right\}\ .
\end{equation}
The \emph{minimal $N$-point covering radius $\eta_N(A)$ of $A$} and the \emph{minimal $N$-point unconstrained covering radius $\eta_N^*(A)$ of $A$} are given by:
\begin{equation}\label{covercover*}
\eta_N(A):= \eta_N(A,A)\ ,\quad\quad \eta_N^*(A):=\eta_N(A,\R^d)\ .
\end{equation}
\end{definition}
\begin{proposition}[{\cite[Thm.\,III.2.1, Thm.\,III.2.2]{prop:bosuwan2013two}}]\label{prop:bosuwan}
If $N\in \N$ is fixed and $A$ is an infinite compact subset of $\R^p$, then
\begin{eqnarray}
\lim_{s\to\infty}\left(\mathcal P_s(A,N)\right)^{1/s} &=&\frac1{\eta_N(A)}\ ,\label{sinfty}\\
\lim_{s\to 0^+}\frac{\mathcal P_s(A,N)-N}{s}&=&\mathcal P_0(A,N)\ .\label{szero}
\end{eqnarray}
%
% where $\eta_N(A)$ is the $N$-point best-covering radius of $A$. 
Moreover, every cluster point of $\mathcal P_s(A,N)$-optimizers in \eqref{sinfty} is an optimal configuration for $\eta_N(A)$ and every cluster point of $\mathcal P_s(A,N)$-optimizers in \eqref{szero} is an optimal configuration for $\mathcal P_0(A,N)$.
\end{proposition}
The above results have their basis in the observations:
$$\lim_{s\to \infty}\left(\sum_{j=1}^N|x_j-a|^{-s}\right)^{1/s}= \left(\min_j(|x_j-a|\right)^{-1}\text{ and } \lim_{s\to 0^+}\frac{|x_j-a|^{-s}-1}{s}=-\log |x_j-a|.
$$

A generalization of \eqref{sinfty} for the problem $\mathcal P_s(A,B,N)$ is presented in \cite[\S 14.4]{book}. By the same proof as in \cite{prop:bosuwan2013two}, we find the analogous asymptotics for the $\mathcal P_s^*(A,N)$ problem:
\begin{proposition}\label{prop:bosuwan*}
The assertions of Proposition \ref{prop:bosuwan} hold if we replace $\mathcal P_s(A,N)$ and $\eta_N(A)$, respectively, by $\mathcal P_s^*(A,N)$ and $\eta_N^*(A)$.
\end{proposition}
An important case, which justifies the alternative name ``Chebyshev problem'', is the setting of $s=0$, $p=2$, namely the study of polarization problems for the kernel $K(x,y)=-\log|x-y|$ in $\mathbb R^2$, here identified with $\mathbb C$. Indeed let $A\subset \mathbb C$ be an infinite compact set. A monic complex polynomial $T_N^A$ of degree $N$ is called \emph{the Chebyshev polynomial of degree $N$ corresponding to the set $A$} if $\|T_N^A\|_A\le \|p\|_A$ for any monic complex polynomial $p$ of degree $N$, where
\begin{equation}\label{chebynorm}
\|p\|_A:=\max_{z\in A}|p(z)|
\end{equation}
is the max norm on the set $A$. Then denoting by $z_1,\ldots,z_N$ the zeros of the polynomial $p$ repeated according to their multiplicity and using an algebraic manipulation, we rewrite \eqref{chebynorm} in the equivalent form
\begin{equation}\label{chebynew}
\log\frac{1}{\|p\|_A}=\log\frac{1}{\max_{z\in A}\prod_{j=1}^N|z-z_j|}=\min_{z\in A}\sum_{j=1}^N\log\frac{1}{|z-z_j|},
\end{equation}
which directly gives a proof of the following well-known result:
\begin{proposition}\label{prop:cheby}
Let $A\subset \mathbb C$ be an infinite compact set. A multiset $\omega_N^*=\{z_1,\ldots,z_N\}$ is optimal for the maximal unconstrained polarization problem on $A$ with respect to the logarithmic potential if and only if $T(z)=(z-z_1)\cdots(z-z_N)$ is the Chebyshev polynomial for $A$. If $T_N^A$ is such a polynomial, then $\mathcal P_0^*(A,N)=\log(1/\|T_N^A\|_A)$.
\end{proposition}
It is well known that $T_N^A$ is unique (see \cite[Thm.\,III.23]{tsuji}), and that by a classical result of Fej\'{e}r \cite{Fe} the zeros of $T_N^A$ lie in the convex hull of $A$; but need not lie on $A$. For example, an application of the maximum modulus principle shows that $T_N^A(z)=z^N$ is the unique Chebyshev polynomial for the unit circle $A=\mathbb S^1$. This observation generalizes to the principle that optimal unconstrained polarization configurations may accumulate away from the set $A$ if $K(x,y)$ is superharmonic in $y$ (see Proposition~\ref{prop:slep-2} below).

\medskip

We recall (see \cite{ohtsuka1967various}) that in a very general setting we may relate the two-plate polarization problem to the
so-called \emph{continuous two-plate polarization (Chebyshev) constant} $T_K(A,B)$ defined in \eqref{t_kab} below. The next theorem describes the large $N$ limit of discrete two-plate polarization.
\begin{theorem}[{\cite{ohtsuka1967various}}]\label{thm:ohtsuka}
Let $X,Y$ be locally compact nonempty Hausdorff spaces, $A\subset X$ be compact nonempty and $B\subset Y$ be nonempty, and the kernel $K:X\times Y\to(-\infty,+\infty]$ be a lower semi-continuous function. Then
\begin{equation}\label{ohtsuka}
\lim_{N\to\infty}\frac{\mathcal P_K(A,B,N)}{N}=T_K(A,B),
\end{equation}
where
\begin{equation}\label{t_kab}
T_K(A,B):=\sup_{\mu\in\mathcal M_1(B)}\inf_{x\in A}\int K(x,y)d\mu(y)\in(-\infty,+\infty],
\end{equation}
and $\mathcal M_1(B)$ is the set of all probability measures with compact support contained in $B$.
\end{theorem}
% %
% For brevity we write
% %
% \begin{equation}\label{tktkstar}
% T_K(A):=T_K(A,A),\quad T_K^*(A):= T_K(A,Y).
% \end{equation}
% %
The integral $\int K(x,y)d\mu(y)$ in \eqref{t_kab} is called the {\em $K$-potential of $\mu$}.

\medskip

  Known results relating the discrete and continuous  one-plate  ($A=B$) polarization problems  for a {\em continuous} kernel  directly extend to the  two-plate case:
\begin{theorem}[{\cite[Prop.\,14.6.6]{book}}]\label{thm:contker}
Let $A$ and $B$ be two nonempty compact metric spaces, and $K\in C(A\times B)$. A sequence $\{\omega_N\}_{N=1}^\infty$ of $N$-point configurations on $B$ satisfies
\begin{equation}\label{limit_cont}
\lim_{N\to\infty} \frac{P_K(A,\omega_N)}{N}=T_K(A,B)
\end{equation}
if and only if every weak-$*$ limit measure $\mu$ of the sequence of the normalized counting measures
\[
\left\{\nu(\omega_N):=\frac{1}{N}\sum_{x\in\omega_N}\delta_x\right\}_{N=1}^\infty
\]
is an extremal measure for the continuous $2$-plate polarization problem; i.e., it satisfies
\begin{equation}\label{extmeasure}
T_K(A,B)=\inf_{x\in A}\int K(x,y)d\mu(y).
\end{equation}
\end{theorem}

We remark that there are few  results 
regarding uniqueness of the extremal measure for the above problem (e.g., see \cite{2013erdelyisaff},  \cite{reznikov2016minimum}, and \cite{Simanek2015}).

\subsection{Main results}

 Our first important property can be viewed as a generalization of
the aforementioned result of Fej\'{e}r \cite{Fe} for zeros of Chebyshev polynomials.
 It asserts that for a large class of kernels, the two problems $P_K^*(A,N)$ and $P_K(A,N)$  are equivalent when $A$ is convex. Hereafter, we always assume $A,B\subset \mathbb R^p, A\neq\emptyset$ and let $\op{conv}(A)$ denote the \emph{convex hull} of $A$.
\begin{proposition}\label{prop:covex_equal}
Let $f:[0,+\infty)\to(-\infty,+\infty]$ be a strictly decreasing function and let $K(x,y):=f(|x-y|)$. If $A\subset\mathbb R^p$ is a compact set, then any configuration $\omega_N^*=\{x_1,\ldots,x_N\}$ such that $P_K(A,\omega_N^*)=\mathcal P_K^*(A,N)<+\infty$ has the property that $x_i\in\op{conv}(A)$ for each $1\le i\le N$.
In particular, if $A$ is convex and $\mathcal P_K^*(A,N)<+\infty$ then $\omega_N^*\subset A$ and $\mathcal P_K(A,N)=\mathcal P_K^*(A,N)$.
\end{proposition}
\begin{proof}
Assume to the contrary that some point of $\omega_N^*$, say $x_1$, satisfies $x_1\notin \op{conv}(A)$. Then after replacing $x_1$ by the nearest-point projection $\pi_{\op{conv}(A)}(x_1)$, the sum $\sum_{i=1}^N K(x_i,y)$ strictly increases, contradicting the optimality of $\omega_N^*$.
\end{proof}

In the following result and hereafter we denote by $\# \omega$ the cardinality of a multiset $\omega$ including repetitions.  We utilize the following  notation for the   $\epsilon$-neighborhood of a set:
\begin{equation}\label{fatten_a}
A_\epsilon:=\{x\in\R^p:\ \op{dist}(x, A)<\epsilon\}.
\end{equation}

Our next result is also a generalization of a well known property of
the zeros of Chebyshev polynomials $T_N^A$ that lie outside the polynomial 
convex hull of the set $A\subset \mathbb{C}.$ Namely, for any compact subset $F$ of the
unbounded component of the complement of $A,$ there is a number $M=M_F$
depending only on $F$ such that each $T_N^A$ has at most $M$ zeros in $F$
(see e.g. \cite[Theorem III.3.4]{SaTo}). The proof of the following theorem which concerns Riesz  kernels will be given in Section \ref{sec:proof_bound_K}.

\begin{theorem}\label{thm:bound_K}
Let $A\subset\R^p$ be a compact set and assume $s > p-2$, $p\ge 2$. There exist $\kappa_{s,p}, c_{s,p}>0$ depending only on $p$ and $s$, such that for every $\epsilon>0$, if $\mathcal P_s^*(A,N)<+\infty$ and $\omega_N^*=\{x_1,\ldots,x_N\}$ satisfies
\begin{equation}\label{optimizing}
P_s(A,\omega_N^*)=\mathcal P_s^*(A,N),
\end{equation}
(i.e., $\omega_N^*$ is an unconstrained $N$-point maximal $s$-Riesz polarization  configuration),
then
\begin{equation}\label{bound_k_1}
\# \left(\omega_N^*\setminus A_\epsilon\right)\le \kappa_{s,p}\frac{\mathcal L_p\left(\left(\op{conv}(A)\right)_{\epsilon c_{s,p}}\right)}{\epsilon^p},
\end{equation}
where $\mathcal L_p$ denotes $p$-dimensional  Lebesgue measure on $\mathbb R^p$.
\end{theorem}
As a direct consequence of Theorem \ref{thm:bound_K}, for any compact set $B$ in the complement of $A$ the number of points of $\omega_N^*$ in $B$ is uniformly bounded in $N$.  In particular, we get the following:
\begin{corollary}\label{cor:weakconv}
Under the hypotheses of Theorem \ref{thm:bound_K}, if $\left(\omega_N^*\right)_{N\in \N}$ is a sequence of unconstrained $N$-point maximal $s$-Riesz polarization  configurations, then any weak* cluster point  of the sequence 
\begin{equation}
\nu(\omega_N^*):=\frac{1}{N}\sum_{x\in\omega_N^*}\delta_{x } ,\qquad N=1,2,3,\ldots,
\end{equation}
 is a probability measure supported on $A$.
\end{corollary}
 We now  describe  new asymptotic results for $\mathcal P_s^*(A,N)$ for fixed $s$ and $N\to\infty$ and their connection to the previously known asymptotics for $\mathcal P_s(A,N)$.  We define the \emph{renormalization factors} and relevant asymptotic quantities as follows:
\begin{equation}\label{tau_sd}
\tau_{s,d}(N):=\left\{
\begin{array}{ll}
N^{s/d},&\ s>d,\\
N\log N,&\ s=d,\\
N,&\ s<d,
\end{array}
\right.
\end{equation}
and
\begin{subequations}\label{h_sd}
\begin{equation}
\underline h_{s,d}^*(A):=\liminf_{N\to\infty}\frac{\mathcal P_s^*(A,N)}{\tau_{s,d}(N)},\quad \overline h_{s,d}^*(A):=\limsup_{N\to\infty}\frac{\mathcal P_s^*(A,N)}{\tau_{s,d}(N)}.
\end{equation}
If $\underline h_{s,d}^*(A)=\overline h_{s,d}^*(A)$, we set
\begin{equation}
h_{s,d}^*(A):=\lim_{N\to\infty}\frac{\mathcal P_s^*(A,N)}{\tau_{s,d}(N)}.
\end{equation}
\end{subequations}
For the problem $\mathcal P_s(A,N)$, the quantities $\underline h_{s,d}(A)$, $\overline h_{s,d}(A)$ and $h_{s,d}(A)$ were analogously defined in \cite{2014borodachovbosuwan} and \cite{bhrs2016preprint}.

\medskip

As a consequence of Theorem \ref{thm:bound_K} in combination with a new geometric deformation technique for optimizers of $\mathcal P_s^*(A,N)$ given in Proposition \ref{prop:replace_points}, we provide conditions that the  asymptotics of $\mathcal P_s^*(A,N)$ are equal to those of $\mathcal P_s(A,N)$.  For this purpose, we make use of the following definition.
\begin{definition}\label{def:regularset}
Let $d'>0$ and let $p>0$ be an integer. A compact set $A\subset\mathbb R^p$ is \emph{$d'$-regular} if there exists a measure $\lambda$ supported on $A$ and a positive constant $C$ such that for any $x\in A$ and $r<\op{diam}(A)$ there holds
\begin{equation}\label{regularset}
C^{-1}r^{d'}\le \lambda(B(x,r))\le Cr^{d'}\ .
\end{equation}
%
% The set $A$ is called \emph{$d'$-regular at $x\in A$} if for some positive $r_1>0$ the set \mip{$A\cap \overline{B(x,r_1)}$} is $d'$-regular.
\noindent A measure $\mu$ is called \emph{upper-$d$-regular} at $x$ if for some constant $c(x)$ and any $r>0$ there holds
\begin{equation}\label{regularmeasure}
\mu(B(x,r))\le c(x)r^d\ .
\end{equation}
\end{definition}

Hereafter, we denote by $\mathcal H_d$ the $d$-dimensional Hausdorff measure on $\R^p$, $d\le p$,   normalized so that the $\mathcal H_d$-measure
of a $d$-dimensional unit cube embedded in $\R^p$ is 1. Furthermore, for a compact set $A\subset \mathbb R^p$ with $0\le s<\mathrm{dim}_{\mathcal H}(A)$ (where $\mathrm{dim}_{\mathcal H}$ denotes the Hausdorff dimension), the \emph{equilibrium measure} $\mu_{s,A}$ is the unique probability measure supported on $A$ that minimizes
\[
\int\int K_s(x,y) d\mu(x)\,d\mu(y)
\]
over all probability measures supported on $A$.
\begin{theorem}\label{thm:compact_body_old}
For integers $p,d$ such that $p\ge 2$, $1\le d\le p$ and $A\subset\R^p$ compact, suppose that one of the following conditions holds:
\begin{enumerate}
\item[\rm(i)] $s>\max\{d,p-2\}$ and $\mathcal H_d(A)>0$,
\item[\rm(ii)] $p-2\le s<d$ and for some $d\le d'\le p$, the set $A$ is $d'$-regular and the equilibrium measure $\mu_{s,A}$ on $A$ is upper $d$-regular at every point $x\in A$.
\end{enumerate}

If the limit $h_{s,d}^*(A)$ exists as an extended real number, then the limit $h_{s,d}(A)$ also exists and
\begin{equation}\label{asymptotics_eq}
h_{s,d}(A)=h_{s,d}^*(A).
\end{equation}
Furthermore, if   (ii) holds, then $h_{s,d}^*(A)$ exists and is   finite; consequently \eqref{asymptotics_eq}
holds.
\end{theorem}
%%%%%%%%%%%%%%%%%%%%%%%%%%%%%%%%%%%%%%%%%%%%%%%%%%%%%%%%%%%%%%
%

We remark that our method of proof of Theorem~\ref{thm:compact_body_old} given in Section~\ref{sec:pf_compact_body_old} requires the weak-separation result from Proposition~\ref{prop:pointsep} which makes use of the assumptions  (i) and (ii) above.   

%Note that by Theorem \ref{thm:ohtsuka} applied to the case $B=\mathbb R^d$, with the kernel $K_s$ for $s<d$, the limit defining $h^*_{s,d}(A)$ exists as an extended real number  due to the fact that for $s<d$ we have $\tau_{s,d}(N)=N$. However we restrict to a subcase in Theorem \ref{thm:compact_body_old}, because our current proof requires the point separation as specified by Proposition \ref{prop:pointsep} in order to bound from below the number of points of $\mathcal P^*(A,N)$-optimizers that lie very close to but outside the set $A$.

\medskip

Concerning the actual values of the quantities in \eqref{asymptotics_eq}, it is established  in  \cite{bhrs2016preprint} (also, see \cite[Chapter 14]{book}) that, for $A$ equal the unit cube in $\R^p$ and $s\ge p$, the limit 
\begin{equation}\label{sigmaspdef}
\sigma_{s,p}:=h_{s,p}([0,1]^p)
\end{equation}
exists as a finite and positive number.

In preparation for the following theorems, we say that a sequence of $N$-point configurations, denoted $\Omega:=\{\omega_N\}_{N\ge 1}$, is  \emph{asymptotically extremal for the unconstrained problem} if $\lim_{N\to\infty}P_s(A,\omega_N)/\mathcal P_s^*(A,N) =1$, with a similar definition for the constrained problem.

\begin{theorem}\label{thm:compact_body}
   If $A\subset\mathbb R^p$ is a compact set and $s\ge p$, then 
\begin{equation}\label{asymptotics}
h_{s,p}^*(A)=h_{s,p}(A)=\frac{\sigma_{s,p}}{\mathcal L_p(A)^{s/p}}.
\end{equation}
Moreover, if $\mathcal L_p(A)>0$, then for any asymptotically extremal sequence $\Omega=\{\omega_N\}_{N\ge 1}$ (for either the constrained or unconstrained polarization problem)  we have the weak-$*$ convergence
\begin{equation}
\frac{1}{N}\sum_{x_i\in\omega_N}\delta_{x_i}\stackrel{*}{\rightharpoonup}\frac{\mathcal L_p|_A}{\mathcal L_p(A)}\quad\mbox{as}\quad N\to\infty,
\end{equation}
where $\mathcal L_p|_A:=\mathcal{L}_p(\cdot\cap A)$ is the restriction to $A$ of  $\mathcal L_p$.

% Consequently (from Theorem~\ref{thm:compact_body_old}), if $A\subset\mathbb R^p$ is a compact set  and $s>p$, then
%%
%\begin{equation}\label{clopen_asy}
%h_{s,p}(A)=\frac{\sigma_{s,p}}{\mathcal L_p(A)^{s/p}}.
%\end{equation}

\end{theorem}
%%%%%%%%%%%%%%%%%%%%%%%%%%%%%%%%%%%%%%%%%%%%%%%%%%%%%%%%%%%%%%
%
{\bf See Appendix \ref{app-corrigendum} for a correction to the statement and proof of Theorem \ref{thm:compact_body}.}

\bigskip

 We further note the following:
\begin{itemize}
\item The second equality in \eqref{asymptotics}   for $s>p$  improves upon the corresponding constrained result in \cite{bhrs2016preprint} which  required the additional assumption that $\mathcal L_p(\partial A)=0$.
\item For $s=p$, the second equality in \eqref{asymptotics} was proved in 
\cite{2014borodachovbosuwan}.
\item Our next result, Theorem~\ref{thm:rectifiable}, is a generalization of Theorem~\ref{thm:compact_body} since the hypotheses of the former are trivially satisfied for $d=p$.    Theorem~\ref{thm:compact_body} is stated separately because it plays an essential role in the proof of the more general theorem and is of independent interest.
\item As shown in \cite{2013erdelyisaff}, it is known that $\sigma_{p,p}=\beta_p$, which is the volume of the $p$-dimensional unit ball.
\item As follows by the result of \cite{hks2013} for $\mathcal P_s(\mathbb S^1,N)$, $\sigma_{s,1}=2\zeta(s)(2^s-1)$ for $s>1$.
\item For $p=2, s>2$, the conjecture in \cite[\S2]{bhrs2016preprint} for $\sigma_{s,2}$ is equivalent to the conjecture that  $\sigma_{s,2}=(3^{s/2}-1)\zeta_\Lambda(s)/2$, where
\begin{equation}\label{zeta}
\zeta_\Lambda(s)=\sum_{(m,n)\in\mathbb Z^2\setminus\{(0,0)\}}\frac1{\left((n+m/2)^2+3m^2/4\right)^{s/2}}
\end{equation}
is the Epstein zeta-function for the hexagonal lattice $\Lambda\subset\R^2$.
\end{itemize}
%
%
%As a special case of Theorem \ref{thm:rectifiable} below, we also have $\sigma_{d,d}^*=\sigma_{d,d}$. It is proved in \cite{2013erdelyisaff} that $\sigma_{d,d}=\beta_d$, which is the volume of the $d$-dimensional unit ball.
%
%

Theorem~\ref{thm:compact_body} provides asymptotics for compact sets of full dimension.   We next consider embedded sets for which it is necessary to consider certain geometric constraints.
Following \cite{federer}, we say that a set $A\subset \mathbb R^p$ is \emph{$d$-rectifiable} if it can be written as $\phi(K)$ for $K\subset\mathbb R^d$ bounded and $\phi:K\to\mathbb R^p$ Lipschitz.  We note the following two important properties of a $d$-rectifiable set $A$ (see  \cite[Theorems 3.2.18 \&  3.2.39]{federer}):
 (a) if $A$ is closed then $\mathcal H_d(A)$ equals the $d$-dimensional Minkowski content $\mathcal M_d(A)$ (see Definition~\ref{def:mink} for the definition of Minkowski content) and (b)
for each $\epsilon>0$, $A$ can be written as the disjoint union
\begin{equation}\label{Adecomp}
 A=A_0\cup \bigcup_{j=1}^\infty \varphi_j(K_j),
\end{equation}
where $\mathcal H_d(A_0)=0$, the maps $\varphi_j:K_j\to\varphi_j(K_j)$ are $(1+\epsilon)$-biLipschitz and $K_j\subset\mathbb R^d$ are compact sets.
We remark that any set $A\subset \R^p$ of the form \eqref{Adecomp} is called {\em$(\mathcal H_d,d)$-rectifiable}.
Setting
\begin{equation}\label{strong_rect}
 R_k:=A_0\cup\bigcup_{j=k+1}^\infty\varphi_j(K_j),
\end{equation}
we may take $k$ large enough so that $\mathcal H_d(R_k)<\epsilon$.

For $\epsilon>0$ and positive integers $d\le p$, we say that ${\mathcal G}$ is a {\em $d$-dimensional $\epsilon$-Lipschitz graph in $\R^p$} if there is a $d$-dimensional subspace $H\subset \R^p$ and an
 $\epsilon$-Lipschitz mapping $\psi: H\to H^\perp$    such that $\mathcal{G}= \{h+\psi(h):h \in H\}$.   It is useful to note that given  an isometry $\iota:\R^d\to H$, the mapping $\varphi:\R^d\to \mathcal{G}$ defined by $\varphi(x)=\iota(x)+\psi(\iota(x))$ is a $(1+\epsilon)$-biLipschitz mapping.
Now we introduce the following stronger requirement, needed below.
\begin{definition}\label{sdrectDef}
We say that $A\subset \R^p$ is {\em strongly $(\mathcal H_d,d)$-rectifiable}
if for each $\epsilon>0$ there exist a  compact set  $R_\epsilon\subset \R^p$ with $\overline{\mathcal M}_d(R_\epsilon)<\epsilon$ and finitely many compact, pairwise disjoint sets $\widetilde {K}_1,\ldots, \widetilde{K}_k \subset \R^p$ such that
\begin{equation}\label{Bdecomp}
 A=R_\epsilon\cup\bigcup_{j=1}^k \widetilde{K}_j,
\end{equation}
where each $\widetilde K_j$ is contained in some $d$-dimensional $\epsilon$-Lipschitz graph $\mathcal{G}_j$ in $\R^p$.  \end{definition}
%\begin{definition}\label{sdrectDef}
%We say that $A\subset \R^p$ is {\em strongly $(\mathcal H_d,d)$-rectifiable}
%if for each $\epsilon>0$ there exist a  compact set  $R_\epsilon\subset \R^p$ with $\mathcal M_d(R_\epsilon)<\epsilon$ and finitely many compact sets $K_1,\ldots, K_k \subset \R^d$ together with $(1+\epsilon)$-biLipschitz mappings $\varphi_j:K_j\to\varphi_j(K_j)\subset \mathbb R^p$ such that $\{\varphi_j(K_j)\}_{j=1}^k$ are pairwise disjoint  and   the following  decomposition holds:
%\begin{equation}\label{Bdecomp}
% A=R_\epsilon\cup\bigcup_{j=1}^k\varphi_j(K_j).
%\end{equation}
%\end{definition}
As shown in Lemma \ref{C1case}, a compact subset of a $C^1$-manifold of dimension $d$ contained in $\R^p$, $d\le p$, is strongly $(\mathcal H_d,d)$-rectifiable.
It can be shown that any strongly $(\mathcal H_d,d)$-rectifiable set is $(\mathcal H_d,d)$-rectifiable.
\medskip

%
%%%%%%%%%%%%%%%%%%%%%%%%%%%%%%%%%%%%%%%%%%%%%%%%%%%%%%%%%%%%%%
%  Theorem 4
%%%%%%%%%%%%%%%%%%%%%%%%%%%%%%%%%%%%%%%%%%%%%%%%%%%%%%%%%%%%%%
%
\begin{theorem}\label{thm:rectifiable}
Let $d$ and $p$ be positive integers with $d\le p$, and $A\subset\R^p$ be a compact strongly $(\mathcal H_d,d)$-rectifiable set. If $s>d$, then
\begin{equation}\label{thm_rect_eq}
h_{s,d}(A)=h_{s,d}^*(A) =\frac{\sigma_{s,d}}{\left[\mathcal H_d(A)\right]^{s/d}}.
\end{equation}
Moreover, if $\mathcal H_d(A)>0$, then for any asymptotically $K_s$-extremal sequence $\Omega=\{\omega_N\}_{N\ge 1}$ (for either the constrained or unconstrained polarization problem)  we have the weak-$*$ convergence
\begin{equation}
\frac{1}{N}\sum_{x_i\in\omega_N}\delta_{x_i}\stackrel{*}{\rightharpoonup}\frac{\mathcal H_d|_A}{\mathcal H_d(A)}\quad\mbox{as}\quad N\to\infty,
\end{equation}
where $\mathcal H_d|_A:=\mathcal{H}_d(\cdot\cap A)$ is the restriction to $A$ of the Hausdorff measure $\mathcal H_d$.
\end{theorem}

\subsection*{Outline of the paper}
Section \ref{sec:resspheres} includes some results and conjectures for unconstrained polarization on the circle and on higher dimensional spheres. Sections \ref{sec:proof_bound_K}, \ref{sec:pf_compact_body_old}, \ref{sec:pf_compact_body} and \ref{sec:pf_rectifiable} are dedicated to auxiliary results and proofs of the main results stated in the introduction. Section \ref{sec:mink} contains a bound of independent interest, stated in Proposition \ref{prop:mink_bound}, and later used in the proof of Theorem \ref{thm:rectifiable} in Section \ref{sec:pf_rectifiable}. Finally, in Section \ref{sec:conj_o_p} we discuss some open problems related to polarization.
%
%
%
%
%%%%%%%%%%%%%%%%%%%%%%%%%%%%%%%%%%%%%%%%%%%%%%%%%%%%%%%%%%%%%%%%%%%%%%%%%%%%%%%
\section{Results for the case of spheres $\mathbb S^{p-1}\subset\mathbb R^p$}\label{sec:resspheres}
This section is dedicated to results for an important special case, the unconstrained polarization on the unit sphere $\mathbb S^{p-1}\subset\mathbb R^p$. We start with the following simple result, valid for rather general kernels.
\begin{proposition}\label{prop:small_N}
Let $f:[0,+\infty)\to (-\infty, +\infty]$ be a strictly decreasing function and $K(x,y):=f(|x-y|)$. If $p\ge 2$,   $1\le N\le p$ and $\omega_N^*=\{x_1,\ldots,x_N\}$ satisfies
\begin{equation}\label{opt_gen_K}
\mathcal P_K^*(\mathbb S^{p-1},N)=\min_{y\in A}\sum_{i=1}^N K(x_i,y)\ ,
\end{equation}
then $x_j=0$ for all $1\le j\le N$.
\end{proposition}
\begin{proof}
Let $N'\ge 1$ be the smallest natural number such that there exists an $N'$-point configuration $\omega_{N'}^*=\{x_1,\ldots,x_{N'}\}$ satisfying \eqref{opt_gen_K} such that for some $1\le j\le {N'}$ there holds $x_j\neq 0$. We will prove by contradiction that $N'\ge d+1$, which is equivalent to our statement.

\medskip

Up to reordering the points, there exists $k\in\{0,\ldots,N'\}$ such that
\[
x_j\neq 0\mbox{ for }j=1,\ldots,k\ ,\qquad x_j=0\mbox{ for }j=k+1,\ldots,N'\ .
\]
Let $\omega_{N'}^0:=\{0,\ldots,0\}$ the configuration composed of $N'$ instances of the origin. Then
\begin{equation}\label{Nprime_better_0}
0\le P_K(\mathbb S^{p-1},\omega_{N'})-P_K(\mathbb S^{p-1},\omega_{N'}^0)=P_K(\mathbb S^{p-1},\{x_1,\ldots,x_k\})-P_K(\mathbb S^{p-1},\omega_k^0)\ .
\end{equation}
thus by the minimality of $N'$ we obtain $k=0$, and \emph{all} points in $\omega_{N'}^*$ are away from the origin.

\medskip

As $f$ is decreasing, for each $x\in \mathbb R^p$ the set $S_x$ composed of all points $y$ at which the potential generated by $0$ is higher than that generated by $x$ is a half-space containing the origin. More precisely,
\begin{equation}\label{0_better_xi}
S_x:=\{y\in \mathbb R^p:\ K(0,y)> K(x,y)\}=\{y\in \mathbb R^p:\ \langle x,y\rangle< |y|/2\}\ .
\end{equation}
The intersection of $N'$ half-spaces $S_{x_1},\ldots,S_{x_{N'}}$ is a convex set containing the origin. If $N'<d+1$ this intersection is also unbounded, and thus it intersects $\mathbb S^{p-1}$ at some point $y_0$. Therefore, using \eqref{0_better_xi}, for $N'<d+1$ we find
\begin{equation}\label{0_better_forallxi}
\forall 1\le i\le N'\ ,\quad K(0,y_0)> K(x_j,y_0)\ .
\end{equation}
Summing up the inequalities \eqref{0_better_forallxi}, we find a contradiction to \eqref{Nprime_better_0}, and thus $N'\ge d+1$, as desired.
\end{proof}
By Theorem~\ref{thm:bound_K}, for subharmonic Riesz kernels (i.e. for $s>p-2\ge 0$), points do not accumulate away from $A$. In contrast, the following result demonstrates that the opposite property can hold for superharmonic Riesz potentials.  
This proposition generalizes a result from \cite{2013erdelyisaff}.
\begin{proposition} \label{prop:slep-2}
Fix $p\ge2$ and $s\in (-\infty, p-2]$. If the compact set $A\subset\mathbb R^p$ is such that
\begin{equation}\label{a_has_hole}
\mathbb S^{p-1}\subset A\subset \mathbb B^p\ ,
\end{equation}
where $\mathbb B^p$ denotes the unit ball in $\mathbb R^p$ centered at the origin, then a multiset $\omega_N^*=\{x_1,\ldots,x_N\}\subset\mathbb R^p$ satisfies
\begin{equation}\label{optimizer_K}
P_s(A,\omega_N^*)=\mathcal P_s^*(A,N)\ ,
\end{equation}
if and only if $x_i=0$ for all $i\in\{1,\ldots, N\}$.
\end{proposition}
\begin{proof}
We first note that for all $-\infty<s\le p-2$ the function $K_s(x,y)=f_s(|x-y|)$ as defined in \eqref{def_Ks} is superharmonic in $x$ and in $y$ separately.

\medskip

\textbf{Step 1.} \emph{The case $A=\mathbb S^{p-1}$.} We consider the case $A=\mathbb S^{p-1}$ first. Let $\omega_N^*=\{x_1,\ldots,x_N\}$ and $y^*\in \mathbb S^{p-1}$ are such that there holds
\begin{equation}\label{optimal_K}
\mathcal P_s^*(\mathbb S^{p-1},N)=P_s^*(\mathbb S^{p-1}, \omega_N^*)=\sum_{i=1}^NK_s(x_i,y^*)\ .
\end{equation}
We assume that the points composing $\omega_N^*$ are ordered so that for some $0\le N_0\le N$ there holds $x_1,\ldots,x_{N_0}\neq 0$ and $x_{N_0+1}=\cdots=x_N=0$. For any choice of $y_0\in \mathbb S^{p-1}$ and denoting $\mu_{SO(p)}$ the right-invariant Haar measure on $SO(p)$, there holds
\begin{eqnarray}
P_s^*(\mathbb S^{p-1},\omega_N^*)-Nf_s(1)&=&\sum_{i=1}^{N_0}K_s(x_i,y^*)-N_0f_s(1)=\min_{y\in \mathbb S^{p-1}}\sum_{i=1}^{N_0}f_s(|x_i-y|)-N_0f_s(1)\nonumber\\
&\le& \int_{SO(p)}\sum_{i=1}^{N_0}f_s(|x_i-Ry_0|)d\mu_{SO(p)}(R)-N_0f_s(1)\label{ineq}\\
&=&\int_{SO(p)}\sum_{i=1}^{N_0}f_s(|R^{-1}x_i-y_0|)d\mu_{SO(p)}(R)-N_0f_s(1)\label{isometry}\\
&=&\sum_{i=1}^{N_0}\frac{1}{\mathcal H_{p-1}(\partial B(0,|x_i|))}\int_{\partial B(0,|x_i|)}f_s(|x-y_0|)d\mathcal H_{p-1}(x)-N_0f_s(1)\nonumber\\
&\le&0 \label{superharm_use}\ ,
\end{eqnarray}
where in \eqref{isometry} we used the fact that rotations $R\in SO(p)$ preserve distances and in \eqref{superharm_use} we used the fact that $f_s(|x-y|)$ is superharmonic in $x$. By \eqref{optimal_K} this shows that the choice $x_i=0$ for all $1\le i\le N$, realizes the optimum in $\mathcal P_s^*(\mathbb S^{p-1},\omega_N^*)$. On the other hand, in order for $\omega_N^*$ to be an optimizer, inequalities \eqref{ineq} and \eqref{superharm_use} must become equalities, thus the value $\sum_{i=1}^{N_0}f_s(|x_i-y|)$ is constant in $y\in\mathbb S^{p-1}$ and hence the multiset $\omega_N^*$ is invariant under rotation. This in turn is possible only if all the points $x_i$ are at the origin, as desired.

\medskip

\textbf{Step 2.} \emph{The case $A=\mathbb B^p$.} In this case by Proposition \ref{prop:covex_equal} we have that the problem reduces to the classical constrained polarization, and the statement was proved in \cite{2013erdelyisaff}.

\medskip

\textbf{Step 3.} \emph{General case $\mathbb S^{p-1}\subset A\subset\mathbb B^p$.} Due to \eqref{inclusion}, \eqref{superharm_use} and Step 2, we have $\mathcal P^*_s(A,N) = Nf_s(1)$ as well. For any multiset $\{x_1,\ldots,x_N\}$ there holds
\begin{equation}\label{bound_gen_A}
\min_{y\in \mathbb B^p}\sum_{i=1}^NK_s(x_i,y)\le\min_{y\in A}\sum_{i=1}^NK_s(x_i,y)\le
\min_{y\in\mathbb S^{p-1}}\sum_{i=1}^NK_s(x_i,y)\ ,
\end{equation}
which implies that the multiset with $x_i=0$ for all $1\le i\le N$ is an optimizer for $\mathcal P_s^*(A,N)$. If by contradiction a distinct optimizer would exist, then it would be an optimizer also for $\mathcal P_s^*(\mathbb S^{p-1},N)$, which is excluded by Step 1. This concludes the proof of Proposition \ref{prop:slep-2}.\end{proof}
The following proposition describes the case where $s>p-2$, $s\neq p-1$, and establishes that $\mathcal P_s^*$-optimal configurations $\omega_N^*$ for $\mathbb S^{p-1}$ lie at a positive distance from $\mathbb S^{p-1}$. We conjecture that the result continues to hold for $s=p-1$, but the proof is left to future work.
\begin{proposition}\label{prop:points_notonsphere}
Let $p\ge 2$ and $s>p-2$, $s\neq p-1$. Then there exists a constant $C>0$ depending only on $s$ and $p$, such that for any $N$-point multiset $\omega_N^*$ satisfying $P_s(\mathbb S^{p-1},\omega_N^*)=\mathcal P_s^*(\mathbb S^{p-1},N)$ there holds
\begin{equation}\label{points_notonsphere}
\op{dist}(\omega_N^*,\mathbb S^{p-1})\ge CN^{-2/(p-1)}\ .
\end{equation}
\end{proposition}
\begin{proof}
For the proof, we will use Proposition \ref{prop:max_away_points} below, with $s,p$ as in the statement of the proposition and $A=\mathbb S^{p-1}$. We need to verify that the hypotheses of Theorem \ref{thm:compact_body_old} (which are inherited by Proposition \ref{prop:max_away_points}) hold. Note that by Proposition \ref{prop:kth_moment}, the equilibrium measure $\mu_{s,\mathbb S^{p-1}}$ is the uniform measure on $\mathbb S^{p-1}$ and that $\mathcal H_{p-1}(\mathbb S^{p-1})>0$. Therefore the hypotheses of point (i) from Theorem \ref{thm:compact_body_old} hold if $s>p-1$ and the conditions of point (ii) from the same theorem hold if $p-2<s<p-1$.

\medskip

By \eqref{max_away_points} of Proposition \ref{prop:max_away_points} there exists a constant $C_1>0$ independent of $N$ such that the minimum value
\[
\mathcal P_s^*(\mathbb S^{p-1},N)=\min_{y\in \mathbb S^{p-1}}\sum_{x\in\omega_N^*}|x-y|^{-s}
\]
is not achieved at $y\in B(x_0,C_1N^{-1/(p-1)})$. Since the function $|x-y|^{-s}$ is continuous away from the diagonal, there exists $\epsilon>0$ such that
\begin{equation}\label{min_far}
\min_{y\in B(x_0,C_1N^{-1/(p-1)})}\sum_{x\in\omega_N^*}|x-y|^{-s}\ge \mathcal P_s^*(\mathbb S^{p-1},N)+\epsilon\ .
\end{equation}
We claim that 
\begin{equation}\label{claimnew}
x_0\in\op{conv}(\mathbb S^{p-1}\setminus B(x_0,C_1N^{-1/(p-1)})).
\end{equation}
Assume by contradiction that this is not true and that
\begin{equation}\label{x0_out_convhull}
\op{dist}\left(x_0,\op{conv}(\mathbb S^{p-1}\setminus B(x_0,C_1N^{-1/(p-1)}))\right)>0.
\end{equation}
Let $x_0'\neq x_0$ be the projection of $x_0$ on $\op{conv}(\mathbb S^{p-1}\setminus B(x_0,C_1N^{-1/(p-1)}))$. Since $|x-y|^{-s}$ is decreasing in $|x-y|$, we find that for any $x_0''$ in the segment $(x_0,x_0']$ there holds
\begin{equation}\label{better_min}
\forall y\in\mathbb S^{p-1}\setminus B(x_0,C_1N^{-1/(p-1)})\ ,\quad\sum_{x\in\omega_N^*}|x-y|^{-s}<\sum_{x\in\omega_N^*\setminus\{x_0\}}|x-y|^{-s}+|x_0''-y|^{-s}\ .
\end{equation}
On the other hand, by \eqref{min_far}, \eqref{x0_out_convhull} and by the continuity of $|x-y|^{-s}$ for $x\neq y$, there exists an open neighborhood $U_\epsilon$ of $x_0$ in the segment $[x_0,x_0']$ such that for $x_0''\in U_\epsilon$ there holds
\begin{equation}\label{min_far_2}
\min_{y\in B(x_0,C_1N^{-1/(p-1)})}\sum_{x\in\omega_N^*\setminus\{x_0\}}|x-y|^{-s}+|x_0''-y|^{-s}\ge \mathcal P_s^*(\mathbb S^{p-1},N)+\epsilon/2\ .
\end{equation}
By \eqref{better_min} and \eqref{min_far_2} there holds $P_s(\mathbb S^{p-1},(\omega_N^*\setminus\{x_0\})\cup\{x_0''\})\ge \mathcal P_s^*(\mathbb S^{p-1},N) +\epsilon/2$, giving the desired contradiction to \eqref{x0_out_convhull}. Thus \eqref{claimnew} holds. 

\medskip

It follows that for any point $y\in \partial B(x_0,C_1N^{-1/(p-1)})\cap\mathbb S^{p-1}$ there holds $\langle x_0-y,x_0\rangle \le 0$ and thus $1=|y|\le |x_0|^2+|x_0-y|^2=|x_0|^2+C_1^2N^{-2/(p-1)}$, from which it follows that $|x_0|\le\sqrt{1-C_1^2N^{-2/(p-1)}}\le 1-\frac{C_1^2}{2}N^{-2/(p-1)}$ and the thesis follows with $C=C_1^2/2$.
\end{proof}
The next result states the equivalence of constrained and unconstrained covering problems, which is a well-known property of spherical coverings:
\begin{proposition}\label{prop:constr-unconstr_sphere}
Let $p\ge 2$ and $N\in \mathbb N$.
\begin{itemize}
\item If $N\le p$, then a configuration realizing the infimum $\eta_N^*(\mathbb S^{p-1})$ in \eqref{covercover*} is given by taking all the $N$ points at the origin of $\mathbb R^p$.
\item If $N\ge p+1$, then for every configuration $\omega_N=\{x_1,\ldots,x_N\}\subset \mathbb S^{p-1}$ that realizes the infimum $\eta_N(\mathbb S^{p-1})$ in \eqref{covercover*}, then $\omega_N^*:=\{r_N x_1, \ldots,r_Nx_N\}$ realizes the infimum $\eta_N^*(S^{p-1})$ in \eqref{covercover*} for $r_N=\sqrt{1-(\eta_N^*)^2(\mathbb S^{p-1})}$. Furthermore,
\begin{equation}\label{relation_etas}
(\eta_N^*)^2(\mathbb S^{p-1}) = \eta_N^2(\mathbb S^{p-1})-\frac14 \eta_N^4(\mathbb S^{p-1}).
\end{equation}
\end{itemize}
\end{proposition}
\begin{proof}
It is not difficult to verify that $\mathbb S^{p-1}$ cannot be covered by $p$ balls of radius less than $1$. This  can be proved by induction on the dimension using the fact that if $r\in(0,1), y\in\mathbb R^p$, then $\mathbb S^{p-1}\setminus B(y,r)$ contains a congruent copy of $\mathbb S^{p-2}$ as long as $p\ge 2$, and the observation that $\mathbb S^0=\{\pm1\}$ requires at least two balls of radius $r$ to be covered. It follows that $\eta_N^*(\mathbb S^{p-1})=1$: a configuration realizing this infimum is given by the case when all the $N$ points are at the origin of $\mathbb R^p$. This proves the first item of the proposition.

\medskip

\noindent If $N\ge p+1$ then $\eta_N(\mathbb S^{p-1})<\sqrt 2$ and $\eta_N^*(\mathbb S^{p-1})<1$, a bound shown by rough estimates for competitor configurations for $\eta_N(\mathbb S^{p-1})$ where $p+1$ of the points form a regular simplex, and those for $\eta_N^*(\mathbb S^{p-1})$ sit at centroids of the faces of such simplex.

\medskip

\noindent In order to compare the two covering problems with minima $\eta_N(\mathbb S^{p-1})$ and $\eta_N^*(\mathbb S^{p-1})$ for $N\ge p+1$, we introduce some notations, as follows. For $y\in\mathbb B^p\setminus\{0\}$ the set $\partial B(y,\rho)\cap \mathbb S^{p-1}$ is nonempty if and only if $\rho\in[1-|y|,1+|y|]$. For these choices of $y,\rho$, we note the following:
\begin{itemize}\item $\partial B(y,\rho))\cap \mathbb S^{p-1}$ is a congruent copy of a $(p-2)$-dimensional sphere of radius given by $f(|y|,\rho)=\frac{\sqrt{4|y|^2-(|y|^2+1-\rho^2)^2}}{2|y|}$. Moreover, for fixed $\rho\in(0,1)$ the function $(1-\rho,1)\ni|y|\mapsto f(|y|,\rho)$ achieves its unique maximum, equal to $\rho$, at $|y|=\sqrt{1-\rho^2}$. 
\item There holds $B(y,\rho)\cap \mathbb S^{p-1}=B (y/|y|,\bar\rho)\cap \mathbb S^{p-1}$, with $\bar\rho=\bar\rho\left(|y|,\rho\right):=\sqrt{\frac{\rho^2-(1-|y|)^2}{|y|}}$. Moreover, we note that in the above range of $\rho$, for $|y|=\sqrt{1-\rho^2}$ we get $\bar\rho=\sqrt2\sqrt{1-\sqrt{1-\rho^2}}$, which is increasing in $\rho$.
\end{itemize}
Now for $N\ge p+1$ let $\omega_N=\{x_1,\ldots,x_N\}\subset \mathbb S^{p-1}$ be at optimizer configuration for $\eta_N(\mathbb S^{p-1})$, in particular
\begin{equation}\label{covernostar}
\bigcup_{j=1}^N B  (x_j,\eta_N(\mathbb S^{p-1}))\supset \mathbb S^{p-1},
\end{equation}
and define simlarly to the claim of the second bullet in the proposition
\begin{equation}\label{competcover}
\rho^\prime:=\eta_N^2(\mathbb S^{p-1})-\frac14\eta_N^4(\mathbb S^{p-1}) \quad\mbox{ and }\quad r^\prime:=\sqrt{1-(\rho^\prime)^2}.
\end{equation}
With this notation the balls $B(r^\prime x_j, \rho^\prime), 1\le j\le N$ cover $\mathbb S^{p-1}$. Indeed, $B(x_j,\eta_N(\mathbb S^{p-1}))\cap \mathbb S^{p-1}=B(r^\prime x_j, \rho^\prime)\cap \mathbb S^{p-1}$ and the claim follows from \eqref{covernostar}. This implies that $\rho^\prime\ge \eta_N^*(\mathbb S^{p-1})$ as well. Also note that 
\begin{equation}\label{optfbarro}f(r^\prime, \rho^\prime)=\rho^\prime\quad \mbox{and}\quad\bar\rho(r^\prime, \rho^\prime)=\bar\rho(\sqrt{1-(\rho^\prime)^2}, \rho^\prime)=\eta_N(\mathbb S^{p-1}).
\end{equation}
We next prove that $\rho^\prime=\eta_N^*(\mathbb S^{p-1})$, which directly implies \eqref{relation_etas}. Assume that this were not true, and that we had $\rho^\prime>\eta_N^*(\mathbb S^{p-1})$. Then there would exist $\widetilde\rho<\rho^\prime$ and a configuration $\{y_1,\ldots,y_N\}\subset \mathbb B^p\setminus\{0\}$ such that
\begin{equation}\label{coverabsurd}
\bigcup_{j=1}^N B (y_j,\widetilde\rho)\supset \mathbb S^{p-1}.
\end{equation}
Up to moving some of the points radially, we may suppose that $|y_j|$ are all equal to $\sqrt{1-\widetilde\rho^2}$ at which the function$f(\cdot,\widetilde\rho)$ achieves its maximum. Since $\rho\mapsto \bar\rho(\sqrt{1-\rho^2},\rho)$ is increasing, we find that
\begin{equation}\label{coverabsurd1}
\bar\rho(\sqrt{1-\widetilde \rho^2},\widetilde \rho)<\bar\rho(\sqrt{1-(\rho^\prime)^2},\rho^\prime)=\eta_N(\mathbb S^{p-1}),
\end{equation}
where we used \eqref{optfbarro}. But due to the geometric interpretation of $\bar\rho$ and to \eqref{coverabsurd}, we also have that
\[
\bigcup_{j=1}^NB\left(\frac{y_j}{|y_j|},\bar\rho\left(\sqrt{1-\widetilde\rho^2},\widetilde\rho\right)\right)\supset\mathbb S^{p-1},
\]
 which implies $\bar\rho(\sqrt{1-\widetilde\rho^2},\widetilde\rho)\ge\eta_N(\mathbb S^{p-1})$, which as desired contradicts \eqref{coverabsurd1}. Therefore we have $\rho^\prime=\eta_N^*(\mathbb S^{p-1})$, and the second bullet of the proposition follows.
\end{proof}
\subsection{Results for $\mathbb S^1\subset\mathbb R^2$}
For $\mathbb S^1$ it is easily seen that the minimal $N$-point covering optimal configurations constrained to $\mathbb S^1$ are given by the vertices of the inscribed regular $N$-gon. In the case of the minimal $N$-point \emph{unconstrained} covering we prove the following more precise version of Proposition \ref{prop:constr-unconstr_sphere}:
\begin{proposition}\label{prop:opt_cover_s1}
The configurations $\omega_N^*$ realizing the infimum in the definition of $\eta_N^*(\mathbb S^1)$ are, up to rotation, the following:
\begin{itemize}
\item For $N=1$, $\omega_1^*=\{0\}$
\item For $N=2$, $\omega_2^*=\{0,0\}$
\item For $N\ge 3$, $\omega_N^*$ consists of the midpoints of the sides of the regular $N$-gon inscribed in $\mathbb S^1$.
\end{itemize}
\end{proposition}
\begin{proof}
The cases $N<2$ of the statement follows directly from Proposition \ref{prop:small_N} by using Proposition \ref{prop:bosuwan*}. Therefore we assume $N\ge 3$ for the rest of the proof.

\medskip

\noindent The midpoints of the sides of an inscribed regular $N$-gon are given by
\begin{equation}
\label{points_pj}
p_j:=(\cos(\pi/N)\cos(\theta +2\pi j/N), \cos(\pi/N)\sin(\theta+2\pi j/N))\in\mathbb R^2\ ,\quad\mbox{for}\ j\in\{0,\ldots,N-1\}\ ,
\end{equation}
where $\theta\in[0,2\pi/N)$ gives the orientation of our $N$-gon. A closed disk of radius $\sin(\pi/N)$ centered at $p_j$ covers the interval $I_j:=\{(\cos\phi,\sin\phi):\ \phi\in[\theta+(2j-1)\pi/N,\theta +(2j+1)\pi/N]\}$ inside $\mathbb S^1$, thus the union of all such disks covers the unit circle $\mathbb S^1$. Note that $\mathbb S^1\cap B(x,r)$ is always an arc of the form
\begin{equation}\label{interval}
I(\theta_0,\rho):=\{(\cos\phi, \sin\phi):\ \phi\in[\theta_0-\rho,\theta_0+\rho]\}\ .
\end{equation}
By direct computation of the local minimum, we find for $0\le \rho<\pi/2$
\begin{equation}\label{best_r}
\min\{r>0:\ \exists x\in\mathbb R^2,\ I(\theta_0,\rho)=\mathbb S^1\cap B(x,r)\}=\sin(\rho)\ ,
\end{equation}
and the unique $x$ realizing the above minimum is the point $(\cos(\rho)\cos(\theta_0), \cos(\rho)\sin(\theta_0))$. Further, we have that for fixed $N$ if $N$ arcs $I(\theta_j, \rho_j), j=1,\ldots,N$ cover $\mathbb S^1$ then
\[
\sum_{j=1}^N2\rho_j\ge 2\pi \quad\mbox{and}\quad\max_{1\le j\le N}\rho_j\ge \pi/N\ ,
\]
and thus
\begin{equation}
\label{bestcover_arcs}
\min\left\{\rho>0:\ \exists \theta_1,\rho_1,\ldots,\theta_N,\rho_N,\ \bigcup_{j=1}^N I(\theta_j,\rho_j)=\mathbb S^1,\ \max_{1\le j\le N}\rho_j\le \rho\right\}=\pi/N\ ,
\end{equation}
and the minimum is realized by a collection of equal intervals. Noting that for $\rho\le \pi/N, N\ge 3$ the function $\rho\mapsto\sin(\rho)$ is increasing, we find that as a consequence of \eqref{bestcover_arcs} and \eqref{best_r}, there holds
\[
\min\left\{r>0:\ \exists \omega_N=\{x_1,\ldots,x_N\},\ \bigcup_{j=1}^N B(x_i,r_j)\supset\mathbb S^1,\ \max_{1\le j\le N}r_j\le r\right\}=\sin(\pi/N)\ ,
\]
and the minimum is realized by the points $p_j$ from \eqref{points_pj}. This completes the proof of Proposition \ref{prop:opt_cover_s1}.
\end{proof}
For $N$-point constrained Riesz $s$-polarization on $\mathbb S^1$, it is proved in \cite{hks2013} that optimal configurations are again equally spaced points on $\mathbb S^1$, for each $0<s<\infty$.  The proof of Proposition \ref{prop:s1} below follows the strategy of \cite{hks2013}.  For related results, see also \cite{ambrus2009, 2013ambrusballerdelyi, 2013erdelyisaff} and \cite{Farkas2018}. For the unconstrained $s$-polarization we have not yet determined the precise optimizers $\omega_{N,s}^*$.  However, numerical evidence (see Figure \ref{fig:s1}) strongly suggests that for $N\ge 3$ the configurations form a regular $N$-gon  inscribed in a circle of radius $\bar r_{N,s}<1$,  where  
\begin{equation}\label{r_ns}
\bar r_{N,s}:=\arg\max_{r\in[0,1]}\sum_{j=1}^N\left(r^2+1-2\cos\frac{(2j+1)\pi}{2N}\right)^{-s/2}.
\end{equation}
\begin{figure}[!h]
\centering
\includegraphics[width=7.5cm]{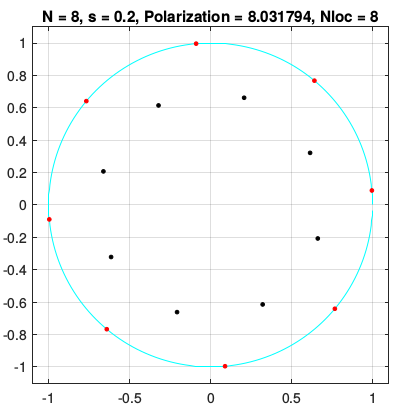}\includegraphics[width=7.5cm]{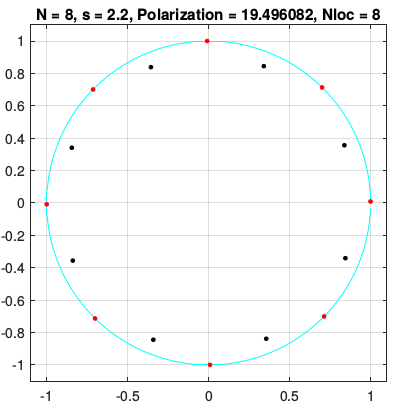}
\caption{Numerically computed optimizers $\omega_N^*$ (black points) for unconstrained Riesz $s$-polarization for $\mathbb S^1$ and points on $\mathbb S^1$ where the minimum of the  $s$-potential is attained (red points) for $N=8$ and  the values $s=0.2$ and $s=2.2$.}
\label{fig:s1}
\end{figure}

We remark that for fixed $N\ge 3$, Propositions~\ref{prop:bosuwan*} and \ref{prop:opt_cover_s1} imply that as $s\to \infty$   maximal $N$-point unconstrained polarization configurations $\omega_{N,s}^*$ (with one of the points fixed at 1) approach the midpoints of the sides of a regular $N$-gon in $\mathbb S^1$.

Under the extra assumption that the optimal configuration $\omega_{N,s}^*$ lies on a concentric circle with radius $r$ satisfying \eqref{eqn:concentric2}, we are able to establish the above conjecture, based on the following result, which is of independent interest and improves the main result of \cite{hks2013} by removing the convexity condition for $f$ on the interval $(0,\pi/N]$.
\begin{figure}[!h]
\centering
\includegraphics[width=7.cm]{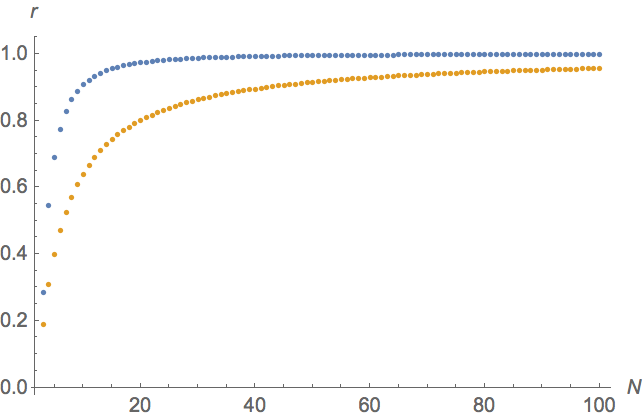}\hspace{.2in} \includegraphics[width=7.cm]{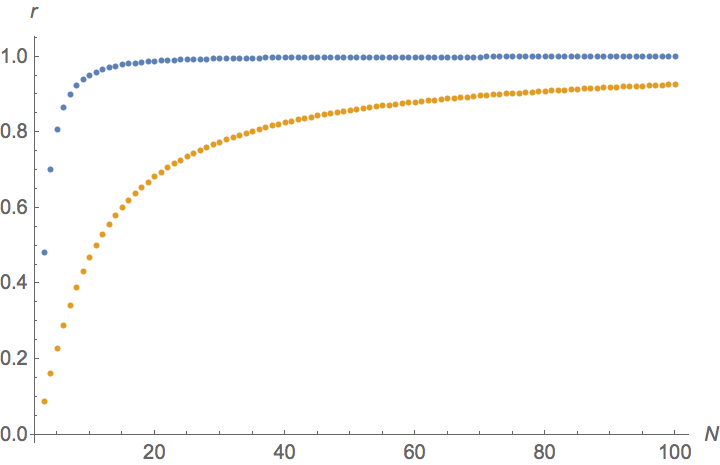}\\
\caption{Graphs of $\bar r_{N,s}$ from \eqref{r_ns} (blue dots) and $R_{N,s}^{-1}$  from \eqref{RNsdef} (orange dots) for $s=1$ (left) and $s=5$ (right) as $N$ ranges from $3$ to $100$. In both cases, $R_{N,s}^{-1}<\bar r_{N,s}$, therefore the range \eqref{eqn:concentric2} of $r$ in which Corollary \ref{coro:concentric} applies includes the expected radius $\bar r_{N,s}$ from \eqref{r_ns}. }
\label{fig:s1bis}
\end{figure}

\medskip

If we take $t\in[-\pi,\pi]$ to parametrize the counterclockwise signed angle between two points $x,y\in\mathbb S^1$, then the \emph{geodesic distance} between $x$ and $y$ is given by $\mathrm{dist}_{\mathbb S^1}(x,y):=\min\{|t|,|2\pi-t|\}$.

\begin{proposition}\label{prop:s1}
For $x,y\in\mathbb S^1$ let $\mathrm{dist}_{\mathbb S^1}(x,y)\in[0,\pi]$ be the geodesic distance (or smallest angle) between $x,y$, and set $K(x,y):=f(\mathrm{dist}_{\mathbb S^1}(x,y))$, for $f:[0,\pi]\to (-\infty, +\infty]$, and assume that the following hypotheses hold:
\begin{enumerate}[(i)]
\item[\rm (i)] the function $f$ is strictly decreasing on $(0,\pi]$ and strictly convex on $(\tfrac{\pi}{N},\pi]$;
\item[\rm (ii)] for the configuration $\omega_{N,\mathrm{eq}}\subset \mathbb S^1$ given by $x_k=e^{i\frac{2\pi k }{N}}$ for $k=1,\ldots,N$, the minimum value $P_K(\omega_{N,\mathrm{eq}})$ is achieved at the midpoints of the arcs between successive points $x_k,x_{k+1}$.
\end{enumerate}
Then any configuration $\omega_N^*\subset \mathbb S^1$ that satisfies $P_K(\mathbb S^1,\omega_N^*)=\mathcal P_K(\mathbb S^1,N)$ equals $\omega_{N,\mathrm{eq}}$, up to rotation.
\end{proposition}
\begin{proof}
We recall that the proof in \cite[Thm.\,1]{hks2013} consisted of starting from a general $N$-point configuration $x_1,\ldots,x_N\in\mathbb S^1$, initially ordered in counterclockwise manner, and applying a sequence of $N$ elementary moves to the points (see \cite[Lem.\,5]{hks2013}). The elementary moves are denoted $\tau_{\Delta_k^*}$, with $1\le k\le N$, $\Delta_k^*\in \R$. The move $\tau_{\Delta_k^*}$ leaves the positions of $x_1,\ldots,x_{k-1},x_{k+2},\ldots, x_N$ unchanged, and replaces the points $x_k$ and $x_{k+1}$ (with indices taken modulo $N$) by new points $x_k^\prime:= x_ke^{-i\Delta_k^*}$ and $x_{k+1}^\prime:=x_{k+1}e^{i \Delta_k^*}$, respectively.  A simple linear algebra argument shows (see \cite[Lem.\,5]{hks2013})   that there is a sequence of elementary moves such that:
\begin{enumerate}
\item[(a)] $\Delta_k^*\ge 0$ for $k=1,\ldots,N$,
\item[(b)] There exists $1\le j\le N$ such that $\Delta_j^*=0$,
\item[(c)]\label{item3} The composition $\tau_\mathbf{\Delta^*}:=\tau_{\Delta_N^*}\circ\cdots\circ\tau_{\Delta_1^*}$ sends $\omega_N$ to a rotation of the configuration $\omega_{N,\mathrm{eq}}$.
\end{enumerate}
We first assume   that none of the elementary moves change  the counterclockwise ordering of the points. 
Let $x^*$ denote the midpoint of the arc between $\tau_\mathbf{\Delta^*}(x_j),\tau_\mathbf{\Delta^*}(x_{j+1})$ for $j$ as in (b). By the above properties, we can prove by backwards induction on $k$ that
\begin{equation}\label{distance_xstar}
\min_{y\in\tau_{\Delta_k^*}\circ\cdots\circ\tau_{\Delta_1^*}(\omega_N)} \mathrm{dist}_{\mathbb S^1}(y,x^*)\ge \frac{\pi}{N}.
\end{equation}
Indeed, this is true for $k=N$ due to item (c) above; furthermore, if it is true for $k=n$ for some $2\le n\le N$ then due to items (a), (b) then it also holds for $k=n-1$.

\medskip

Next, as in \cite[Lem.\,4]{hks2013}, we prove that the potential generated by the points \emph{increases} on the arc $\gamma_{k,N}$ going from  $x_{k+1}'e^{i\pi/N}$ to $x_k'e^{-i\pi/N}$ in the counterclockwise direction, during the move $\tau_{\Delta_k^*}$. Towards this end, let $x\in \gamma_{k,N}$ and consider 
$\ell_k=\mathrm{dist}_{\mathbb S^1}(x_k,x)$, $\ell_k'=\mathrm{dist}_{\mathbb S^1}(x_k',x)$, $\ell_{k+1}=\mathrm{dist}_{\mathbb S^1}(x_{k+1},x)$,  and  $\ell_{k+1}'=\mathrm{dist}_{\mathbb S^1}(x_{k+1}',x)$.
Without loss of generality we may assume  $\ell_k\le \ell_{k+1}$, in which case
 we note that
 \begin{equation}\label{orderpts}
\ell_k'= \ell_k-\Delta_k^* \text{ and } \ell_{k+1}'\le \ell_{k+1}+\Delta_k^*.\end{equation}
Extending $f$ as a decreasing convex function on  $[\pi/N,\infty)$ and using $\ell_k\le \ell_{k+1}$, it follows that
 \begin{equation}\label{changepot_k}
\left[f(\ell_{k+1}')- f(\ell_{k+1}) \right]+ \left[f(\ell_{k}')-f(\ell_k)\right]\ge \left[f(\ell_{k+1}+\Delta_k^*)- f(\ell_{k+1}) \right]+ \left[f(\ell_{k}-\Delta_k^*)-f(\ell_k)\right]\ge 0.
\end{equation} 

\medskip

Due to \eqref{distance_xstar}, $x^*$ belongs to all the intervals $\gamma_{k,N}$ as above, for $k=1,\ldots,N, k\neq j$. As a consequence of the inequality \eqref{changepot_k}, during the sequence of moves as in the above steps (a),(b),(c) the value of the polarization potential at $x^*$ increases. Thus we have
\begin{equation}\label{int_f}
P_K(\omega_N)\le \sum_{x\in \omega_N}f(\mathrm{dist}_{\mathbb S^1}(x,x^*))\le\sum_{x\in \omega_{N,\mathrm{eq}}}f(\mathrm{dist}_{\mathbb S^1}(x,x^*))=P_K(\omega_{N,\mathrm{eq}}),
\end{equation}
where for the last equality we used hypothesis (ii). This shows that $\omega_{N,\mathrm{eq}}$ is an optimal configuration, as desired.

If not all of the elementary moves preserve the counterclockwise ordering, then we modify the above argument by considering compositions of moves $\tau_{t_k\mathbf \Delta^*}:=\tau_{t_k\Delta_N^*}\circ\cdots\circ\tau_{t_k\Delta_1^*}$, $k=1,\ldots, n$, 
for $t_k>0$ sufficiently small so that the ordering   is preserved (see \cite[Lem.\,6]{hks2013}) at each step and such that $\sum_{k=1}^n t_k=1$.  

\medskip

If $f$ is strictly convex on $[\pi/N,\pi]$, then the fact that in the middle inequality in \eqref{int_f} the equality holds, implies that during all the moves all the terms as in \eqref{int_f} are zero, which can only be true if $\Delta_k^*=0$ for all $k$, showing that $\omega_N=\omega_{N,\mathrm{eq}}$ up to rotation in this case.
\end{proof}
The following lemma gives two important cases in which the hypothesis (ii) from Proposition \ref{prop:s1} holds, the second of which is due to Nikolov and Rafailov \cite[Thm.\,1.2 (1)]{nr2011sum}.
\begin{lemma}\label{lem:nikorafa}
Let $K:\mathbb S^1\times \mathbb S^1\to (-\infty,+\infty]$ be given by $K(x,y)=f(\mathrm{dist}_{\mathbb S^1}(x,y))$ for a fixed function $f:[0,\pi]\to(-\infty,+\infty]$. Assume that we are in one of the following cases:
\begin{enumerate}[(i)]
\item[\rm (i)] the function $f$ satisfies the hypothesis {\rm (i)} of Proposition~\ref{prop:s1} and furthermore
\[\min_{\theta\in[0,\pi/N]}\left(f(\theta)+f(\tfrac{2\pi}{N}-\theta)\right)=2f(\tfrac{\pi}{N});\]
\item[\rm (ii)] there exist $R,s>0$ such that $f(t)=(R^2+1-2R\cos(t))^{-s/2}$.
\end{enumerate}
Then hypothesis {\rm (ii)} of Proposition~\ref{prop:s1} holds, namely
\begin{equation}\label{hyp_2}
\min_{y\in\mathbb S^1}\sum_{k=1}^NK\left(e^{i\frac{2\pi k}{N}},y\right)=\sum_{k=1}^NK\left(e^{i\frac{2\pi k}{N}},e^{i\frac{\pi}{N}}\right).
\end{equation}
\end{lemma}
\begin{proof}
The proof of the claim in the case   (ii)  is precisely \cite[Thm.\,1.2 (1)]{nr2011sum}, therefore we need to prove the claim in the case (i) only.

\medskip

Let $x_k=e^{i\frac{2\pi k}{N}}$ for $k=1,\ldots, N$. By symmetry, we consider the values of $\sum_{k=1}^Nf(\mathrm{dist}_{\mathbb S^1}(x_k,y))$ only for $y=e^{i\theta}$ with $\theta\in[0,\pi/N]$. We split $\omega_{N,\mathrm{eq}}$ into pairs of points $x_k, x_{N-k}$, for $k=1,\ldots,\lfloor N/2\rfloor$, to which we add, if $N=2n-1$ is odd, the potential $f(\mathrm{dist}_{\mathbb S^1}(x_n,y))=\min\{|\pi-\theta|, |\pi+\theta|\}$. The latter potential has a minimum at $\theta=0$ due to the decreasing nature of $f$. For the remaining pairs of points, we claim that the potential of each pair has a minimum at $\theta=0$ as well, and by superposition this will prove the claim.

\medskip

The points $x_k, x_{N-k}$ generate at $e^{i\theta}$ the joint potential equal to
\begin{equation}\label{pairpot_y}
f\left(\mathrm{dist}_{\mathbb S^1}(x_k, e^{i\theta})\right) +
f\left(\mathrm{dist}_{\mathbb S^1}(x_{N-k}, e^{i\theta})\right) = f\left(\frac{2k-1}{N}\pi-\theta\right)+f\left(\frac{2k+1}{N}\pi+\theta\right).
\end{equation}
For $k=1$ this is minimized at $\theta=0$ by the second hypothesis on $f$ from the statement of the proposition, whereas for $k>1$ we may use the convexity of $f$ to obtain that $(f(a)+f(b))/2\ge f((a+b)/2)$ for $a=(2k-1)\pi/N-\theta$ and $b=(2k+1)\pi/N+\theta$, in order to show again, using the symmetry of the configuration that the minimum is achieved at $\theta=0$, as desired.
\end{proof}
\begin{corollary}\label{coro:concentric}
Let  $s,r>0$, $N\in \mathbb N$, $A=\mathbb S^1$,  $B=r\mathbb S^1$, and define
\begin{equation} \label{xrsdef} x_{r,s}:=\frac{-(1+r^2)+\sqrt{(1+r^2)^2+ 4r^2s(2+s)}}{2 r s}
\end{equation}
and
\begin{equation}\label{RNsdef}
R_{N,s}:=\frac{1}{2} \sec ({\pi}/{N}) \left( s \sin
   ^2({\pi}/{N})+2+\sqrt{\sin
   ^2({\pi}/{N}) \left((s+2)^2-s^2 \cos
   ^2({\pi}/{N})\right)}\right).
\end{equation}
If $N, r$, and $s$ satisfy
\begin{equation}\label{eqn:concentric1} \cos({\pi}/{N})\le x_{r,s},\end{equation} or
\begin{equation}\label{eqn:concentric2} R_{N,s}^{-1}\le r\le R_{N,s}, \end{equation}
then any $\omega_N^*\subset B$ such that $P_s(A,\omega_N^*)=\mathcal P_s(A,B,N)$ equals, up to rotation, the regular $N$-gon inscribed in the circle $B$.

\end{corollary}
\begin{proof}
The function $f_s(\theta):=(1 + r^2 - 2 r \cos \theta)^{-s/2}$ is decreasing for $\theta\in [0,\pi]$.  Differentiating $f_s$ twice gives
$$f''_s(\theta)=-\frac{r s}{2}\  \frac{ 2 (1 + r^2) \cos \theta + r (-4 + 2s ( \cos^2 \theta-1))}{(1 + r^2 - 2 r \cos \theta)^{2 +s/
  2}}.
  $$
Letting  $ g(r,s,x):=2 (1 + r^2) x + r (-4 + 2s ( x^2-1))$ then   $f''_s(\theta)$ is positive on any interval where $g(r,s,\cos \theta)$ is negative.  Noting that $ g(r,s,x)$ is an increasing function of $x$ (with $r$ and $s$ fixed) for $x>0$ and  that $g(r,s,x_{r,s})=0$ shows that $g(r,s,x)\le 0$  if and only if $x\in [-1,x_{r,s}]$.  Hence, if  \eqref{eqn:concentric1} holds, then $f_s$ is convex on $[\pi/N,\pi]\subset [\arccos x_{r,s},\pi]$ and so  we may use Proposition~\ref{prop:s1} to prove  that  any $\omega_N^*\subset B$ such that $P_s(A,\omega_N^*)=\mathcal P_s(A,B,N)$ must consist of $N$ equally spaced points in the circle $B$.

To complete the proof we show that \eqref{eqn:concentric2} implies that \eqref{eqn:concentric1} holds.  Towards this end, let
$$
r_{x,s}^{\pm}:=\frac{s+2-s
   x^2\pm\sqrt{\left(1-x^2\right)
   \left((s+2)^2-s^2 x^2\right)}}{2 x},
$$
denote the solutions to $g(r,s,x)=0$ for fixed $x>0$ and $s$ and note that  $g(r,s,x)<0$ for $r_{x,s}^-<r<r_{x,s}^+$ and   $r_{x,s}^-r_{x,s}^+=1$.  Observe that $R_{N,s}=r_{\cos(\pi/N),s}^+$ and   $R_{N,s}^{-1}=r_{\cos(\pi/N),s}^-$.  Therefore, if $R_{N,s}^{-1}\le r\le  R_{N,s}$, we have
$g(r,s,\cos(\pi/N))\le 0$ and so it follows that $\cos(\pi/N)\le x_{r,s}$; i.e., that \eqref{eqn:concentric1} holds.
\end{proof}
Note that, due to the fact that $|x-y|^{-s}$ is symmetric, up to inverting the roles of $A,B$ we can restrict to the case $r\in[\bar r_{N,s},1]$, where $\bar r_{N,s}$ is as in \eqref{r_ns}. We found good numerical evidence (as shown in special cases in Figure \ref{fig:s1bis}) that for $s>0$ there exists $N_0(s)\in\mathbb N$ such that for all $N\ge N_0(s)$ there holds $R_{N,s}^{-1}<\bar r_{N,s}<1<R_{N,s}$, therefore the range \eqref{eqn:concentric2} of $r$ in which Corollary \ref{coro:concentric} applies includes the expected radius $\bar r_{N,s}$ from \eqref{r_ns}. We found numerically that $N_0(s)=2$ for $s\ge 0.7$.
%
%
%
%
%%%%%%%%%%%%%%%%%%%%%%%%%%%%%%%%%%%%%%%%%%%%%%%%%%%%%%%%%%%%%%%%%%%%%%%%%%%%%%%%%%%
\section{Proof of Theorem \ref{thm:bound_K}}\label{sec:proof_bound_K}
%%%%%%%%%%%%%%%%%%%%%%%%%%%%%%%%%%%%%%%%%%%%%%%%%%%%%%%%%%%%%%%%%%%%%%%%%%%%%%%%%%%%
%
%
We first prove an auxiliary result, Lemma \ref{lem:no_concentr}, and then proceed to the proof of Theorem \ref{thm:bound_K}. This result in turn uses the result stated in Remark~\ref{tk_spheres}, a special case of Proposition~\ref{prop:kth_moment}. A result similar to Lemma~\ref{lem:no_concentr}, with a non-sharp version of bound \eqref{claim_main} below, and with an additional convexity requirement on the set $A$, appears in \cite[Thm.\,2.3]{2017reznikovsaffvolberg}. What allows us to obtain a stronger result are two ingredients: (a) the precise statement on homogeneous spaces of Proposition~\ref{prop:kth_moment} and, in particular, the study of the case of spheres described in Remark~\ref{tk_spheres}; and (b) the fact that we don't need to restrict to convex sets $A$ simplifies our constructions.
\medskip

Let $G$ be a locally compact topological group. We recall that a metric space $X$ is a \emph{homogeneous space with group $G$} if there exists a transitive $G$-action on $X$, i.e., for each $x,y\in X$ there exists $g\in G$ such that $g(x)=y$. In this case we may assume that there exists a subgroup $H\subset G$ such that $X=G/H$, endowed with the canonical multiplication action of $G$ (see \cite{helgason}). In this case $G$ acts on $X$ transitively. If $G$ is compact, then we denote by $\mathcal H_{X,G}$ the unique probability measure on $X$ that is invariant under each $g\in G$, which is the projection of the Haar measure of $G$.
\begin{proposition}\label{prop:kth_moment}
Let $G$ be a locally compact topological group and $X$ be a compact homogeneous space with group $G$ and let $K:X\times X\to(-\infty,+\infty]$ be a lower semicontinuous kernel that satisfies $K(g(x),g(y))=K(x,y)$ for every $x,y\in X$ and for every $g\in G$. Then the continuous single-plate polarization problem
\begin{equation}\label{polar_spin}
T_K(X):=\max_{\mu\in\mathcal M_1(X)}\min_{y\in X}\int K(x,y) d\mu(x)
\end{equation}
%=\frac1{2^k}\max_{\mu\in\mathcal P(\mathbb S^{p-1})}\min_{y\in\mathbb S^{p-1}}\int(2-|x-y|^2)^k d\mu(x)
is realized by $\mathcal H_{X,G}$. Moreover, a probability measure $\mu\in\mathcal M_1(X)$ is an optimizer of \eqref{polar_spin} if and only if the $K$-potential of $\mu$ is constant on $X$, and we have
\begin{equation}\label{constpot}
\int K(x,y)d\mu(x)=\int K(x,y)d\mathcal H_{X,G}(x)=T_K(X) \quad\mbox{for all }y\in X.
\end{equation}
\end{proposition}
\noindent As emphasized in Remark \ref{tk_spheres}, polarization-optimizing measures need not be unique.

\medskip

\noindent For use in the following proof, we introduce the notation $f_\#\mu\in\mathcal M(Y)$ to denote the \emph{pushforward} of a Radon measure $\mu\in\mathcal M(X)$ by the measurable function $f:X\to Y$, and is defined by requiring that, for every test function $g\in C^0(Y)$, there holds
\[
\int g(y)df_\#\mu(y)=\int g(f(x)) d\mu(x).
\]
\begin{proof}
% The equality in \eqref{polar_spin} follows from the property that $x,y\in\mathbb S^{p-1}$. Let $\mu$ be a maximizer in \eqref{best_kth_moment_sphere}.
Using the fact that $\mathcal H_{X,G}$ and $K$ are $G$-invariant and $G$ acts transitively on $X$, we find that for any $x,x_0\in X$, there  exists $g_{x,x_0}\in G$ such that $g_{x,x_0}(x)=x_0$, $(g_{x,x_0})_\# \mathcal H_{X,G}=\mathcal H_{X,G}$ and for any $x^\prime\in X$, there holds $K(x,x^\prime)=K(x_0,g_{x,x_0}(x^\prime))$. This allows us to write
\begin{eqnarray}\label{symm}
\int K(x,x^\prime)d\mathcal H_{X,G}(x^\prime)&=&\int K(x_0,g_{x,x_0}(x^\prime))d\mathcal H_{X,G}(x^\prime)\nonumber\\
&=&\int K(x_0,x^\prime)d\left((g_{x,x_0})_\#\mathcal H_{X,G}\right)(x^\prime)= \int K(x_0,x^\prime)d\mathcal H_{X,G}(x^\prime).
\end{eqnarray}
Using \eqref{symm} and the fact that $\mathcal H_{X,G}$ and $\mu$ are probability measures, we may compare the minima of the potentials generated by $\mu$ and $\mathcal H_{X,G}$ as follows:
\begin{eqnarray}
\min_{y\in X}\int K(x,y) d\mu(x) &\le& \int \int K(x,x^\prime)d\mu(x)d\mathcal H_{X,G}(x^\prime)=\int \int K(x,x^\prime)d\mathcal H_{X,G}(x^\prime)d\mu(x)\label{min}\\
&\stackrel{\text{\eqref{symm}}}{=}&\int K(x_0,x^\prime)d\mathcal H_{X,G}(x^\prime)=\min_{y\in X}\int K(y,x^\prime)d\mathcal H_{X,G}(x^\prime)\nonumber .
\end{eqnarray}
This shows that $\mathcal H_{X,G}$ realizes the maximum in \eqref{polar_spin}, and thus  \eqref{constpot} holds. If the minimum in \eqref{min} is not achieved at all points $y\in X$, then a strict inequality holds in \eqref{min} implying that $\mu$ is not a maximizer.
\end{proof}
\begin{rmk}\label{tk_spheres}
We note, as a special case of the above, that we could take $K(x,y):=\langle x,y\rangle^k$ with $k\in\mathbb N$ an even integer, $X=\mathbb S^{p-1}$ and $G=O(p)$, where $O(p):=\{M\in \mathbb R^{p\times p}:\ M^t=M^{-1}\}$ is the group of orthogonal matrices, acting on $X$   by $M(x):= M\cdot x$. In this case the optimal $K$-polarization can be explicitly computed. Denoting by $\bar \sigma$ the uniform measure on $\mathbb S^{p-1}$, we have
\begin{equation}\label{eqtk_spheres}
T_{\langle\cdot,\cdot\rangle^k}(\mathbb S^{p-1})=\int_{\mathbb S^{p-1}}\langle x_0,x^\prime\rangle^kd\bar\sigma(x^\prime)=\frac{|\mathbb S^{p-2}|}{|\mathbb S^{p-1}|}B\left(\frac{k+1}{2},\frac{p-1}{2}\right)=\frac{1}{\sqrt{\pi}}\ \frac{\Gamma\left(\frac{p}2\right)\Gamma\left(\frac{k+1}{2}\right)}{\Gamma\left(\frac{p+k}{2}\right)}\ ,
\end{equation}
where $B(\cdot,\cdot)$ denotes the Beta function and $\left|\mathbb S^d\right|$ is the surface area of $\mathbb S^d$.

\medskip

As a special case which will be used in the proof of the next lemma, we note that for $k=2$ the above expression gives $T_{\langle\cdot,\cdot\rangle^2}(\mathbb S^{p-1})=1/p$, and this value is also achieved as the continuous single-plate polarization of the measure $\mu_p:=(p+1)^{-1}\sum_{i=0}^p\delta_{v_i}$, where $\omega_{\triangle_p}:=\{v_0,\ldots, v_p\}$ is the (multi)set of vertices of a regular simplex inscribed in $\mathbb S^{p-1}$. This fact is a consequence of the property that $\omega_{\triangle_p}$ is a spherical $2$-design, see \cite{dgs}.
\end{rmk}
\begin{lemma}\label{lem:no_concentr}
Let $p\ge 2$ and $A\subset\mathbb R^p$ be a compact set. Then for each $s>p-2$, there exists a constant $0<c_{s,p}<1/2$ depending only on $p$ and $s$ such that if $N$ is an integer such that $\mathcal P_s^*(A,N)<+\infty$, for any $\mathcal P_s^*(A,N)$-optimizing multiset $\omega_N^*$ and any $\hat x\in\R^p\setminus A$ there holds
\begin{equation}\label{claim_main}\#\left[\omega_N^*\cap B\left(\hat x, c_{s,p}\op{dist}(\hat x, A)\right)\right]\le p\ .
\end{equation}
\end{lemma}
\begin{figure}[!h]
\centering
\includegraphics[width=15cm]{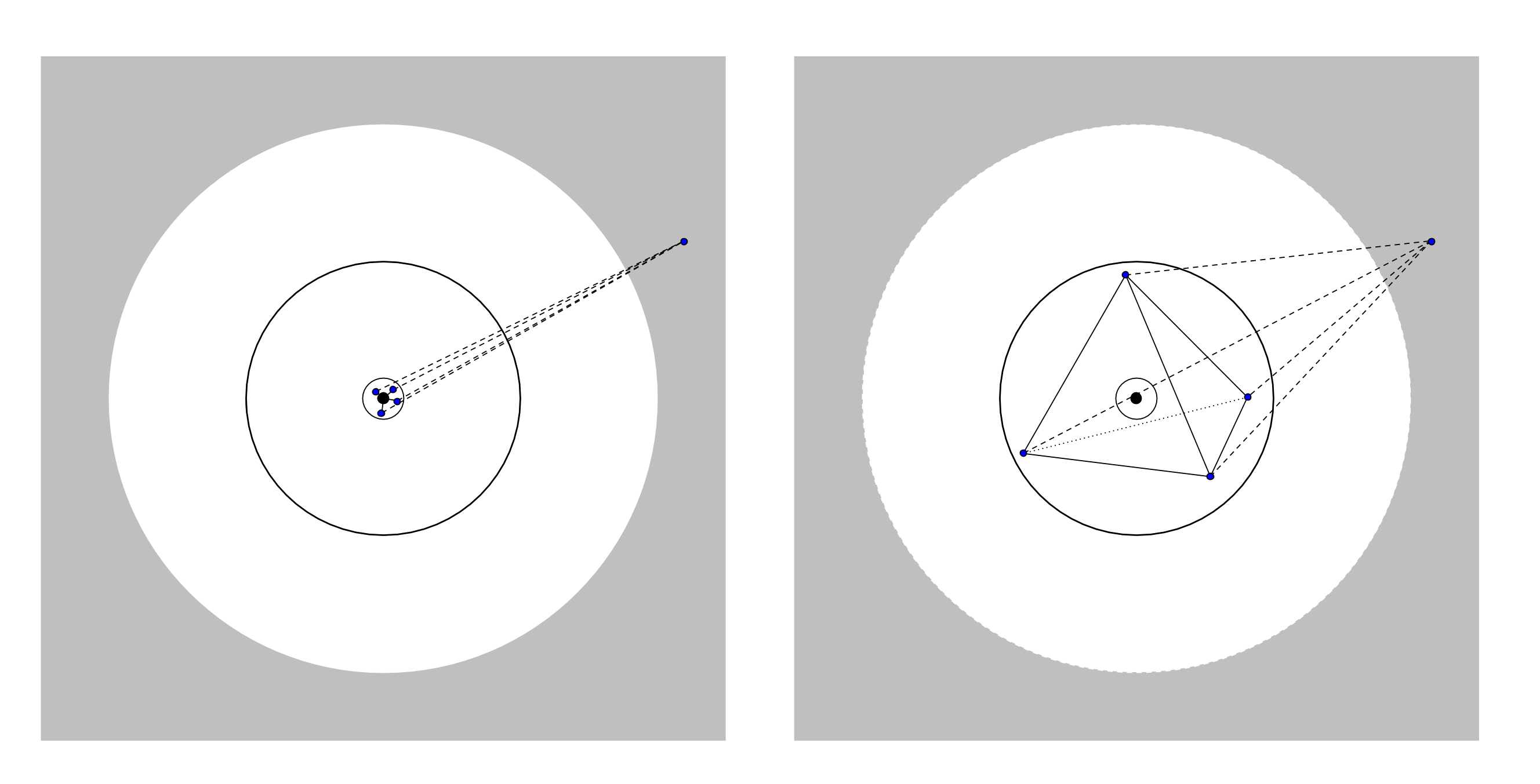}
\caption{Idea for the proof of Lemma~\ref{lem:no_concentr}, for $p=3$: (left) the four points in the small sphere, create, in the grey area, a potential close to $4$ times the one of a charge in the center; (right) after moving the points out to form a regular simplex, the potential in the grey area is then increased, as can be shown by using the Taylor expansion of the potential, using the assumption that $s>p-2$.}
\label{fig:no_concentr}
\end{figure}
\begin{rmk} Regarding the sharpness of the lemma, Propositions \ref{prop:small_N} and \ref{prop:slep-2} show that the bound in \eqref{claim_main} cannot be replaced by $p-1$ when $A=\mathbb S^{p-1}$.
\end{rmk}
\begin{proof}[Proof of Lemma \ref{lem:no_concentr}:]
\noindent\textbf{Step 1.} To simplify notation, we write
\begin{equation}\label{defr}
r:=\op{dist}(\hat x, A)>0.
\end{equation}
For $c_1\in(0,1/2)$, assume that $\omega_N^*$ contains $p+1$ points inside $B(\hat x, c_1r)$, say
\begin{equation}\label{contradiction}
x_0,\ldots,x_p\in B(\hat x, c_1r)\ .
\end{equation}
Our goal is to prove that there exists a constant $c_{s,p}\le 1/2$ such that $c_1<c_{s,p}$ gives a contradiction to the minimality of $\omega_N^*$.

\medskip

\noindent\textbf{Step 2.} For $c_2>0$ consider the new configuration
\begin{equation}\label{define_new_om}
\widetilde{\omega}_N:= \omega_N^*\setminus\{x_0,\ldots,x_p\}\bigcup\left(\bar x+c_2 r \omega_{\triangle_p}\right),\quad\mbox{where}\quad \bar x:=\frac{1}{p+1}\sum_{i=0}^px_j,
\end{equation}
and $\omega_{\triangle_p}$ is as in Remark \ref{tk_spheres}, the set of vertices of a regular simplex inscribed in $\mathbb S^{p-1}$.

\medskip

\noindent For $y\in \R^p$, define
\[
f_y(x):=\frac{1}{\verti{y-x}^s}.
\]
We will consider the following Taylor expansions of $f_y$ around $\bar x$ under the condition that $|x|\le \frac14|\bar x- y|$:
\begin{eqnarray}
f_y(\bar x+x)&=&f_y(\bar x) - s \frac{\langle x, \bar x - y\rangle}{\verti{\bar x-y}^{s+2}} + \frac{\mathcal R_2(x)}{\verti{\bar x-y}^{s+2}}\label{taylor_fy_2}\\
&=&f_y(\bar x) - s \frac{\langle x, \bar x - y\rangle}{\verti{\bar x-y}^{s+2}} + \frac{s}2\frac{(s+2)\langle x, \bar x -y\rangle^2 - \verti{x}^2\verti{\bar x-y}^2}{\verti{\bar x -y}^{s+4}} + \frac{\mathcal R_3(x)}{\verti{\bar x-y}^{s+3}},\label{taylor_fy}
\end{eqnarray}
where for some constants $\gamma_{s,p}>0$ depending only on $p,s$ we have
\begin{equation}\label{remainder_bounds}
\verti{\mathcal R_2(x)}\le \gamma_{s,p}|x|^2 \quad\mbox{and}\quad\verti{\mathcal R_3(x)}\le \gamma_{s,p}|x|^3.
\end{equation}
\textbf{Step 3.} As discussed in Remark \ref{tk_spheres}, for $k=2$ we have $T_{\langle\cdot,\cdot\rangle^2}(\mathbb S^{p-1})=1/p$, which is attained by the regular simplex $\omega_{\triangle_p}$. Therefore the condition $s>p-2$ can be rewritten as $(s+2)T_{\langle\cdot,\cdot\rangle^2}(\mathbb S^{p-1})>1$.  Thus there exists a positive number $\varepsilon_{s,p}>0$ depending only on $s,p$ such that if $v_i$ denote, as in Remark \ref{tk_spheres}, the vertices of a regular simplex inscribed in $\mathbb S^{p-1}$, then
\begin{equation}\label{finite_pk}
\frac{s+2}{p+1}\sum_{i=0}^p\langle v_i,y\rangle^2- 1= (s+2)T_{\langle\cdot,\cdot\rangle^2}(\mathbb S^{p-1})-1>2\varepsilon_{s,p}, \quad\mbox{for all}\quad y\in\mathbb S^{p-1}\ .
\end{equation}
Now note that for any $v\in\R^p$ there holds
\begin{equation}\label{bound2}
\sum_{i=0}^p v_i = 0, \quad \frac{s+2}{p+1}\sum_{i=0}^p \langle c_2rv_i, v\rangle^2 \ge (1+2\varepsilon_{s,p}) c_2^2r^2\verti{v}^2, \quad \frac{1}{p+1}\sum_{i=0}^p\verti{c_2rv_i}^2 = c_2^2 r^2 \ ,
\end{equation}
where for the middle inequality we used \eqref{finite_pk}. From the assumption \eqref{contradiction}, since $c_1<c_{s,p}\le 1/2$ and $\bar x\in\mathrm{conv}\{x_0,\ldots,x_p\}\subset B(\hat x,c_1r)$, we obtain
\begin{equation}\label{assum_far}
\min_{y\in A}|\bar x -y|\ge (1-c_1)r\ge \frac{r}{2}\quad\mbox{and}\quad \max_{0\le j\le p}|x_j-\bar x|\le2c_1r.
\end{equation}
Conditions \eqref{assum_far} and the fact that $\omega_{\triangle_p}\subset\mathbb S^{p-1}$ allow to obtain that for $c_1\le 1/16$ and $c_2\le 1/4$ the conditions $|x|\le \frac14|\bar x- y|$ required for \eqref{taylor_fy} and \eqref{remainder_bounds} to hold are satisfied for $x=x_j-\bar x$ and for $x\in c_2 r \omega_{\triangle_p}$. 

\medskip
\noindent We now sum \eqref{taylor_fy} over $x\in c_2 r \omega_{\triangle_p}$. Using \eqref{remainder_bounds}, \eqref{bound2} and the first bound in \eqref{assum_far}, we can then estimate
\begin{eqnarray}\label{bound_newconf}
\frac1{p+1}\sum_{i=0}^pf_y(\bar x+c_2rv_i)&\ge&  f_y(\bar x)+ c_2^2 r^2\frac{s\ \varepsilon_{s,p}}{\verti{\bar x -y}^{s+2}} - c_2^3\gamma_{s,p}\frac{r^3}{\verti{\bar x - y}^{s+3}}\nonumber\\
&\ge&  f_y(\bar x)+c_2^2\left( s\ \varepsilon_{s,p} - 2c_2\gamma_{s,p}\right)\frac{r^2}{\verti{\bar x -y}^{s+2}}\ .
\end{eqnarray}
\textbf{Step 4.} By writing the expansion \eqref{taylor_fy_2} at $x=x_j-\bar x$, for $j\in\{0,\ldots,p\}$ we find
\[
f_y(x_j)=f_y(\bar x) - s\frac{\langle x_j-\bar x, \bar x -y\rangle}{\verti{\bar x - y}^{s+2}}+\frac{\mathcal R_2(x_j-\bar x)}{\verti{\bar x - y}^{s+2}}.
\]
We now sum the above equation over $j=0,\ldots,p$, and divide by $p+1$, and get
\begin{eqnarray}\label{bound_oldconf_1}
\frac{1}{p+1}\sum_{j=0}^p f_y(x_j)
&=& f_y(\bar x) - \frac{s}{p+1}\frac{1}{\verti{\bar x - y}^{s+2}}\left\langle \sum_{j=1}^p (x_j-\bar x), \bar x -y\right\rangle + \frac{1}{p+1}\frac{1}{\verti{\bar x - y}^{s+2}}\sum_{j=0}^p\mathcal R_2(x_j-\bar x)\nonumber\\
&=&f_y(\bar x) + \frac{1}{p+1}\frac{1}{\verti{\bar x - y}^{s+2}}\sum_{j=0}^p\mathcal R_2(x_j-\bar x)\nonumber\\
&\le&f_y(\bar x)+4c_1^2\,\gamma_{s,p}\,\frac{r^2}{\verti{\bar x-y}^{s+2}}\ ,
\end{eqnarray}
where to obtain the second line we note that the first term on the right in the first line vanishes due to the definition of $\bar x$ from \eqref{define_new_om}, and for obtaining the inequality in the last line we use the first bound in \eqref{remainder_bounds} together with the second bound from \eqref{assum_far}.

\medskip

\noindent 
\textbf{Step 5.} Now, using \eqref{define_new_om}, we find that the bounds \eqref{bound_newconf} and \eqref{bound_oldconf_1} give
\begin{subequations}\label{expressi}
\begin{eqnarray}
\sum_{x\in\widetilde{\omega}_N}\frac{1}{\verti{x-y}^s}- \sum_{x\in\omega_N^*}\frac{1}{\verti{x-y}^s}&=&\sum_{i=0}^pf_y(\bar x+c_2rv_i)- \sum_{j=0}^pf_y(x_j)\nonumber\\
&=& \sum_{i=0}^p\left[f_y(\bar x+c_2rv_i)- f_y(\bar x)\right]- \sum_{j=0}^p\left[f_y(x_j)  -f_y(\bar x)\right]\nonumber\\
&\ge& (p+1)\left[c_2^2(s\ \varepsilon_{s,p} - 2c_2\gamma_{s,p}) - 4c_1^2\gamma_{s,p}\right]\frac{r^2}{\verti{\bar x-y}^{s+2}}\label{compare_taylor}\\
&:=&E_{s,p}(c_1,c_2)\frac{r^2}{\verti{\bar x-y}^{s+2}}\label{expressioo}\ ,
\end{eqnarray}
\end{subequations}
for any $y\in A$. As a function of $c_2$, the value of $c_2^2(s\ \varepsilon_{s,p} - 2c_2\gamma_{s,p})$ in \eqref{compare_taylor} is positive and increasing for $c_2\in (0, s\varepsilon_{s,p}/(3\gamma_{s,p})]$, and we will take $c_2=\bar c_2:=\min\{1/4,\  s^3\varepsilon_{s,p}^3/(27\gamma_{s,p}^2)\}>0$. By comparing this value with the term $4c_1^2\gamma_{s,p}$ from \eqref{compare_taylor}, we find that if $c_1$ satisfies
\begin{equation}\label{prop_c1}
c_1<\min\left\{\sqrt{\bar c_2^2(s\ \varepsilon_{s,p} - 2\bar c_2\gamma_{s,p})}\,,\,\frac{1}{16}\right\}:=c_{s,p},
\end{equation}
then the expression defined in \eqref{expressioo} satisfies $E_{s,p}\left(c_1,\bar c_2\right)>0$. If $\widetilde{y}\in A$ achieves the minimum of $P_s(\widetilde{\omega}_N,y)$ in \eqref{compare_taylor} and $c_1<c_{s,p}$, then from \eqref{expressi} we obtain
\begin{equation}\label{final_noconc_contrad}
P_s(\widetilde{\omega}_N,A)=\sum_{x\in\widetilde{\omega}_N}\frac{1}{\verti{x-\widetilde{y}}^s}\ge E_{s,p}\left(c_1,\bar c_2\right)\frac{r^2}{\verti{\bar x-\widetilde{y}}^{s+2}} + \sum_{x\in\omega_N^*}\frac{1}{\verti{x-\widetilde{y}}^s}> P_s(\omega_N^*,A),
\end{equation}
which contradicts the optimality of $\omega_N^*$. Therefore the value $c_{s,p}$ defined in \eqref{prop_c1} is as required in Step 1, and this concludes the proof of the lemma.
\end{proof}
\begin{proof}[Completion of Proof of Theorem \ref{thm:bound_K}:]
We first note that as a consequence of Proposition \ref{prop:covex_equal}, for each $N$, optimal configurations $\omega_N^*$ are contained in the convex hull $\op{conv}(A)$, which has diameter equal to $\op{diam}(A)<\infty$. We then note that, for  $c_{s,p}$ chosen according to Lemma \ref{lem:no_concentr},
\begin{equation}\label{coverepsilon}
\op{conv}(A)\setminus A_\epsilon\subset\bigcup_{x\in \op{conv}(A)\setminus A_\epsilon} B\left(x,\epsilon c_{s,p}\right)\subset \left(\op{conv}(A)\right)_{\epsilon c_{s,p}}.
\end{equation}
We then apply Besicovitch's covering lemma and find a finite subcover of $\op{conv}(A)\setminus A_\epsilon$ by at most $N_{\mathrm{Bes},p}$ families of disjoint balls. Note that, in particular, we have for all $x\notin A_\epsilon$ that $B(x,\epsilon c_{s,p})\subset B(x,\op{dist}(x, A))$. We may thus apply the bound \eqref{claim_main} to each one of the above $N_{\mathrm{Bes},p}$ families, and then sum the bounds. Thus we find that via a direct volume bound
\[
\#\left(\omega_N^*\setminus A_\epsilon\right)\le p\, N_{\mathrm{Bes},p} \frac{\mathcal L_p\left(\left(\op{conv}(A)\right)_{\epsilon c_{s,p}}\right)}{\mathcal L_p(B(0,\epsilon c_{s,p}))},
\]
which concludes the proof of the theorem.
\end{proof}
%

%
%
%%%%%%%%%%%%%%%%%%%%%%%%%%%%%%%%%%%%%%%%%%%%%%%%%%%%%%%%%%%%%%%%%%%%%%
\section{Weak   separation and proof of Theorem \ref{thm:compact_body_old}}\label{sec:pf_compact_body_old}
%%%%%%%%%%%%%%%%%%%%%%%%%%%%%%%%%%%%%%%%%%%%%%%%%%%%%%%%%%%%%%%%%%%%%%%
%
%
In this section we first present Proposition \ref{prop:pointsep} on the weakly well-separated property of maximal unconstrained polarization configurations and its consequences in Proposition \ref{prop:max_away_points}. Then we prove the general point replacement result of Proposition \ref{prop:replace_points}. Finally, Propositions \ref{prop:max_away_points} and \ref{prop:replace_points} together with Theorem \ref{thm:bound_K} allow us to prove Theorem \ref{thm:compact_body_old} in Section \ref{ssec:proof_th_cbo}.
\subsection{Weakly well-separated families of configurations}
The results on the asymptotics of $\mathcal P_s^*(A,N)$ presented in this section are set in a framework similar to the one for $ \mathcal P_s(A,N)$ from \cite{bhrs2016preprint}. 

\medskip

Recall the following definition from \cite{2017reznikovsaffvolberg} and \cite{2017hrsv}:
\begin{definition}\label{def:weaksep}
Let $0<d\le p$ be integers. A family $\Omega$ of multisets $\omega \subset \R^p$   is called \emph{weakly well-separated for dimension $d$ and parameter $\eta>0$} if there exists a number $M>0$ such that for each $\omega\in\Omega$ and each $x\in\R^p$, there holds
\begin{equation}\label{weakwellsep}
\#\left(\omega\cap B(x,\eta\cdot(\#\omega)^{-1/d})\right)\le M\ ,
\end{equation}
where $B(x,r)$ denotes the $p$-dimensional open ball with center $x$ and radius $r$.
\end{definition}

\begin{proposition}\label{prop:pointsep}
Under the same conditions as in Theorem~\ref{thm:compact_body_old}, 
  there exists a constant $\eta>0$ depending on $s$, $d$ and $A$ such that the family of all optimal configurations
\begin{equation}\label{omegas}
\Omega_s:=\left\{\omega\subset \mathbb R^p:\ P_s(A,\omega)=\mathcal P_s^*(A,\#\omega)\right\}
\end{equation}
is weakly well-separated for dimension $d$ and parameter $\eta$ with $M=p$.
\end{proposition}
\begin{rmk}Note that the value $M=p$ in the above proposition is optimal, as a consequence of Proposition \ref{prop:small_N}. The proof in \cite{2017reznikovsaffvolberg} is done for $M=2p-1$ but can be modified along the lines of the proof of our Lemma \ref{lem:no_concentr} (applying the perturbation as in Figure \ref{fig:no_concentr}) in order to achieve the value $M=p$ as stated in Proposition \ref{prop:pointsep}.
\end{rmk}
Proposition \ref{prop:rough_upper_bound} below follows as in \cite[Thm.\,2.4]{2013erdelyisaff} (simply note that the restriction $\omega_N\subset A$ for finite-$N$ configurations is never used in the proof from \cite{2013erdelyisaff}).
\begin{proposition}\label{prop:rough_upper_bound}
Let $p\ge 2$ be an integer and let $1\le d\le p$ be a real number.\\
 
\noindent
{\rm (i)} For $s\ge d$ there exists a constant $c_s>0$ depending only on $s$, such that if $A\subset \mathbb R^p$ is a compact set such that $\mathcal H_d(A)>0$, then for $N\ge 2$ there holds
\begin{subequations}
\begin{eqnarray}
\mathcal P_s^*(A,N)&\le& \frac{c_s}{s-d}N^{s/d}\ ,\quad\mbox{ if }\quad s>d\ ,\label{up_bd_s>d}\\
\mathcal P_d^*(A,N)&\le& c_d N\log N\ ,\quad\mbox{ if }\quad s=d\ .\label{up_bd_s=d}
\end{eqnarray}
\noindent
{\rm (ii)} If $A\subset \mathbb R^p$ is a compact set then there exists a probability measure $\mu_A$ supported on $A$ and a constant $C_A\in(0,\infty)$ such that
\[
\int_A\frac{1}{|x-y|^s}d\mu_A(y)\le C_A\ ,\quad\mbox{ for }\quad x\in \R^p\ ,
\]
then, for all $N\ge 1$,
\begin{equation}\label{up_bd_s_gen}
\mathcal P_s^*(A,N)\le N C_A\ ,\quad\mbox{ for }\quad s>0\ .
\end{equation}
\end{subequations}
\end{proposition}

\medskip

\noindent The next result is proved as in \cite[Prop. 1.5]{2017reznikovsaffvolberg} for the case $s>d$ and as in \cite[Thm.\,2.5]{2017hrsv} for the cases $p-2<s<d$. Indeed, in the proofs of those results the fact that the configurations are constrained to the set $A$ is not used; furthermore, Proposition \ref{prop:rough_upper_bound} precisely replaces the use of results from \cite{2013erdelyisaff} in those proofs.
\begin{proposition}\label{prop:max_away_points}
Under the same hypotheses on $p$, $d$, $d'$, $s$, $A$ and $\mu_{s,A}$ as in Proposition \ref{prop:pointsep}, let $\omega_N\subset \mathbb R^p$ be an $N$-point configuration and $y^*\in A$ be a point such that
\[
\sum_{x\in \omega_N}|x-y^*|^{-s}=\min_{y\in A}\sum_{x\in \omega_N}|x-y|^{-s}
\]
is achieved. There exists a constant $C>0$, which depends only on $s$ if $s>d$ and only on $s$ and on the upper $d$-regularity constant of $\mu_{s,A}$ if $p-2<s<d\le d'\le p$, but is in either case independent of $N$, of $\omega_N$ and of the choice of $y^*$, such that
\begin{equation}\label{max_away_points}
\min_{x\in\omega_N}|x-y^*| \ge C N^{-1/d}\ .
\end{equation}
\end{proposition}
\noindent With Proposition \ref{prop:max_away_points} at hand, Proposition \ref{prop:pointsep} follows by a modification of the proof of Lemma \ref{lem:no_concentr}. These results allows us to proceed with the same overall strategy as for the analogous result for constrained polarization $\mathcal P_s(N,A)$, see \cite[Thm.\,2.3] {2017reznikovsaffvolberg} and \cite[Thm.\,2.3]{2017hrsv}.
\begin{proof}[{Proof of Proposition \ref{prop:pointsep}}]
For any $N$-point configuration $\omega_N$, set
\[
S_s(A,\omega_N):=\left\{y\in A:\ \sum_{x\in\omega_N}K_s(x,y)=\min_{y'\in A}\sum_{x\in\omega_N}K_s(x,y')\right\}\,.
\]
Then, by the bound \eqref{max_away_points}, for $C>0$ as in Proposition \ref{prop:max_away_points},
\begin{equation}\label{separation}
\op{dist}(\omega_N,S_s(A,\omega_N))\ge CN^{-1/d} \ .
\end{equation}
Now assume that for a radius $R>0$ and for some $x\in \mathbb R^p$ and some optimal $s$-polarization configuration $\omega_N^*$ there exist $p+1$ distinct points
\begin{equation}\label{points_bR}
x_0,\ldots,x_p\in \omega_N^*\cap B(x,R).
\end{equation}
By using the hypothesis that $s>p-2$, we will proceed along the same lines as in the proof of Lemma \ref{lem:no_concentr} in order to reach a contradiction if
\begin{equation}\label{hyp_R}
R<\eta N^{-1/d},\quad \mbox{where}\quad\eta:=\frac12 c_{s,p}C>0,
\end{equation}
where $C$ is as in Proposition \ref{prop:max_away_points} and $c_{s,p}$ is as in Lemma \ref{lem:no_concentr}. Indeed, set $R=\frac12 c_1 r$, where $r:= CN^{-1/d}$ and  $0<c_1<c_{s,p}$. Then, with these values of $c_1$ and $r$, and for a choice of $c_2>0$ to be determined, we can use the same formulas \eqref{define_new_om} as in Step 2 of the proof of Lemma \ref{lem:no_concentr} to define $\bar x$ and $\widetilde{\omega}_N$. Due to \eqref{separation}, to \eqref{points_bR} and to the choice of $c_1$, we verify that for any $y=\tilde y\in S_s(A,\widetilde{\omega}_N)$ the bounds \eqref{assum_far} hold. Then the estimates of the proof of Lemma \ref{lem:no_concentr} continue to hold, and we determine with the same choice of $c_2$ as in Step 5 that \eqref{expressi} and \eqref{final_noconc_contrad} hold for $y=\tilde y$. As a consequence of \eqref{final_noconc_contrad}, and of the assumed optimality of $\omega_N^*$, we have
\begin{equation}
\mathcal P_s(A,N)\ge P_s(A,\widetilde{\omega}_N) > P_s(A,\omega_N^*)=\mathcal P_s(A,N),
\end{equation}
which is a contradiction. It follows that under condition \eqref{hyp_R} there cannot exist $p+1$ points such that \eqref{points_bR} holds, which concludes the proof of the proposition.
\end{proof}
\subsection{Proof of Theorem \ref{thm:compact_body_old}}\label{ssec:proof_th_cbo}
The main new tool that we will use in the proof of Theorem \ref{thm:compact_body_old} is the geometric result of Proposition \ref{prop:replace_points} below, which holds for a very general class of kernels. It allows us to replace a charge $x$ positioned at positive distance from $A$ by a \emph{bounded number} of charges in $A$, without decreasing the polarization value on $A$. The principle underlying this proposition is illustrated in Figure~\ref{fig:replace_points}.
\begin{figure}[!h]
\centering
\includegraphics[width=4cm]{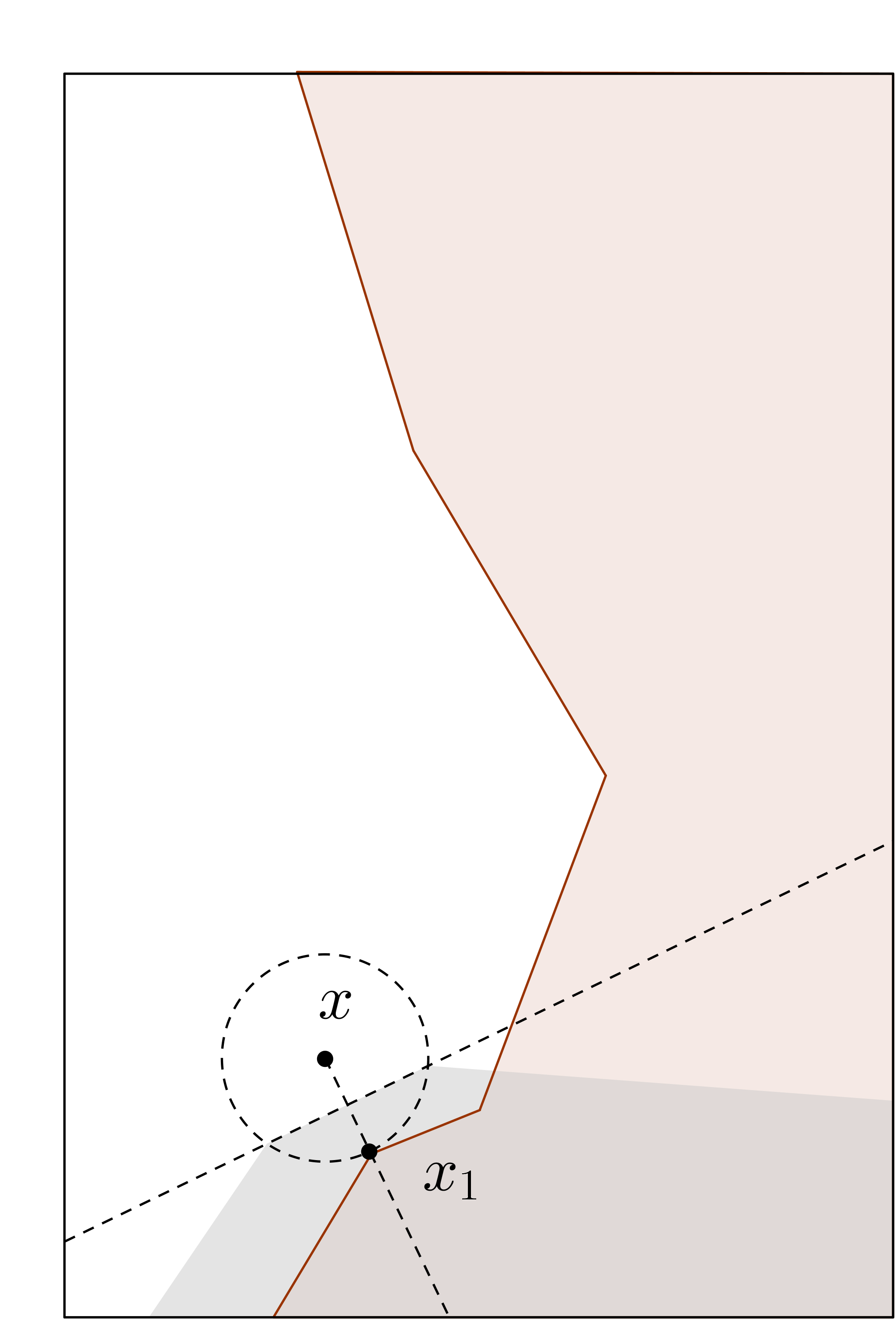}\quad
\includegraphics[width=4cm]{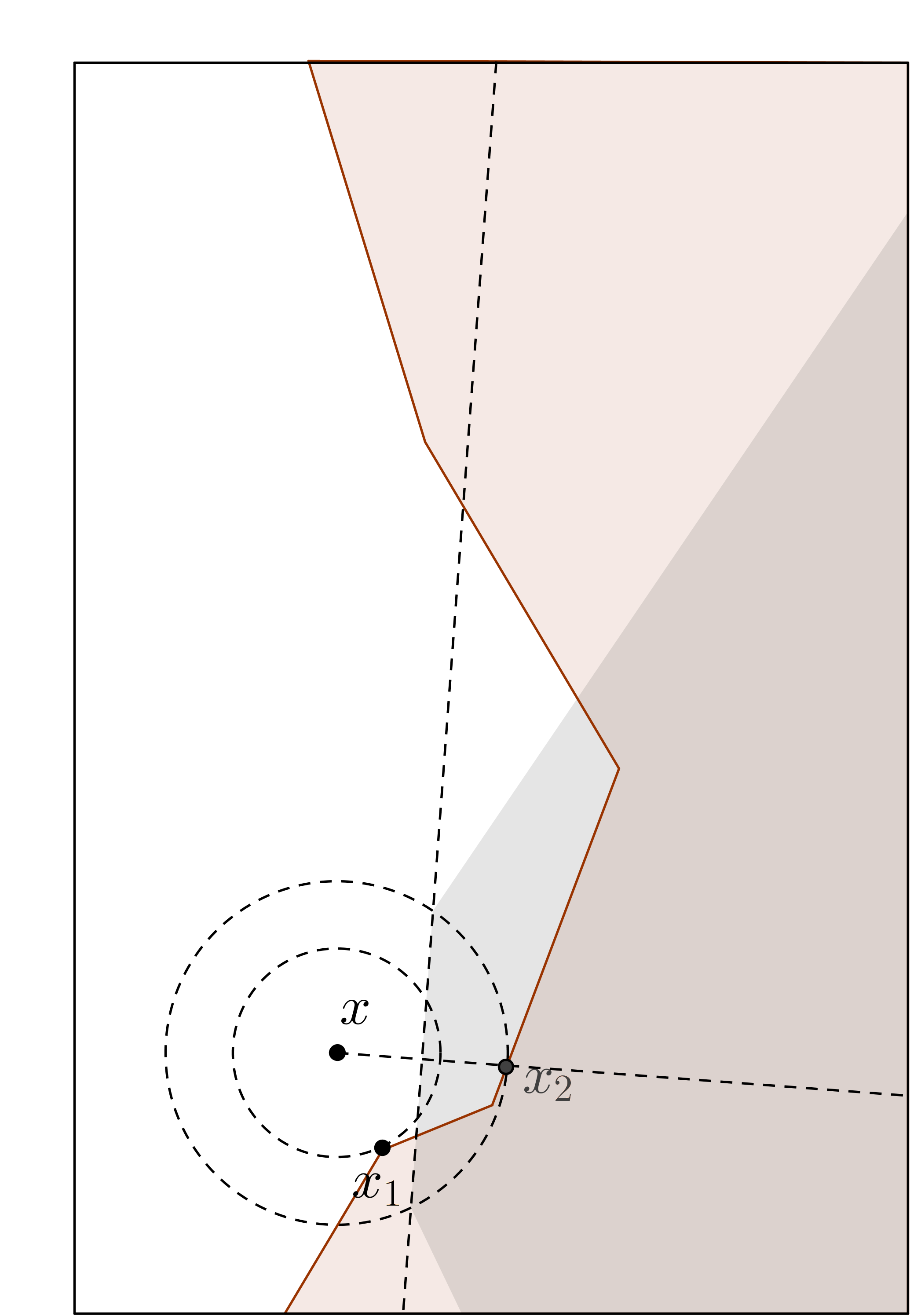}\quad
\includegraphics[width=4cm]{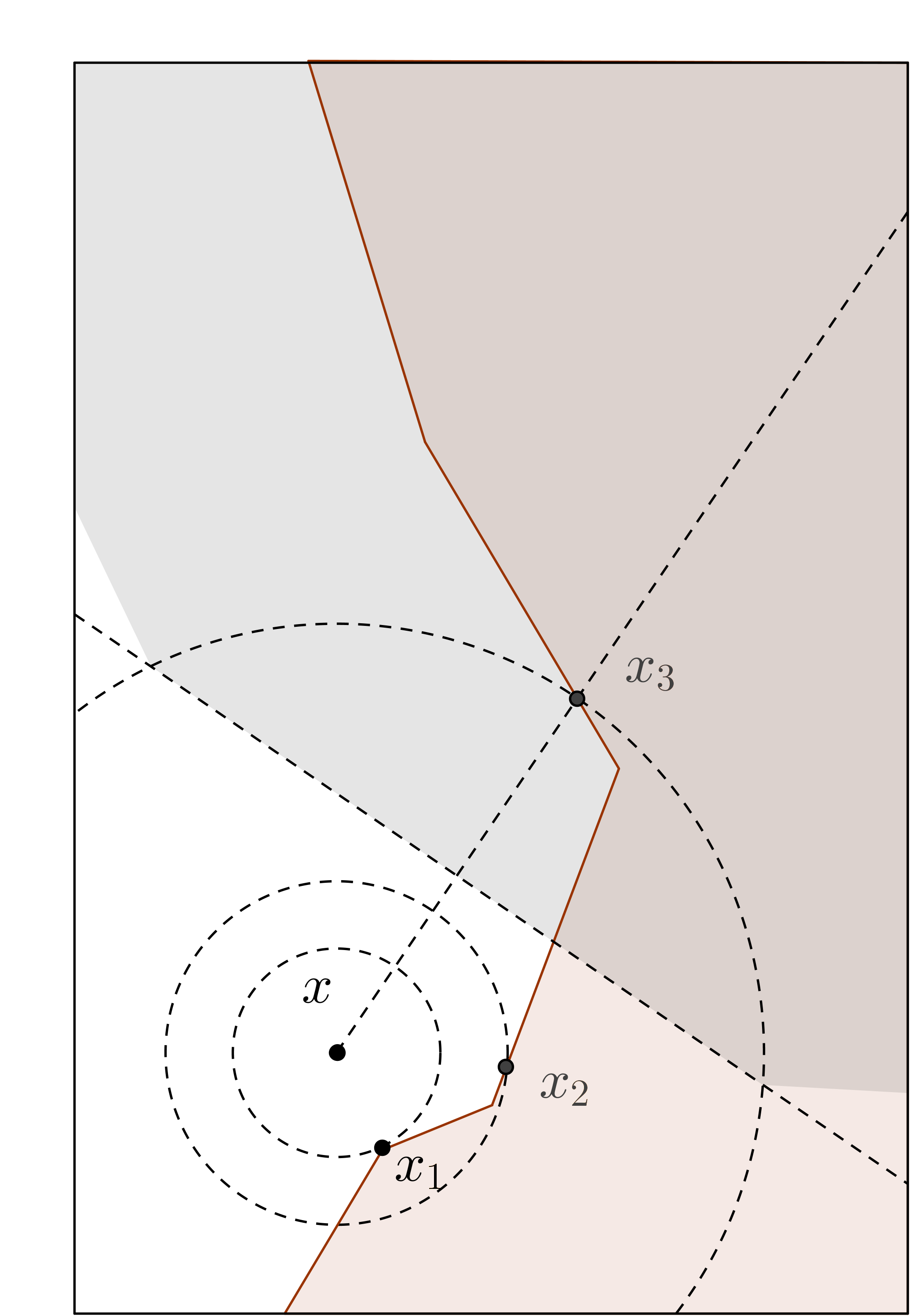}\quad
\caption{The construction from Proposition~\ref{prop:replace_points}, for $p=2$. The set $A$ is shaded in brown. Iteratively we select points $x_j\in A$ such that a charge positioned at $x_j$ creates a higher potential than $x$ (at least) on the intersection of the shaded region (which itself is the intersection of a cone from $x$ and a hyperplane) with $A$. The union of such regions eventually covers $A$. Further, any two of the so-constructed points $x_j$, viewed from $x$, form angles of at least $\pi/3$; thus, by a simple best-packing upper bound on $\mathbb S^{p-1}$, the necessary number of points can be controlled, depending only on the dimension.}
\label{fig:replace_points}
\end{figure}
\begin{proposition}\label{prop:replace_points}
For each $p\ge 2$, let $n_{\pi/6,p}>0$ be the cardinality of the best packing of $\mathbb S^{p-1}$ by spherical caps of angle $\pi/6$. Let $A\subset \R^p$ be a compact set, and let $x\not \in  A$. Then there exist points $x_1,\ldots,x_n\in A$ with $n\le n_{\pi/6,p}$, such that for all decreasing $f:\R^+\to (-\infty,+\infty]$ there holds
\begin{equation}\label{replace_points}
\forall y\in A,\quad f\left(\verti{x-y}\right)\le \max_{1\le j\le n}f\left(\verti{x_j-y}\right).
\end{equation}
\end{proposition}
\begin{proof}
Set
\begin{equation}\label{radial_A_x}
\op{rad}(A,x):=\left\{y\in A:\ \forall \lambda\in[0,1),\ x+\lambda(y-x)\notin A\right\}.
\end{equation}
In other words, $\op{rad}(A,x)$ contains the first contact point with $A$ of each ray starting from $x$ that intersects $A$. Also set
\[
\op{rad}_1(A,x):=\{(y-x)/\verti{y-x}:\ y\in \op{rad}(A,x)\}.
\]
Note that the projection
\begin{equation}\label{proj1}
\pi_{1,x}:\R^p\setminus\{x\}\to \mathbb S^{p-1}, \qquad \pi_{1,x}(y):=\frac{y-x}{\verti{y-x}}
\end{equation}
induces a bijection between $\op{rad}(A,x)$ and $\op{rad}_1(A,x)$.

\medskip

We now iteratively construct the set $x_1,\ldots,x_n$ as required in the statement of the proposition.

\medskip

\noindent
\textbf{Step $1$.} Fix a point $x_1\in\op{rad}(A,x)$ such that
\begin{equation}\label{choice_x1}
|x_1-x|=\min\{|x'-x|: x'\in \op{rad}(A,x)\}.
\end{equation}
As $f$ is decreasing, $f\parenth{\verti{x-y}}\le f\parenth{\verti{x_1-y}}$ for all $y$ belonging to the half-space $H(x,x_1)$, where for $a\neq b\in \mathbb R^p$ we set
\begin{equation}\label{h_12}
H(a,b):=\{y\in\R^p:|y-a|\ge |y-b|\}=\left\{y\in\R^p:\ \langle y-a,b-a\rangle \ge \frac12\verti{b-a}^2\right\}.
\end{equation}
We next let $\mathcal{K}(x_1)\subset \mathbb S^{p-1}$ be the spherical cap of angle $\pi/3$ centered at $\pi_{1,x}(x_1)$. Then  
\[
\mathcal{K}(x_1)=\pi_{1,x}\left\{u\in B({x,\verti{x-x_1}}) :\langle \pi_{1,x}(u),\pi_{1,x}(x_1)\rangle \ge \frac{1}{2}\right\}=\pi_{1,x}\left(H(x,x_1)\cap B({x,\verti{x-x_1}})\right),
\]
and by \eqref{choice_x1} and \eqref{radial_A_x}, we obtain
\begin{equation}\label{K1_prop}
A\cap\pi_{1,x}^{-1}(\mathcal{K}(x_1))\subset H(x,x_1).
\end{equation}
\textbf{Step $k+1$.} For $k\ge1$, suppose that the points $x_1,\ldots,x_k$ have already  been chosen such that
\begin{subequations}\label{partialverif}
\begin{equation}
\pi_{1,x}(x_1),\ldots,\pi_{1,x}(x_k)\in\mathbb S^{p-1}\quad \mbox{form a }\pi/3\mbox{-separated set,}
\end{equation}
with respect to the geodesic distance on $\mathbb S^{p-1}$ and such that
\begin{equation}\label{partialver1}
A\cap\pi_{1,x}^{-1}\parenth{\bigcup_{j=1}^k \mathcal{K}(x_j)}\subset A\cap \bigcup_{j=1}^kH(x,x_j).
\end{equation}
\end{subequations}
If we next choose $x_{k+1}\in\op{rad}(A,x)\setminus\pi_{1,x}^{-1}\left( \bigcup_{j=1}^k\mathcal{K}(x_j)\right)$ such that
\[
|x_{k+1}-x|=\min\left\{|x'-x|:\ x\in \op{rad}(A,x)\setminus \pi_{1,x}^{-1}\left(\bigcup_{j=1}^k\mathcal{K}(x_j)\right)\right\},
\]
then automatically $\pi_{1,x}(x_{k+1})$ is $\pi/3$-separated from $\pi_{1,x}(x_1),\ldots,\pi_{1,x}(x_k)$. Combining this with the bound \eqref{K1_prop} for the point $x_{k+1}$, conditions \eqref{partialverif} now hold with $k$ replaced by $k+1$. Directly from the definition of $n_{\pi/6,p}$, we see that the above iterative construction must stop at step $n$ for some $n\le n_{\pi/6,p}$. After step $n$ we have
\begin{equation}\label{cover_caps}
\op{rad}_1(A,x)= \bigcup_{j=1}^n\mathcal{K}(x_j)
\end{equation}
and by \eqref{partialver1},
\begin{equation}\label{final_covercones}
f\left(\verti{x-y}\right)\le \max_{1\le j\le n}f\left(\verti{x_j-y}\right)\quad\mbox{for}\quad y\in A\cap \pi_{1,x}^{-1}\parenth{\bigcup_{j=1}^n \mathcal{K}(x_j)}= A.
\end{equation}
The last inclusion in \eqref{final_covercones} follows from \eqref{cover_caps}. The claim \eqref{replace_points} now follows from \eqref{final_covercones}.
\end{proof}
\begin{proof}[Completion of Proof of Theorem \ref{thm:compact_body_old}:]
The statement follows from the two inequalities
\begin{equation}\label{twoineq}
\limsup_{N\to\infty}\frac{\mathcal P_s(A,N)}{\tau_{s,d}(N)}\le\lim_{N\to\infty}\frac{\mathcal P_s^*(A,N)}{\tau_{s,d}(N)}\le \liminf_{N\to\infty}\frac{\mathcal P_s(A,N)}{\tau_{s,d}(N)}.
\end{equation}
The first inequality follows directly from the simple bound \eqref{inequality}, so we only need to prove the second inequality. For this purpose, fix $\epsilon>0$ and consider for fixed $N$ a configuration $\omega_N^*$ optimizing $\mathcal P_s^*(A,N)$. By \eqref{bound_k_1} of Theorem \ref{thm:bound_K} we have
\begin{equation}\label{card_not_ae}
\#\parenth{\omega_N^*\setminus A_\epsilon}\le \kappa_{s,p}\frac{\mathcal L_p\left(\left(\op{conv}(A)\right)_{\epsilon c_{s,p}}\right)}{\epsilon^p}:=C_0(\epsilon).
\end{equation}
Next, for $\eta>0$ depending on $s,d,A$ as in Proposition \ref{prop:pointsep} we use the Besicovitch covering theorem in order to cover $A_\epsilon\setminus A$ by a finite collection of balls of radius $\eta N^{-1/p}$ which is the union of at most $N_{\mathrm{Bes},p}$ collections of disjoint balls, where $N_{\mathrm{Bes},p}$ depends only on $p$. In particular, all balls in the cover are then contained in $(A_\epsilon\setminus A)_\epsilon$ if  $N>(\eta/\epsilon)^p$.

\medskip

By the weak point separation bound of Proposition \ref{prop:pointsep} combined with a volume comparison argument, for $N>(\eta/\epsilon)^p$ we have
\begin{equation}\label{bound_card}
\#\parenth{\omega_N^*\cap\parenth{A_\epsilon\setminus A}}\le \frac{p\ N_{\mathrm{Bes},p}}{\beta_p\eta^p}\mathcal L_p\parenth{\parenth{A_\epsilon\setminus A}_\epsilon}N=:C_1(\epsilon)N\quad\mbox{and}\quad\lim_{\epsilon\to 0}C_1(\epsilon)=0,
\end{equation}
where $\beta_p$ is volume of the $p$-dimensional unit ball $B_p(0,1)$ and where in the last part we used the regularity of the $\mathcal L_p$-measures and the fact that $A$ being compact implies $\mathcal L_p(A)<\infty$.

\medskip

By Proposition \ref{prop:replace_points} for each $x\in\omega_N\cap\parenth{A_\epsilon\setminus A}$ there exists a configuration $\omega_x\subset A$ such that
\begin{equation}\label{omegax}
1\le\#\omega_x\le n_{\pi/6,p},\quad\mbox{and}\quad\forall y\in A,\ \frac{1}{\verti{x-y}^s}\le\max_{x^\prime\in\omega_x}\frac{1}{\verti{x^\prime-y}^s}\le\sum_{x^\prime\in\omega_x}\frac{1}{\verti{x^\prime-y}^s}.
\end{equation}
We then define a new configuration $\omega_{M_N}\subset A$ of cardinality $M_N$ by
\begin{equation}\label{def_omegam}\
\omega_{M_N}:=\parenth{\omega_N^*\cap A}\medcup\bigcup_{x\in\omega_N^*\cap\parenth{A_\epsilon\setminus A}} \omega_x,
\end{equation}
where by \eqref{card_not_ae}, \eqref{bound_card} and the first part of \eqref{omegax} we have
\begin{equation}\label{bound_pol2}
N-C_0(\epsilon)\le M_N\le \#(\omega_N^*\cap A)+n_{\pi/6,p}\#\left(\omega_N^*\cap(A_\epsilon\setminus A)\right)\le(1+n_{\pi/6,p}C_1(\epsilon))N.
\end{equation}
Then by the bounds \eqref{card_not_ae} and the second part of \eqref{omegax}, we find that
\begin{eqnarray}\label{bound_pol}
\mathcal P_s(A,M_N)&\ge& P_s(A,\omega_{M_N})\nonumber\\
&\ge& P_s(A,\omega_N^*\cap A_\epsilon) \nonumber\\
&\ge& P_s(A,\omega_N^*) - \max_{y\in A}\sum_{x\in \omega_N^*\setminus A_\epsilon}\frac{1}{\verti{x-y}^s}\nonumber\\
&\ge&P_s(A,\omega_N^*) - \frac{C_0(\epsilon)}{\epsilon^s}=: P_s(A,\omega_N^*) - C_2(\epsilon),
\end{eqnarray}
where $C_2(\epsilon)$ depends only on $\epsilon,p,s,A$; in particular $C_2(\epsilon)$ is independent of $N$.

\medskip

Let now $\{N_k\}_{k\in\N}$ be a strictly increasing subsequence that realizes the limit inferior in \eqref{twoineq} and let the sequence $\overline N_k$ be such that, for each $k\in\N$,
\begin{equation}\label{nk_barnk}
\frac{N_k}{1+n_{\pi/6,p}C_1(\epsilon)}\in \left[\overline N_k,\overline N_k+1\right).
\end{equation}
Note that $\overline N_k\to\infty$ as $k\to\infty$. Using the fact that $\mathcal P_s(A,N)$ is increasing in $N$, \eqref{nk_barnk}, \eqref{bound_pol2} and \eqref{bound_pol} give for $s>d$ the bounds
\begin{eqnarray}\label{end_asy}
\frac{\mathcal P_s(A,N_k)}{N_k^{s/d}}&\ge& \frac{\mathcal P_s(A,M_{\overline N_k})}{N_k^{s/d}}\ge\frac{\mathcal P_s^*(A,\overline N_k)-C_2(\epsilon)}{N_k^{s/d}}\nonumber\\
&\ge&\frac{\mathcal P_s^*(A,\overline N_k) - C_2(\epsilon)}{\parenth{(1+n_{\pi/6,p}C_1(\epsilon))(\overline N_k+1)}^{s/d}}\nonumber\\
&=&\frac{1}{\parenth{1+n_{\pi/6,p}C_1(\epsilon)}^{s/d}}\ \frac{\mathcal P_s^*(A,\overline N_k)}{\overline N_k^{s/d}}\ \frac{\mathcal P_s^*(A,\overline N_k)-C_2(\epsilon)}{\mathcal P_s^*(A,\overline N_k)}\left(\frac{\overline N_k}{\overline N_k +1}\right)^{s/d}.
\end{eqnarray}
Due to the fact that $\mathcal P_s^*(A,N)\to\infty$ as $N\to\infty$ by compactness of $A$ and to the fact that $\overline N_k\to\infty$ as $k\to\infty$, using \eqref{bound_card} we find
\begin{equation}\label{limit_taus}
\lim_{k\to\infty}\frac{\mathcal P_s^*(A,\overline N_k)-C_2(\epsilon)}{\mathcal P_s^*(A,\overline N_k)}\left(\frac{\overline N_k}{\overline N_k +1}\right)^{s/d}=1.
\end{equation}
By \eqref{limit_taus} and \eqref{end_asy}, we thus find
\begin{equation}\label{bound_pf}
\liminf_{N\to\infty}
\frac{\mathcal P_s(A,N)}{N^{s/d}}
=\lim_{k\to\infty}\frac{\mathcal P_s(A,N_k)}{N_k^{s/d}}\ge\parenth{1+n_{\pi/6,p}C_1(\epsilon)}^{-s/d}\lim_{k\to\infty}\frac{\mathcal P_s^*(A,\overline N_k)}{\overline N_k^{s/d}}.
\end{equation}
In \eqref{bound_pf} we use the hypothesis that the limit of $\mathcal P_s^*(A,N)/N^{s/d}$ exists as an extended real number. By now taking $\epsilon\to 0$ and using \eqref{bound_card}, the desired second inequality in \eqref{twoineq} follows, and this completes the proof of the theorem as well for the case $s>d$. The remaining range of exponents $p-2<s<d$ is treated as above, with the difference that the function $\tau_{s,d}(N)=N^{s/d}$ is replaced according to the definition \eqref{tau_sd}. We leave the verifications to the reader.

Finally, suppose the hypotheses of case (ii) of Proposition \ref{prop:pointsep}  hold. Theorem~\ref{thm:ohtsuka} and Proposition~\ref{prop:rough_upper_bound}(ii) imply the limit $h_{s,d}^*(A)$ exists and is finite.   
\end{proof}
%
%

%

%%%%%%%%%%%%%%%%%%%%%%%%%%%%%%%%%%%%%%%%%%%%%%%%%%%%%%%%%%%%%%%%%%%%%%%%%%%%%%%%%%
\section{Proof of   Theorem \ref{thm:compact_body}}\label{sec:pf_compact_body}

We begin with a known lemma for constrained polarization.   
\begin{lemma}[\cite{bhrs2016preprint} or \cite{book}]\label{lem:constsub}
Let $1\le d\le p$, $s\ge d$,   and  $A, B\subset\R^p$ be nonempty sets. Then
\begin{equation}\label{ubounda}
\underline h_{s,d} (A\cup B)^{-d/s}\le\underline{h}_{s,d} (A)^{-d/s} + \underline h_{s,d} (B)^{-d/s}.
\end{equation}
\end{lemma}
We remark that the analogous subadditivity result holds with $\underline h_{s,d}$ replaced by $\underline h_{s,d}^*$ in \eqref{ubounda}, but we will not need that result in this paper.  
However,  the two related results given in the next lemma do play an essential role in the proofs of part (ii) of Theorem~\ref{thm:compact_body} and of Theorem~\ref{thm:rectifiable}.  This lemma is proved using similar arguments as in \cite[Sec. 6, 7]{bhrs2016preprint} and \cite[Sec. 14.7]{book} for the one-plate polarization problem $\mathcal P_s$.   We provide a sketch of the proof for the convenience of the reader.
\newcommand{\OL}{\overline}
\newcommand{\W}{\widetilde}

\begin{lemma}\label{lem:up_low_bound}
Let $1\le d\le p$, $s\ge d$,   and  $A, B\subset\R^p$ be nonempty sets. 
\begin{enumerate}
\item[\rm (i)] If the limits $h_{s,d}^*(A\cup B)$ and $h_{s,d}^*(B)$ exist, then
\begin{equation}\label{ubound}
h_{s,d}^*(A\cup B)^{-d/s}\le\overline{h}_{s,d}^*(A)^{-d/s} + h_{s,d}^*(B)^{-d/s}.
\end{equation}
\item[\rm (ii)] If $\op{dist}(A,B)>0$, then
\begin{equation}\label{lbound}
\overline{h}_{s,d}^*(A\cup B)^{-d/s}\ge\overline{h}_{s,d}^*(A)^{-d/s} + \overline{h}_{s,d}^*(B)^{-d/s}\ .
\end{equation}
%
%Moreover, if $\op{dist}(A,B)>0$ and $\{\W\omega_N\}_{N\in\mathcal{N}}$ is any sequence of $N$-point configurations in $\R^p$  such that 
%\begin{equation}\label{uboundseq}
%\lim\limits_{N\to \infty \atop N\in \mathcal N}{\frac {P_s(A\cup B,\W\omega_N)}{N^{s/d}}}=\overline{h}_{s,d}^*(A)^{-d/s} + \overline{h}_{s,d}^*(B)^{-d/s},
%\end{equation}
%then for any $0<\epsilon<\frac12\mathrm{dist}(A,B)$,
%$$
%\lim\limits_{N\to \infty \atop N\in \mathcal N}\frac{\#(\W\omega_N\cap A_\epsilon)}{N}=\frac {\overline h_{s,d}^*(B)^{d/s}}{\overline h_{s,d}^*(A)^{d/s}+\overline h_{s,d}^*(B)^{d/s}} $$ and $$
%\lim\limits_{N\to \infty \atop N\in \mathcal N}\frac{\#(\W\omega_N\cap B_\epsilon)}{N}=\frac {\overline h_{s,d}^*(A)^{d/s}}{\overline h_{s,d}^*(A)^{d/s}+\overline h_{s,d}^*(B)^{d/s}}.$$
\item[\rm (iii)] If $A\subset \R^p$ is such that  $0<\overline{h}_{s,d}^*(A)<\infty$, $\mathcal N\subset \mathbb N$ is any sequence and $\{\W\omega_N\}_{N\in\mathcal{N}}$ are $N$-point configurations in $\R^p$  such that 
\begin{equation}\label{uboundseq}
\lim\limits_{N\to \infty \atop N\in \mathcal N}{\frac {P_s(A,\W\omega_N)}{N^{s/d}}}=\overline{h}_{s,d}^*(A),
\end{equation}
then for any $B\subset A, B\neq \emptyset$ and any $\epsilon>0$,
\begin{equation}\label{NBepsilon}
\liminf\limits_{N\to \infty \atop N\in \mathcal N}\frac{\#(\W\omega_N\cap B_\epsilon)}{N}\ge  \left(\frac{\overline h_{s,d}^*(A)}{\overline h_{s,d}^*(B)}\right)^{d/s} .
\end{equation}
\end{enumerate}
\end{lemma}

We remark that  assertion (iii) above with $B=A$ shows that if $\{\W\omega_N\}_{N\in\mathcal{N}}$ satisfies \eqref{uboundseq}, then any weak-$*$ limit measure of the normalized counting measures 
$\{\nu(\widetilde \omega_N)\}_{N\in \mathcal{N}}$ is supported on the closure of $A$. 

The following elementary result (whose proof is omitted) will be useful in the proof of Lemma~\ref{lem:up_low_bound}.
%%%Book INSERT BELOW%%%
\begin {lemma}\label {p_aux}
Let $s\ge d>0$ and $b,c\geq 0$. Then the function $f(t):=\min\left\{t^{s/d}b,(1-t)^{s/d}c\right\}$ has maximum value $\left(b^{-d/s}+c^{-d/s}\right)^{-s/d}$ on the interval $[0,1]$. If both numbers $b$ and $c$ are positive, the maximum is attained at the unique point
$$
t^\ast:=\frac {c^{d/s}}{b^{d/s}+c^{d/s}}.
$$
\end {lemma}

\begin{proof}[Proof of Lemma~\ref{lem:up_low_bound}]

 We leave it to the reader to verify that the inequalities  in  Lemma~\ref{lem:up_low_bound} hold if any of its terms   are 0 or $\infty$.   Thus, hereafter, we
 assume the terms appearing in these inequalities are positive and finite.

 We first establish the inequality   \eqref{ubound}.
Let   $N,N_1,N_2\in \N$ be such that $N_1+N_2=N$. Let $\omega_{N_1}^*\subset\R^p$ be an $N_1$-point configuration such that $P_s(A,\omega_{N_1}^*)=\mathcal P_s^*(A,N_1)$ and let $\omega_{N_2}^*\subset\R^p$ be an $N_2$-point configuration such that $P_s(B,\omega_{N_2}^*)=\mathcal P_s^*(B,N_2)$.  Then, with $\widetilde\omega_N:=\omega_{N_1}^*\cup \omega_{N_2}^*$, we have
\begin {equation}\label{Psunion}
\begin {split}
\mathcal {P}_s^*(A\cup B,N) &\ge P_s(A\cup B,\widetilde \omega_N)\\
&= \min \left\{P_s(A,\widetilde \omega_N),P_s( B,\widetilde \omega_N)\right\}\\
&\ge \min \left\{P_s(A,  \omega_{N_1}^*),P_s( B, \omega_{N_2}^*)\right\}\\
&=\min \left\{\mathcal{P}^*_s(A,  N_1),\mathcal{P}^*_s(B,  N_2))\right\},
\end {split}
\end {equation}
%and so, 
%\begin{equation}\label{Psunion}
%\hspace{-.14in}\mathcal P_s^*(A\cup B,N)\geq \min\left\{\mathcal P_s^*(A,N_1),\mathcal P_s^*(B,N_2)\right\},
%\end{equation}
and so, 
\begin{equation}\label{Psunion2}
\frac{\mathcal P_s^*(A\cup B,N)}{\tau_{s,d}(N)}\ge  \min \left\{\left(\frac {\tau_{s,d}(N_1)}{\tau_{s,d}(N)}\right) \frac {\mathcal P_s^*(A,N_1)}{\tau_{s,d}(N_1)},\left(\frac {\tau_{s,d}(N_2)}{\tau_{s,d}(N)}\right) \frac {\mathcal P_s^*(B,N_2)}{\tau_{s,d}(N_2)}\right\}.
\end{equation}

Suppose  that both $h_{s,d}^*(A\cup B)$ and $h_{s,d}^*( B)$ exist and define
\begin{equation}\label{alphadef_2}
\alpha:=\frac  { h_{s,d}^*(B)^{d/s}}{\overline h_{s,d}^*(A)^{d/s}+  h_{s,d}^*(B)^{d/s}}.
\end{equation} 
For  $N_1 \in \N$,   let  $N=\lfloor N_1/\alpha\rfloor$ and $N_2=N-N_1$ so that  $N_1+N_2=N$ as above.  
Let $\mathcal{N}_1\subset \N$ be such that 
\[
\lim_{N_1\to \infty \atop N_1\in \mathcal{N}_1} \frac {\mathcal P_s^*(A,N_1)}{\tau_{s,d}(N_1)}=\overline h_{s,d}^*(A)^{d/s}.
\]
Note that $\alpha\in(0,1)$ due to our hypothesis on the terms in the lemma not being $0$ or $\infty$, and in this case we have for $s\ge d$,
\begin{equation}\label{log_fract}
 \lim_{N\to\infty}\frac{\tau_{s,d}(N_1)}{\tau_{s,d}(N)}= \alpha^{s/d} \quad \text{ and }\quad \lim_{N\to\infty}\frac{\tau_{s,d}(N_2)}{\tau_{s,d}(N )} = (1-\alpha)^{s/d}.
\end{equation}
Then, taking the limit as $N_1\to \infty$, $N_1\in \mathcal{N}_1$, of 
\eqref{Psunion2}, using \eqref{log_fract}  and   Lemma~\ref{p_aux} we obtain
\begin {equation}\label {9L}
\begin {split}
 h_{s,d}^*(A\cup B)&\geq \min\left\{\alpha^{s/d}\overline h_{s,d}^*(A),(1-\alpha)^{s/d}  h_{s,d}^*(B)\right\} =\left( \overline{h}_{s,d}^*(A)^{-d/s} +   h_{s,d}^*(B)^{-d/s}\right)^{-s/d},\end {split}
\end {equation}
which proves assertion (i).

%\begin {proof}[Proof of Lemma \ref {sup-additivity}]
 To prove \eqref{lbound}, let $\op{dist}(A,B)>0$ and $\{\W \omega_N\}_{N\in \mathcal N_0}$ be any sequence of $N$-point configurations in $\R^p$ such that
\begin{equation}\label{AcupBopt}
\lim\limits_{N\to \infty \atop N\in \mathcal N_0}{\frac {P_s(A\cup B,\W\omega_N)}{N^{s/d}}}=\OL h_{s,d}^*(A\cup B).
\end{equation}
Then for any $N\in \mathcal N_0$ and $\epsilon>0$,
\begin {equation}\label {p_upper}
\begin {split}
&P_s(A\cup B,\W\omega_N)=\min\left\{P_s(A,\W\omega_N), P_s(B,\W\omega_N)\right\}\\
&\le \min\left\{P_s(A,\W\omega_N\cap A_\epsilon), P_s(B,\W\omega_N\cap B_\epsilon)\right\} +N\epsilon^{-s}\\
&\le \min\left\{\mathcal{P}_s^*(A,N_{A,\epsilon}), \mathcal{P}_s^*(B,N_{B,\epsilon})\right\} +N\epsilon^{-s},
\end {split}
\end {equation}
where
$$
N_{A,\epsilon}:=\# \left(\W\omega_N\cap A_\epsilon\right)\ \ \ {\rm and}\ \ \ N_{B,\epsilon}:=\#(\W\omega_N\cap B_\epsilon).
$$\\

Let $\mathcal N_1\subset \mathcal N_0$ be any infinite subset such that the limit
$$
\alpha:=\lim\limits_{N\to\infty\atop N\in \mathcal N_1}{\frac {N_{A,\epsilon}}{N}}
$$
exists and belongs to $(0,1)$, leaving the cases $\alpha=0$ and $\alpha=1$ to the reader. Then from \eqref {p_upper}, we  have
\begin {equation}\label {p_main}
\begin {split}
&\OL h_{s,d}^*(A\cup B)=\lim\limits_{N\to \infty\atop N\in \mathcal N_1}{\frac {P_s(A\cup B,\W\omega_N)}{\tau_{s,d}(N)}}\\
&\leq \limsup\limits_{N\to \infty\atop N\in \mathcal N_1}\ \min\left\{\left(\frac {\tau_{s,d}(N_{A,\epsilon})}{\tau_{s,d}(N)}\right)\cdot \frac {\mathcal P_s^*(A,N_{A,\epsilon})}{\tau_{s,d}(N_{A,\epsilon})},\,\left(\frac {\tau_{s,d}(N_{B,\epsilon})}{\tau_{s,d}(N)}\right)\cdot \frac {\mathcal P_s^*(B,N_{B,\epsilon})}{\tau_{s,d}(N_{B,\epsilon})}\right\}.
\end {split}
\end {equation}
If $\epsilon<\frac12 \mathrm{dist}(A,B)$ then $A_\epsilon$ and $B_\epsilon$ are disjoint, therefore $N_{A,\epsilon}+N_{B,\epsilon}\le N$. Using this and the fact that $\alpha\in (0,1)$, we obtain
\[
 \limsup\limits_{N\to\infty\atop N\in\mathcal N_1}\frac{\log N_{A,\epsilon}}{\log N}=1\quad\mbox{and}\quad \limsup\limits_{N\to\infty\atop N\in\mathcal N_1}\frac{\log N_{B,\epsilon}}{\log N}\le 1,
\]
and thus for all $s\ge d$ there holds
\[
  \limsup\limits_{N\to\infty\atop N\in\mathcal N_1}\frac{\tau_{s,d}(N_{A,\epsilon})}{\tau_{s,d}(N)}=\alpha^{s/d}\quad\mbox{and}\quad \limsup\limits_{N\to\infty\atop N\in\mathcal N_1}\frac{\tau_{s,d}(N_{B,\epsilon})}{\tau_{s,d}(N)}\le (1-\alpha)^{s/d}.
\]
Plugging the above into \eqref{p_main} we get
\begin {equation}\label {p_100}
\OL h_{s,d}^*(A\cup B)\leq \min\left\{\alpha^{s/d}\OL h_{s,d}^*(A),(1-\alpha)^{s/d}\OL h_{s,d}^*(B)\right\}.
\end {equation}
Appealing to Lemma \ref{p_aux}, it follows that
$$
\OL h_{s,d}^*(A\cup B)\leq \left(\OL h_{s,d}^*(A)^{-d/s}+\OL h_{s,d}^*(B)^{-d/s}\right)^{-s/d},
$$
which proves assertion (ii).   
%Furthermore, equality holds in \eqref{p_100} only if both $\alpha=
%\frac  { \overline h_{s,d}^*(B)^{d/s}}{\overline h_{s,d}^*(A)^{d/s}+ \overline h_{s,d}^*(B)^{d/s}}$ and $\lim\limits_{N\to \infty\atop N\in \mathcal N_1}\frac{N_{B,\epsilon}}{N}=1-\alpha,
%$
%which proves the final assertion of the lemma. 

Finally, suppose $B\subset A$ and $\{\W\omega_N\}_{N\in\mathcal{N}}$ is such that \eqref{uboundseq} holds.
The inequality \eqref {p_upper} with $B\subset A$ gives
\begin {equation}\label {p_upper2}
\frac{P_s(A,\W\omega_N)}{\tau_{s,d}(N)}\le  \frac{\mathcal{P}_s^*(B,N_{B,\epsilon})}{\tau_{s,d}(N_{B,\epsilon})}\frac{\tau_{s,d}(N_{B,\epsilon})}{\tau_{s,d}(N)} +\frac{N\epsilon^{-s}}{\tau_{s,d}(N)},
\end {equation}
Taking the limit inferior as $N\to \infty$ with $N\in \mathcal{N}$
\begin {equation}\label {p_upper3}
\overline{h}_{s,p}^*(A)=\liminf\limits_{N\to \infty \atop N\in \mathcal N}\frac{P_s(A,\W\omega_N)}{\tau_{s,d}(N)}\le  \liminf\limits_{N\to \infty \atop N\in \mathcal N}\frac{\mathcal{P}_s^*(B,N_{B,\epsilon})}{\tau_{s,d}(N_{B,\epsilon})}\frac{\tau_{s,d}(N_{B,\epsilon})}{\tau_{s,d}(N)} \le \overline{h}_{s,p}^*(B)
\left(\liminf\limits_{N\to \infty \atop N\in \mathcal N} \frac{ N_{B,\epsilon}}{ N}\right)^{s/d},
\end {equation}
which proves assertion (iii).   
\end {proof}

%
%
%%%%%%%%%%%%%%%%%%%%%%%%%%%%%%%%%%%%%%%%%%%%%%%%%%%%%%%%%%%%%%%%%%%%%%
\noindent{\em Completion of Proof of  Theorem \ref{thm:compact_body} (see Appendix \ref{app-corrigendum} as well).}\\
%%%%%%%%%%%%%%%%%%%%%%%%%%%%%%%%%%%%%%%%%%%%%%%%%%%%%%%%%%%%%%%%%%%%%%
%
As mentioned in the remarks following the statement of  Theorem \ref{thm:compact_body}, it is proved in \cite{bhrs2016preprint}  that for $s>p$  the second equality in \eqref{asymptotics} holds     for compact sets in $\R^p$ with  boundary of $\mathcal L_p$ measure zero and is known from \cite{2014borodachovbosuwan} that this equality holds for arbitrary compact sets when $s=p$.   We will make use of these facts in our proof. 

Let $A\subset\R^p$ be compact. We will separately establish for $s\ge p$ the following two inequalities:
\begin{equation}\label{underh_p}
\underline{h}_{s,p}^*(A)\ge\frac{\sigma_{s,p}}{\mathcal L_p(A)^{s/p}}
\end{equation}
and
\begin{equation}\label{overh_p}
\overline{h}_{s,p}^*(A)\le\frac{\sigma_{s,p}}{\mathcal L_p(A)^{s/p}}.
\end{equation}
To prove \eqref{underh_p}, let $\epsilon>0$ and select a set $G\subset \R^p$ such that $\mathcal L_p(\partial G)=0$, $A\subset G$ and $\mathcal L_p(G\setminus A)<\epsilon$. Then using \eqref{inclusion}, we find
\begin{equation}
\underline h_{s,p}^*(A)\ge \underline h_{s,p}^*(G)\ge h_{s,p}(G)=\frac{\sigma_{s,p}}{\mathcal L_p(G)^{s/p}}\ge \frac{\sigma_{s,p}}{\left(\mathcal L_p(A)+\epsilon\right)^{s/p}}.
\end{equation}
Letting $\epsilon\downarrow0$, we obtain \eqref{underh_p} for $\mathcal L_p(A)>0$ and also that $h_{s,p}^*(A)=h_{s,p}(A)=\infty$ if $\mathcal L_p(A)=0$ which establishes \eqref{asymptotics} when    $\mathcal L_p(A)=0$. 

Hereafter we assume  $\mathcal L_p(A)>0$.   To prove \eqref{overh_p}, let
\begin{equation}\label{densityset}
A^*:=\left\{x\in A:\ \limsup_{r\to 0^+}\frac{\mathcal L_p(\overline{B(x,r)}\cap A)}{\mathcal L_p(\overline{B(x,r)})}=1\right\}.
\end{equation}
Then by the Lebesgue density theorem there holds $\mathcal L_p(A\setminus A^*)=0$. By an iterative covering argument using Besicovitch's covering theorem, we can find a finite collection of disjoint closed balls $B_i, i\in\{1,\ldots,n\}$ of radii $r_i\in (0,1)$, such that
\begin{subequations}\label{properties_bi}
\begin{equation}\label{bi_1}
\forall i\in\{1,\ldots,n\},\quad \frac{\mathcal L_p(A\cap B_i)}{\mathcal L_p(B_i)}\ge 1-\epsilon,
\end{equation}
\begin{equation}\label{bi_2}
\mathcal L_p\left(\bigcup_{i=1}^n(A\cap B_i)\right) = \sum_{i=1}^n\mathcal L_p(A\cap B_i)\ge (1-\epsilon)\mathcal L_p(A).
\end{equation}
\end{subequations}
Now \eqref{inclusion} together with \eqref{lbound} of Lemma \ref{lem:up_low_bound} gives
\begin{equation}\label{bound_a_1}
\overline{h}_{s,p}^*(A)\le \overline h_{s,p}^*\left(\bigcup_{i=1}^n (A\cap B_i)\right)\le\left(\sum_{i=1}^n\overline h_{s,p}^*(A\cap B_i)^{-p/s}\right)^{-s/p}.
\end{equation}
Due to \eqref{bi_1}, and to the regularity of the Radon measure $\mathcal L_p$, there exist sets $G_i\subset B_i$ such that $\mathcal L_p(\partial G_i)=0$ and $B_i\setminus A\subset G_i$ and $\mathcal L_p(G_i)<2\epsilon\mathcal L_p(B_i)$. Now we use \eqref{ubound} of Lemma \ref{lem:up_low_bound}, with the choices $d\mapsto p$, $A\mapsto A\cap B_i$ and $B\mapsto G_i$, obtaining
\begin{eqnarray}\label{bound_bi}
\overline{h}_{s,p}^*(A\cap B_i)^{-p/s}&\stackrel{\eqref{ubound}}{\ge}& h_{s,p}(B_i)^{-p/s}- h_{s,p}(G_i)^{-p/s}\nonumber\\
&=&\left(\sigma_{s,p}\right)^{-p/s}\left(\mathcal L_p(B_i)-\mathcal L_p(G_i)\right)\ge(1-2\epsilon)\left(\sigma_{s,p}\right)^{-p/s}\mathcal L_p(B_i).
\end{eqnarray}
By \eqref{bi_2}, \eqref{bound_a_1} and \eqref{bound_bi} we obtain
\begin{eqnarray}\label{finalbound}
\overline{h}_{s,p}^*(A)&\le& \frac{\sigma_{s,p}}{(1-2\epsilon)^{s/p}}\left(\sum_{i=1}^N\mathcal L_p(B_i)\right)^{-s/p}\le\frac{\sigma_{s,p}}{(1-2\epsilon)^{s/p}}\left(\sum_{i=1}^N\mathcal L_p(A\cap B_i)\right)^{-s/p}\nonumber\\
&\le&\frac{\sigma_{s,p}}{((1-2\epsilon)(1-\epsilon))^{-s/p}}\mathcal L_p(A)^{-s/p}.
\end{eqnarray}
By taking the limit $\epsilon\downarrow0$ in \eqref{finalbound} we obtain \eqref{overh_p}, as desired.  Combining  \eqref{underh_p} and  \eqref{overh_p}
proves $$
h_{s,p}^*(A)=\frac{\sigma_{s,p}}{\mathcal L_p(A)^{s/p}},
$$
for $s\ge p$.
The fact that the same equality holds for the constrained case  follows from Theorem~\ref{thm:compact_body_old} in the case $s>p$ since $\mathcal L_p(A)>0$.  For $s=p$ the constrained equality is proved in \cite{2014borodachovbosuwan}.\hfill$\qed$
%
%
%
%
%
%
%
%%%%%%%%%%%%%%%%%%%%%%%%%%%%%%%%%%%%%%%%%%%%%%%%%%%%%%%%%%%%%%%%%%%%%%
\section{A general lower bound via Minkowski content}\label{sec:mink}
%%%%%%%%%%%%%%%%%%%%%%%%%%%%%%%%%%%%%%%%%%%%%%%%%%%%%%%%%%%%%%%%%%%%%%%
%
%
The main result of this section, Proposition \ref{prop:mink_bound}, is the analogue for the case of polarization problems of the rough bound \cite[Lemma 8]{hsw2012} in the setting of energy minimization problems.

We start by recalling the definition of Minkowski content:
\begin{definition}\label{def:mink}
The \emph{upper} and \emph{lower Minkowski contents} of $A$, denoted respectively by $\overline{\mathcal M}_d(A)$, $\underline{\mathcal M}_d(A)$ are respectively defined as
\[
\overline{\mathcal M}_d(A):=\limsup_{r\downarrow 0}\frac{\mathcal L^p(A_r)}{\beta_{p-d}r^{p-d}},\quad \underline{\mathcal M}_d(A):=\liminf_{r\downarrow 0}\frac{\mathcal L^p(A_r)}{\beta_{p-d}r^{p-d}},
\]
where $A_r$ is as defined in \eqref{fatten_a} and $\beta_k>0$ is for $k\in\mathbb N, k\ge 1$ the volume of the $k$-dimensional unit ball
\begin{equation}\label{volkball}
\beta_k:= \frac{\pi^{k/2}}{\Gamma\left(\frac k2 + 1\right)}.
\end{equation}
If $\overline{\mathcal M}_d(A)=\underline{\mathcal M}_d(A)$, their common value is called the \emph{Minkowski content} of $A$, and denoted by $\mathcal M_d(A)$.
\end{definition}
The next proposition is essentially a generalization of \cite[Lemma 8.2]{bhrs2016preprint} and will be used in the proof of Theorem~\ref{thm:rectifiable}
given in Section~\ref{pfThmRect}.
\begin{proposition}\label{prop:mink_bound}
Let $p\ge d\ge 1$ be natural numbers and let $s>d$. Then there exists a constant $\mathcal C_{p,d,s}$ depending only on $p,d,s$ such that the following holds. Let $A\subset \R^p$ be a set such that $\overline{\mathcal M}_d(A)<\infty$. Then
\begin{equation}\label{mink_bound}
  \underline h_{s,d}(A) \ge \frac{\mathcal C_{p,d,s}}{\overline{\mathcal M}_d(A)^{s/d}}.
\end{equation}
\end{proposition}
The above proposition follows from Lemmas \ref{lem:mink_covers} and \ref{lem:cover_pol} below. Lemma \ref{lem:mink_covers} says that Minkowski content controls best-covering at all scales, and this will enable us to bound the polarization constant from below by covering constants in Lemma \ref{lem:cover_pol}.
\begin{lemma}\label{lem:mink_covers}
Let $p\ge d\ge 1$ be natural numbers. Then there exists a constant $\mathcal C_{p,d}>0$ depending only on $p,d$ such that for any set $A\subset\R^p$ with $0<\overline{\mathcal M}_d(A)<\infty$, for all $r>0$ sufficiently small there exists $W_r\subset\R^p$ such that
\[
A\subset \bigcup_{x\in W_r}B(x,r)
\]
and
\begin{equation}\label{mink_covers}
\# W_r\le \mathcal C_{p,d}\frac{\overline{\mathcal M}_d(A)}{r^d}.
\end{equation}
If $A\subset\R^p$ is such that $\overline{\mathcal M}_d(A)=0$, then for any $\epsilon>0$ there is some $r_0>0$ such that for all $r\in(0,r_0)$,
\begin{equation}\label{mink_covers0} 
\# W_r\le  \frac{\epsilon}{r^d}.
\end{equation}
\end{lemma}
\begin{lemma}\label{lem:cover_pol}
Let $p,d,s,A$ be fixed as in Proposition~\ref{prop:mink_bound}. If there exists $C>0$ such that for every sufficiently small $r>0$ there exists a covering of $A$ by balls of radius $r$ of cardinality at most $C/r^d$, then
\begin{equation}\label{cover_pol}
\underline h_{s,d} (A)\ge \left(\widetilde{\mathcal C}_{p,d}C\right)^{-s/d},
\end{equation}
where $\widetilde{\mathcal C}_{p,d}>0$ is a constant depending only on $p,d$.
\end{lemma}
\begin{proof}[Proof of Proposition~\ref{prop:mink_bound}:]
  By Lemma~\ref{lem:mink_covers}, the hypotheses of Lemma~\ref{lem:cover_pol} hold for the choice $C=\mathcal C_{p,d}\overline{\mathcal M}_d(A)$ when $\overline{\mathcal M}_d(A)>0$ and for any $C>0$ when $\overline{\mathcal M}_d(A)=0$, where $\mathcal C_{p,d}>0$ is the constant from Lemma~\ref{lem:mink_covers}. Then the inequality \eqref{cover_pol} directly gives \eqref{mink_bound} for the choice
\begin{equation}
\mathcal C_{p,d,s}:=\left(\widetilde{\mathcal C}_{p,d}\mathcal C_{p,d}\right)^{-s/d}.
\end{equation}
\end{proof}
We now provide the proofs for the above lemmas.
\begin{proof}[Proof of Lemma~\ref{lem:mink_covers}:]
Let $\epsilon>0$ be as in the statement of the lemma and let $\tilde{\epsilon}>0$ be a constant which will be fixed below depending only on $A, p, d, \epsilon$. If $r>0$ is small enough (depending on $A$ and $\tilde\epsilon$), there holds
\begin{equation}\label{neigh_vol_bd}
\frac{\mathcal L_p(A_{2r})}{2^{p-d}\beta_{p-d}r^{p-d}}\le  \overline{\mathcal M}_d(A)+\tilde{\epsilon}.
\end{equation}
There exists $N_{\mathrm{Cov},p}\in\mathbb N$, depending only on $p$ such that for any $r>0$, there are points $x_1,\ldots, x_{N_{\mathrm{Cov},p}}\in[0,2r)^p$ such that the open $r$-balls with centers in $\{x_j+y:\ y\in(2r\mathbb Z)^p, 1\le j\le N_{\mathrm{Cov},p}\}$ cover $\mathbb R^p$. 
Then, for each $j\in\{1,\ldots,N_{\mathrm{Cov},p}\}$, the $r$-balls with centers in $W_r^j:=A_r\cap (x_j+(2r\mathbb Z)^p)$ are disjoint and the set $W_r:=\bigcup_{j=1}^{N_{\mathrm{Cov},p}}W_r^j$ satisfies
\begin{equation}\label{balls_in_nbd_A}
A\subset\bigcup_{x\in W_r}B(x,r)\subset A_{2 r}.
\end{equation}
Due to \eqref{balls_in_nbd_A} and \eqref{neigh_vol_bd}, for $j=1,\ldots,N_{\mathrm{Cov},p}$, we have
\[
\#\left( W_r^j\right)\cdot \beta_pr^p=\left|\bigcup_{x\in W_r^j}B(x,r)\right|\le\left|A_{2r}\right|\le2^{p-d}\beta_{p-d}r^{p-d}(\overline{\mathcal M}_d(A)+\tilde{\epsilon}),
\]
where $|\cdot|$ denotes the $\mathcal L_p$-measure of a set. By summing over $j$ we obtain
\[
\# W_r\le N_{\mathrm{Cov},p}\frac{\beta_{p-d}2^{p-d}}{\beta_p}\frac{\overline{\mathcal M}_d(A)+\tilde{\epsilon}}{r^d}.\]
If $\overline{\mathcal M}_d(A)>0$, then choosing $\tilde{\epsilon}=\overline{\mathcal M}_d(A)$ shows that \eqref{mink_covers} holds with
$ \mathcal C_{p,d}:=\frac{\beta_{p-d}2^{p-d+1}}{\beta_p}$.  If $\overline{\mathcal M}_d(A)=0$, then choosing $\tilde{\epsilon}=\epsilon/\mathcal C_{p,d}$ proves \eqref{mink_covers0}.
\end{proof}
\begin{proof}[Proof of Lemma~\ref{lem:cover_pol}:]
Let $\mathcal{B}=\{B(x_i,r):i=1,\ldots,N_r\}$ be a minimum-cardinality covering of $A$ by $r$-balls and for each $i=1,\ldots,M$, choose $\tilde x_i\in A\cap B(x_i,r)$. Setting $W_r=\{\tilde x_i: i=1,\ldots, N_r\}$, we have %
\begin{equation}\label{cardomega}
N_r=\# W_r\le \frac{C}{r^d},
\end{equation}
due to the hypothesis of the lemma. Since for each point in $A$ there exists a point in $W_r$ at distance at most $2r$ from $A$, we have
\begin{equation}\label{pol_bound1}
\mathcal P_s(A,N_r)\ge P_s(A, W_r)\ge (2r)^{-s}.
\end{equation}
We set $N_{\mathrm{Cov},p}$ to be the minimum number of balls of radius $1$ in $\mathbb R^p$ required to cover a ball of radius $2$. Then we have, for all $r>0$,
\[
N_r\le N_{r/2}\le N_{\mathrm{Cov},p}N_r.
\]
Thus by \eqref{cardomega}, for fixed $N$ there exists $r=r(N)>0$ such that
\begin{equation}\label{bound_nr1}
N_r\le N< N_{r/2}\le N_{\mathrm{Cov},p}N_r.
\end{equation}
Then we have
\[
\mathcal P_s(A,N)
\stackrel{\text{\eqref{bound_nr1}}}{\ge}\mathcal P_s(A,N_r)
\stackrel{\text{\eqref{pol_bound1}}}{\ge} (2r)^{-s}
\stackrel{\text{\eqref{cardomega}}}{\ge} \left(\frac{N_r}{2^dC}\right)^{s/d}
\stackrel{\text{\eqref{bound_nr1}}}{\ge}\frac{N^{s/d}}{(2^dN_{\mathrm{Cov},p}C)^{s/d}}
\stackrel{\text{\eqref{tau_sd}}}{=}\frac{\tau_{s,d}(N)}{(2^dN_{\mathrm{Cov},p}C)^{s/d}},
\]
where we have also used the fact that the polarization value is increasing in $N$ for the first inequality. Now by reordering the terms and by passing to the limit in $N$ along a subsequence that realizes the value of $\underline h_{s,d}(A)$, the bound \eqref{cover_pol} follows if we set $\widetilde{\mathcal C}_{p,d}:=2^dN_{\mathrm{Cov},p}$.
\end{proof}
\begin{rmk}\label{rmk:mink_1plate}   Of course, Proposition~\ref{prop:mink_bound} also provides a lower bound for $\underline{h}_{s,d}^*(A)$ since this quantity is at least as large as its constrained analog. 
 Thanks to Proposition~\ref{prop:mink_bound}  the asymptotic lower bounds in \cite{bhrs2016preprint}  now follow without needing to appeal to  the energy   results of \cite{bhs2008}. \end{rmk}
%
%
%
%
%%%%%%%%%%%%%%%%%%%%%%%%%%%%%%%%%%%%%%%%%%%%%%%%%%%%%%%%%%%%%%%%%%%%%%
\section{Proof of Theorem \ref{thm:rectifiable}}\label{sec:pf_rectifiable}
%%%%%%%%%%%%%%%%%%%%%%%%%%%%%%%%%%%%%%%%%%%%%%%%%%%%%%%%%%%%%%%%%%%%%%%
%
%
%%%%%%%%%%%%%%%%%%%%%%%%%%%%%%%%%%%%%%%%%%%%%%%%%%%%%%%%%%%%%%%%%%%%%%
\subsection{Some geometric measure theory tools}
%%%%%%%%%%%%%%%%%%%%%%%%%%%%%%%%%%%%%%%%%%%%%%%%%%%%%%%%%%%%%%%%%%%%%%%
%
%
We first quantify the increase of interpoint distances under projection on $L$-Lipschitz graphs:
\begin{lemma}\label{lem:proj}
For $p\in \N$ and $\epsilon>0$, let $\mathcal G$   be  a $d$-dimensional graph in $\R^p$ of an $\epsilon$-Lipschitz function $\psi:H\to H^\perp$ over a $d$-dimensional subspace $H\subset\mathbb R^p$  having orthogonal complement $H^\perp$; i.e, ${\mathcal G}:=\{h+\psi(h):\, h \in H\}$. If $\pi_H:\mathbb R^p\to H$ is the orthogonal projection onto $H$ and $C_\epsilon:=\sqrt{1+\epsilon^2}$, then for any $x\in\mathbb R^p$ and any $y\in {\mathcal G}$,
\begin{equation}\label{proj_error_L}
 |\pi_H(x)+\psi(\pi_H(x))-y|\le C_\epsilon  |x-y|.
\end{equation}
%
%In particular, we may take $C_L$ such that
%%
%\begin{equation}\label{c_tozero}
%\lim_{L\to 0}C_L=1.
%\end{equation}
%
\end{lemma}
\begin{proof}
For  $x\in \R^p$, let $x':=\pi_H(x)$, $x'':=x-x'$, and note that $x''\in H^\perp$.   If $y\in {\mathcal G}$, then $y''=\psi(y')$ and we have
\begin{equation*}
|(x'+\psi(x'))-y|^2=|x'-y'|^2+|\psi(x')-\psi(y')|^2 \le (1+\epsilon^2)|x'-y'|^2\le  (1+\epsilon^2)|x-y|^2,
\end{equation*}
which proves the lemma.
\end{proof}

  Lemma~\ref{lem:proj} directly implies the following rough bound for unconstrained polarization for Lipschitz graphs:
\begin{corollary}\label{coro:e_lip_graphs}
Under the hypotheses of Lemma~\ref{lem:proj}, if $\widetilde{K}$ is a compact subset  of $\mathcal G$ and $N\in\mathbb N$, then
\begin{equation}\label{pol_lip_graphs_1}
\mathcal P_s(\widetilde{K},{\mathcal G},N) \le \mathcal P_s^*(\widetilde{K},N)\le C_\epsilon^s \mathcal P_s(\widetilde{K},{\mathcal G},N).
\end{equation}
\end{corollary}
We also state the following simple deformation result without proof.
\begin{lemma}\label{lem:lip_area}
If   $\Phi:\mathbb R^m\to \mathbb R^n$ is an $(1+\epsilon)$-biLipschitz map for some $\epsilon>0$,  $K\subset \mathbb R^m$   a compact set,  $\omega\subset \mathbb R^m$   a finite set,
and $s>0$, then
\begin{subequations}\label{lip_area}
\begin{equation}\label{lip_area_1}
(1+\epsilon)^{-s} P_s(K,\omega) \le P_s(\Phi(K),\Phi(\omega))\le (1+\epsilon)^{s} P_s(K,\omega),
\end{equation}
and
\begin{equation}\label{lip_area_2}
 (1+\epsilon)^{-d}\mathcal H_d( K)\le  \mathcal H_d(\Phi(K))\le (1+\epsilon)^d\mathcal H_d( K).
\end{equation}
\end{subequations}
\end{lemma}
%
%
%
%%%%%%%%%%%%%%%%%%%%%%%%%%%%%%%%%%%%%%%%%%%%%%%%%%%%%%%%%%%%%%%%%%%%%%
\subsection{Proof of Theorem~\ref{thm:rectifiable}}\label{pfThmRect}
%%%%%%%%%%%%%%%%%%%%%%%%%%%%%%%%%%%%%%%%%%%%%%%%%%%%%%%%%%%%%%%%%%%%%%%
%
%
We recall our definition of $A$ being strongly $(\mathcal H_d,d)$-rectifiable: for any $\epsilon>0$ and for $k\in \mathbb N$ large enough depending on $\epsilon$ we may write $A$ as
\begin{equation}\label{decomp_lip1}
A=R_\epsilon\cup\bigcup_{j=1}^k \widetilde K_j,\quad\mbox{where}\quad\left\{\begin{array}{l}\widetilde K_j\subset\mathbb R^p\mbox{ are compact and disjoint,} \\ \widetilde K_j\mbox{ are contained in }d\mbox{-dimensional }\epsilon\mbox{-Lipschitz graphs,}\\ \overline{\mathcal M}_d(R_\epsilon)<\epsilon.\end{array}\right.
\end{equation}
More explicitly, for each $j=1,\ldots,k$ there is a $d$-dimensional subspace $H_j\subset \mathbb R^p$ and an $\epsilon$-Lipschitz map $\psi_j:H_j\to H_j^\perp$ such that $\widetilde K_j$ is included in the graph $\mathcal G_j$ of $\psi_j$.  For each $j=1,\ldots,k$, let $\iota_j:\mathbb R^d\to H_j$ be an isometry.  As mentioned just before Definition~\ref{sdrectDef}, the mapping $ \varphi_j:\R^d\to\mathcal G_j$ defined by $ \varphi_j(x):=\iota_j(x)+\psi_j(\iota_j(x))$  is then $(1+\epsilon)$-biLipschitz for every $j$ with $\widetilde K_j=\varphi_j(K_j)$, where $K_j\subset \mathbb R^d$ is compact.
%
% By Kirszbraun's theorem (see, e.g., \cite{federer}), we may further assume that the mappings $\phi_j$ are $(1+\epsilon)$-Lipschitz on $\R^d$.

\medskip

Let $A\subset\R^d$ be strongly $(\mathcal H_d,d)$-rectifiable.   We shall prove separately the inequalities
\begin{subequations}\label{cptbody_pf}
\begin{equation}\label{underh_d}
 \underline{h}_{s,d}(A)\ge\frac{\sigma_{s,d}}{\mathcal H_d(A)^{s/d}},
\end{equation}
\begin{equation}\label{overh_d}
\overline{h}_{s,d}^*(A)\le\frac{\sigma_{s,d}}{\mathcal H_d(A)^{s/d}},
\end{equation}
\end{subequations}
which, since $\underline{h}_{s,d}(A)\le \underline{h}_{s,d}^*(A)$ and $\overline{h}_{s,d}(A)\le  \overline{h}_{s,d}^*(A)$, imply  \eqref{thm_rect_eq}.

%We have $\sigma_{s,d}(A)=\sigma_{s,d}(A)$ by Theorem \ref{thm:compact_body_old} and 
%\[\lim_{N\to\infty}\frac{\mathcal P_s^*(A,N)}{\tau_{s,d}(N)}=\lim_{N\to\infty}\frac{\mathcal P_s(A,N)}{\tau_{s,d}(N)}
%\] 
%holds for $s>p-2$ by Theorem \ref{thm:bound_K}, thus the bounds \eqref{cptbody_pf} allow to obtain \eqref{thm_rect_eq}.
%
\medskip

We shall show \eqref{underh_d} using the decomposition \eqref{decomp_lip1}.  By \eqref{mink_bound} of Proposition \ref{prop:mink_bound}, we have
\begin{equation}\label{lb_error}
\underline h_{s,d} (R_\epsilon)\ge \frac{\mathcal C_{p,d,s}}{\overline{\mathcal M}_d(R_\epsilon)^{s/d}}\ge \frac{\mathcal C_{p,d,s}}{\epsilon^{s/d}}.
\end{equation}
By using \eqref{ubounda} from Lemma \ref{lem:up_low_bound} and the $(1+\epsilon)$-biLipschitz parameterizations of $\widetilde K_j=\varphi_j(K_j)$, we find that
\begin{eqnarray}\label{chain_lb}
\underline h_{s,d}(A)^{-d/s}&\stackrel{\text{\eqref{decomp_lip1}, \eqref{ubounda}}}{\le}&\underline h_{s,d}(R_\epsilon)^{-d/s}+\sum_{j=1}^k\underline h_{s,d}(\varphi_j(K_j))^{-d/s}\nonumber\\
&\stackrel{\text{\eqref{lb_error}, \eqref{ineq_pol}}}{\le}&\mathcal C_{p,d,s}^{-d/s}\epsilon +\sum_{j=1}^k\left[\liminf_{N\to\infty} \frac{\mathcal P_s(\varphi_j(K_j),N)}{\tau_{s,d}(N)}\right]^{-d/s}\nonumber\\
&\stackrel{\text{\eqref{lip_area_1}}}{\le}&\mathcal C_{p,d,s}^{-d/s}\epsilon +(1+\epsilon)^d\sum_{j=1}^k\left[\liminf_{N\to\infty}\frac{\mathcal P_s(K_j,N)}{\tau_{s,d}(N)}\right]^{-d/s}\nonumber\\
&\le&\mathcal C_{p,d,s}^{-d/s}\epsilon +(1+\epsilon)^d\sum_{j=1}^k\underline h_{s,d}(K_j)^{-d/s}\nonumber\\
&\stackrel{\text{Thm.\,\ref{thm:compact_body} (ii)}}{=}&\mathcal C_{p,d,s}^{-d/s}\epsilon + (1+\epsilon)^d(\sigma_{s,d})^{-d/s}\sum_{j=1}^k\mathcal H_d(K_j)\nonumber\\
&\stackrel{\text{\eqref{lip_area_2}}}{\le}&\mathcal C_{p,d,s}^{-d/s}\epsilon + (1+\epsilon)^{2d}(\sigma_{s,d})^{-d/s}\sum_{j=1}^k\mathcal H_d(\varphi_j(K_j)).
\end{eqnarray}
 Since $\sum_{j=1}^k\mathcal H_d(\varphi_j(K_j))\le \mathcal{H}_d(A)$, taking the limit as $\epsilon\downarrow 0$ in \eqref{chain_lb} yields the bound \eqref{underh_d}.  Note that this shows $\underline h_{s,d}(A)=+\infty$ if $\mathcal H_d(A)=0$. 
 
% A similar argument as above   using the subadditivity of $\underline{h}_{s,d}(\cdot)^{-d/s}$ (see \cite[Lemma 7.1]{bhrs2016preprint}) shows 
% $\underline{h}_{s,d}(A)\ge \frac{\sigma_{s,d}}{\mathcal H_d(A)^{s/d}}$.   

To prove \eqref{overh_d}, we use the decomposition \eqref{decomp_lip1}, the bound \eqref{inclusion} and the bound \eqref{lbound} of Lemma~\ref{lem:up_low_bound}, and we obtain
\begin{eqnarray}\label{ubound_d}
\overline h_{s,d}^*(A)^{-d/s}&\stackrel{\text{\eqref{inclusion}}}{\ge}& \left[\overline h_{s,d}^*\left(\bigcup_{j=1}^k\varphi_j(K_j)\right)\right]^{-d/s}\stackrel{\text{\eqref{lbound}}}{\ge} \sum_{j=1}^k\overline h_{s,d}^*(\varphi_j(K_j))^{-d/s}\nonumber\\
&\stackrel{\text{\eqref{pol_lip_graphs_1}}}{\ge}&C_\epsilon^{-d}\sum_{j=1}^k\left[\limsup_{N\to\infty} \frac{\mathcal P_s(\varphi_j(K_j),\varphi_j(\R^d),N)}{\tau_{s,d}(N)}\right]^{-d/s} \nonumber\\
&\stackrel{\text{\eqref{lip_area_1}}}{\ge}&C_\epsilon^{-d}(1+\epsilon)^{-d}\sum_{j=1}^k\overline h_{s,d}^*(K_j)^{-d/s}\stackrel{\text{Thm.\,\ref{thm:compact_body} (ii)}}{\ge}(\sigma_{s,d})^{-d/s} C_\epsilon^{-d}(1+\epsilon)^{-d}\sum_{j=1}^k\mathcal H_d(K_j)\nonumber\\
&\stackrel{\text{\eqref{lip_area_2}}}{\ge}&(\sigma_{s,d})^{-d/s} C_\epsilon^{-d}(1+\epsilon)^{-2d}\sum_{j=1}^k\mathcal H_d(\varphi(K_j)) \nonumber  \\&\stackrel{\text{\eqref{decomp_lip1}}}{\ge}&(\sigma_{s,d})^{-d/s} C_\epsilon^{-d}(1+\epsilon)^{-2d}(\mathcal H_d(A)-\epsilon).
\end{eqnarray}
Since   $\lim_{\epsilon\downarrow 0}C_\epsilon=1$, taking   $\epsilon\downarrow 0$ in \eqref{ubound_d} we find the desired bound \eqref{overh_d}.

Finally, suppose that $\mathcal H_d(A)>0$  and $ \{\omega_N\}_{N\in \mathcal{N}}$ satisfies
\begin{equation}\label{Aextseq}
\lim\limits_{N\to \infty \atop N\in \mathcal N}{\frac {P_s(A,\W\omega_N)}{N^{s/d}}}= {h}_{s,d}^*(A)=\sigma_{s,d}(\mathcal{H}_d(A))^{-s/d}.
\end{equation}
For $N\in \mathcal{N}$, let  $\nu_N:=\nu(\widetilde \omega_N)$ denote the normalized counting measure associated with $\widetilde \omega_N$ and let $\mu_A$ denote the measure $\frac{\mathcal H_d|_A}{\mathcal H_d(A)}$.   
Let  $G\subset \R^p$ be open.  For $\delta>0$, let $B$  be  a closed subset of  $A\cap G$ such that $\mu_A(B)\ge (1-\delta)\mu_A(G)$.  Since $A$ is compact, there is some $\epsilon>0$ such that $B_\epsilon\subset G$.  Since $A$ is strongly $(\mathcal H_d,d)$-rectifiable, $B$ is also strongly $(\mathcal H_d,d)$-rectifiable and so 
${h}_{s,d}^*(B)=\sigma_{s,d}(\mathcal{H}_d(B))^{-s/d}$.  Using   \eqref{NBepsilon} gives
$$
\liminf\limits_{N\to \infty \atop N\in \mathcal N}\nu_N(G)\ge\liminf\limits_{N\to \infty \atop N\in \mathcal N} \frac{\#(\widetilde\omega_N\cap B_\epsilon)}{N}\ge 
   \left(\frac{  h_{s,d}^*(A)}{  h_{s,d}^*(B)}\right)^{d/s} =\mu_A(B)\ge (1-\delta)\mu_A(G),
$$ 
and since $\delta>0$ is arbitrary, 
$$
\liminf\limits_{N\to \infty \atop N\in \mathcal N}\nu_N(G)\ge\mu_A(G).
$$ 
The Portmanteau Theorem (e.g., see \cite{Bill1999})  then implies that $\nu_N$ converges in the weak* topology to $\mu_A$.   \\
\hfill$\qed$

\medskip

We conclude this section by showing that compact subsets of $C^1$-embedded manifolds are strongly $(\mathcal H_d,d)$-rectifiable:
\begin{lemma}\label{C1case}
 Let $M\subset \mathbb R^p$ be a $C^1$-embedded submanifold of dimension $d$ and let $A\subset M$ be a compact set. Then $A$ is strongly $(\mathcal H_d,d)$-rectifiable.
\end{lemma}
\begin{proof}
As $M$ is a $C^1$-embedded submanifold, for each $\epsilon>0$ there exists a radius $\rho=\rho(\epsilon)>0$ such that for every $x\in A$ the intersection $M\cap\overline{B(x,\rho)}$ is an $\epsilon$-Lipschitz graph over the tangent subspace $T_xM$ of $M$ at $x$. As $A$ is compact, we can find a cover by balls $B (x_i,\rho)$, $i=1,\ldots,k_0$ with $x_i\in M$. We will introduce a small parameter $\epsilon_1\in(0,1)$ to be appropriately restricted later. We define the sets
\[
\widetilde K_1:=A\cap\overline{B(x_1,(1-\epsilon_1)\rho)}\quad\mbox{and}\quad\widetilde K_{k+1}:=\left(A\cap\overline{B(x_{k+1},(1-\epsilon_1)\rho)}\right)\setminus\bigcup_{j=1}^kB(x_j,\rho)\quad\mbox{for }1\le k\le k_0-1.
\]
Each $\widetilde K_j$ is compact and contained in an $\epsilon$-Lipschitz graph over the tangent space $T_{x_j}M\subset\mathbb R^p$. The sets $\widetilde K_j$ are at distance at least $\epsilon_1>0$ from each other and the points of $A$ not covered by any of the $\widetilde K_j$ are contained in the set
\[
 R_k:=A\cap\bigcup_{i=1}^{k_0}\left(\overline{B (x_i,\rho)\setminus B(x_i,(1-\epsilon_1)\rho)}\right).
\]
In order to prove \eqref{decomp_lip1} it remains to prove that for $\epsilon_1\in(0,1)$ small enough, $R_k$ has $\mathcal M_d(R_k)<\epsilon$. Indeed, $R_k$ is a compact subset of $M$ and thus $\mathcal M_d(R_k)=\mathcal H_d(R_k)$ by a known result valid for closed subsets of $d$-rectifiable sets, see \cite[Thm. 3.2.39]{federer}. By \eqref{lip_area_2} we then bound
\begin{eqnarray*}
 \mathcal H_d(R_k)&\le& \sum_{i=1}^{k_0}\mathcal H_d\left(M\cap \overline{B(x_i,\rho)\setminus B (x_i,(1-\epsilon_1)\rho)}\right)\\
 &\le& k_0(1+\epsilon)^d\mathcal H_d\left(B(0,\rho)\setminus B(0,(1-\epsilon_1)\rho)\right)=k_0(1+\epsilon)^d \rho^d|B_1|(1-(1-\epsilon_1)^d),
\end{eqnarray*}
where the right hand side tends to zero as $\epsilon_1\to 0$, verifying that $\epsilon_1>0$ can be chosen small enough so that $\mathcal M_d(R_k)=\mathcal H_d(R_k)<\epsilon$. Therefore we have found a decomposition of $A$ as in \eqref{decomp_lip1}, as desired.
\end{proof}

%
%%%%%%%%%%%%%%%%%%%%%%%%%%%%%%%%%%%%%%%%%%%%%%%%%%%%%%%%%%%%%%%%%%%%%%%%%%%%%%%%%%%%%
\section{Some conjectures and open problems}\label{sec:conj_o_p}
%%%%%%%%%%%%%%%%%%%%%%%%%%%%%%%%%%%%%%%%%%%%%%%%%%%%%%%%%%%%%%%%%%%%%%%%%%%%%%%%%%%%%
%
%
\subsection{Optimal $N$-point configurations for $\mathcal P_s^*(\mathbb S^{p-1},N)$.}
For conjectures regarding $\mathbb S^1$ see Section \ref{sec:resspheres}. The question of what are the $N$-point configurations on $\mathbb R^p$ that optimize $\mathcal P_s^*(\mathbb S^{p-1},N)$ is open, except for the simple cases $N=1,2,3$, in which all points sit at the center of the sphere (see Proposition \ref{prop:small_N}). We conjecture that for $N=p+1$ a regular simplex on a concentric sphere of smaller radius is optimal. Note that for the constrained case of $\mathcal P_s(\mathbb S^{p-1},p+1)$, the inscribed regular simplex is known to be optimal in all dimensions, see \cite{BorSimp} and \cite{yuthesis} for $p=3$.

For $N=5$, conjectures regarding the constrained polarization $\mathcal P_s(\mathbb S^2,5)$ are discussed in   \cite[Chapter 14]{book}. Concerning the problem $\mathcal P_s^*(\mathbb S^2,5)$, based on numerical experiments optimal configurations do not seem to lie on a concentric sphere and in this case it is an open problem to find the geometric structure of optimal configurations.

As mentioned in Proposition \ref{prop:bosuwan*}, the limit of the maximal polarization problem on the sphere for $s\to\infty$ is the question of best \emph{unconstrained covering}. For the sphere, due to Proposition~\ref{prop:constr-unconstr_sphere}, the one-plate and unconstrained best covering problems are equivalent, and thus the former gives information on the latter, and produces useful candidates for the configurations optimizing $\mathcal P_s^*(\mathbb S^{p-1},N)$ for very large $s$. Optimal configurations for the constrained covering of $\mathbb S^2$ were determined for $N=4,6,12$ by L. Fejes T\'oth (see \cite{toth2014regular}), for $N = 5$ and $7$ by Sch\"utte \cite{schutte1955}, for $N = 8$ by L. Wimmer \cite{wimmer2017} and for $N = 10$ and $14$ by G. Fejes T{\'o}th \cite{fejes1969}.

\subsection{The large $N$ limit of optimal polarization configurations}
If $K$ is a lower semicontinuous integrable kernel on $A\times A$ and for each $N\ge 1$ we choose an optimal multiset $\omega_N^*\subset \mathbb R^p$ that realizes the maximum in the definition of $\mathcal P_K^*(A,N)$, where $A\subset\mathbb R^p$ is a compact set of positive $K$-capacity (i.e., there exists some probability measure $\mu$ supported on $A$ whose $K$-potential is $\mu$ integrable), then is it true that every weak-$*$ limit $\mu$ of the sequence
\[
\left\{\frac1N\sum_{x_j\in\omega_N^*}^N\delta_{x_j}\right\}_{N=1}^\infty
\]
satisfies
\[
\min_{y\in A}\int K(x,y)d\mu(X)= T_K^*(A),
\]
where $T_K^*(A):=T_K(A,\mathbb R^p)$?

\subsection{Polarization for lattices in $\R^2$}\label{sec:periodic}
A natural question is the following. Assume $f:(0,\infty)\to(0,\infty)$ is a decreasing convex function and let $K(x,y)=f(|x-y|^2)$. Which lattices $\Lambda\subset\mathbb R^2$ of determinant $1$ maximize the polarization value
\begin{equation}\label{pol_periodic}
\min_{y\in\mathbb R^2} \sum_{x\in \Lambda\setminus\{0\}}f(|x-y|^2)\quad ?
\end{equation}
We note that under rapid decay conditions on $f$ that ensure that the sum in \eqref{pol_periodic} converges, there exist optimizers $\Lambda, y$ that realize the above value. In dimension $d=2$ we conjecture that for completely monotone $f,$ the optimizer of \eqref{pol_periodic} is the hexagonal lattice $A_2$. In \cite{yuthesis} it is shown that the minimum in \eqref{pol_periodic}  for such $f$ and   $\Lambda=A_2$ occurs at the centroids of the equilateral triangles that divide each fundamental domain in half. 

%It is reasonable to contemplate the possibility that, in dimension $d=24,$ the Leech lattice is likewise universally optimal for polarization among lattices in $\mathbb R^{24}$. (For some related results for the case of $f(r)=e^{-\alpha r}$, see \cite{betpet}.)
%
%\medskip
%
%The configuration that realizes the best covering among lattices is known in dimensions $1\le d\le5$ and in all these cases it is the lattice $A_d^*$. It is conjectured that the best covering in $d=24$ is given by the Leech lattice; however in dimension $d=8$ it is known that the best covering amongst lattices is not realized by the $E_8$ lattice, which is outperformed by the $A_8^*$ lattice (see \cite[Chapter 2, Sec. 1.3]{ConwSlBook}).

\subsection{Optimal infinite configurations in $\mathbb R^d$}
Related to the conjecture for $\sigma_{s,2}$ presented in the introduction, it is interesting to explore the generalization of the maximization of \eqref{pol_periodic} for infinite configurations in $\mathbb R^d$. If $\omega_\infty\subset\mathbb R^d$ is a countable configuration such that
\begin{equation}\label{adm_omegainf}
\limsup_{R\to\infty}\frac{\#(\omega_\infty\cap[-R/2,R/2]^d)}{R^d}=1,
\end{equation}
then we define as in \cite{bhrs2016preprint} for $K(x,y)=f(|x-y|^2)$ the polarization constant
\begin{equation}\label{infinite_pol}
P_K(\omega_\infty):=\limsup_{R\to\infty}P_K([-R/2,R/2]^d, \omega_\infty\cap[-R/2,R/2]^d).
\end{equation}
Is it true that under suitable conditions on $f$ the supremum of \eqref{infinite_pol} among $\omega_\infty\subset \mathbb R^d$ satisfying \eqref{adm_omegainf} equals the maximum of \eqref{pol_periodic} over unit density lattices in low dimensions?

\subsection{Weighted unconstrained polarization}

Part (ii) of Theorem \ref{thm:compact_body} can be extended to the case of weighted kernels. This procedure represents a setup, or modification, of the theory presented so far, which allows us to prescribe, or to control, the asymptotic distribution of polarization points at the expense of modifying the kernels $K_s(x,y)=|x-y|^{-s}$ by a suitable weight; i.e. working with $K_s^w(x,y):=w(x,y)|x-y|^{-s}$ where $w(x,y)$ a CPD-weight as defined in \cite[Def. 2.3]{bhrs2016preprint}. Under these conditions, analogues of Theorems \ref{thm:compact_body}, \ref{thm:compact_body_old} and \ref{thm:rectifiable} are expected to hold for $K_s^w$ for the cases $s\ge d$, allowing to relax the hypotheses of \cite[Thm.\,2.3, Thm.\,3.1]{bhrs2016preprint} and to formulate analogues for the unconstrained polarization. We leave this endeavor to future work.

\subsection{Point separation for maximum-polarization configurations}
Is it true that, for $s>p-2$, there exists a constant $c_{s,p}>0$ independent of $N$ such that for any optimizer $\omega_N^*=\{x_{N,1}^*,\ldots,x_{N,N}^*\}\subset\mathbb R^p$ for the problem $\mathcal P_s^*(\mathbb S^{p-1},N)$ we have
\[
\min_{1\le i\neq j\le N}|x_{N,i}^*-x_{N,j}^*|\ge c_{s,p}N^{-1/p}\quad \mbox{for}\quad N=1,2,\ldots\quad ?
\]
The weak separation analogue of the above, giving rise to this question in the constrained polarization problem, has been considered in \cite{2017hrsv}.\\

\section*{Glossary of notation}
\begin{longtable}{rll}
$\omega_N=\{x_1,\ldots,x_N\}$&-&an $N$-point configuration  (multiset) in $\R^p$\\
$\nu(\omega_N)=\frac{1}{N}\sum_{x\in\omega_N}\delta_x$ &-&probability measure associated to a point configuration\\
$A_r:=\{x\in\R^p:\ \op{dist}(x, A)<r\}$&-& $r$-neighborhood of a set, for $A\subset \mathbb R^p$\\
$\mathrm{conv}(A)$&-& convex hull of a set $A$\\
%$\mathrm{dist}(x,A):=\inf\{|x-y|:\ y\in A\}$&-& distance from a point to a set\\
%$\mathrm{dist}(A,B):=\inf\{|x-y|:\ x\in A,y\in B\}$&-& distance between two sets\\
$\mathrm{dist}_{\mathbb S^1}(x,y)=\min\{|t|,|2\pi-t|\}$&-& geodesic distance between $x,y\in\mathbb S^1$ such that $y=e^{it}x$.\\
$\mathcal L_p(A)$&-&  $p$-dimensional Lebesgue measure of set $A$\\
$\mathcal H_d(A)$&-&  $d$-dimensional Hausdorff measure of a set $A$\\
$\beta_k:= \frac{\pi^{k/2}}{\Gamma\left(\frac k2 + 1\right)}$&-& volume of the $k$-dimensional Euclidean ball\\
$P_K(A,\omega_N), \mathcal P_N(A,B,N)$&-&polarization of a configuration,  two-plate polarization,   \eqref{2plate_polarization}.\\
$\mathcal P_K(A,N)$ &-& constrained best $N$-point polarization (single-plate problem) \eqref{ps}\\
$\mathcal P_K^*(A,N)$&-&unconstrained best $N$-point polarization \eqref{psstar}\\
$K_s(x,y)$&-&inverse-power kernel \eqref{def_Ks}\\
$P_s(A,\omega_N), \mathcal P_s(A,N), \mathcal P_s^*(A,N)$&-& see \eqref{def_ps}\\
$\eta_N(A,B), \eta_N(A), \eta_N^*(A)$ &-& two-plate/constrained/unconstrained covering radii, \eqref{relcoverrad}, \eqref{covercover*}\\
$T_K(A,B), T_K(X)$&-& continuum polarization problems \eqref{t_kab}, \eqref{polar_spin}\\
$\tau_{s,d}(N)$&-& scaling factor for the optimal polarization,   \eqref{tau_sd}\\
$\underline h_{s,d}^*(A),\overline h_{s,d}^*(A),h_{s,d}^*(A)$&-& asymptotic values of rescaled optimal polarization,   \eqref{h_sd}\\
%$N_{\mathrm{Bes},p}$&-& bound given by Besicovitch theorem in $\mathbb R^p$ on the number of\\
%&& families of disjoint balls into which a ball covering can be decomposed\\
%$n_{\pi/12,p}$&-& maximum cardinality of a packing of $\mathbb S^{p-1}$ by geodesic balls of radius $\frac{\pi}{12}$\\
$n_{\pi/6,p}$&-& maximum number of  balls with angular radius $\frac{\pi}{6}$ that pack $\mathbb S^{p-1}$\\
$\overline{\mathcal M}_d(A)$, $\underline{\mathcal M}_d(A), \mathcal M_d(A)$&-&Minkowski contents defined in Definition \ref{def:mink}\\
\end{longtable}

\noindent \textbf{Acknowledgement:} The authors thank Alexander Reznikov for his helpful comments and the anonymous referees for their very careful reading of the paper and their suggestions on improving the presentation.

\appendix
\section{Corrigendum to Theorem \ref{thm:compact_body}}\label{app-corrigendum}
The authors are grateful to  Alex Vlasiuk for pointing out that the derivation of equation (5.23) from equation (5.22) in the proof of Theorem \ref{thm:compact_body} did not take into account the measure of the boundary of  $A$. We provide here, in Proposition \ref{mainprop} below, a substitute for this derivation valid in the case $s>p$, whereas for $s=p$ we add to Theorem \ref{thm:compact_body} the additional hypothesis $\mathcal L^p(\partial A)=0$; namely, that $A$ is Jordan-measurable. The amended statement of Theorem \ref{thm:compact_body} is therefore as follows:
\begin{theorem}[{Replacement of Theorem \ref{thm:compact_body}}]\label{mainthm}
If $A\subset\mathbb R^p$ is a compact set and $s>p$, or if $s=p$ and $\mathcal L^p(\partial A)=0$, then
\begin{equation}\label{asymptotics}
h_{s,p}^*(A)=h_{s,p}(A)=\frac{\sigma_{s,p}}{\mathcal L_p(A)^{s/p}}.
\end{equation}
Moreover, if $\mathcal L_p(A)>0$, then for any asymptotically extremal sequence $\Omega=\{\omega_N\}_{N\ge 1}$ (for either the constrained or unconstrained polarization problem)  we have the weak-$*$ convergence
\begin{equation}
\frac{1}{N}\sum_{x_i\in\omega_N}\delta_{x_i}\stackrel{*}{\rightharpoonup}\frac{\mathcal L_p|_A}{\mathcal L_p(A)}\quad\mbox{as}\quad N\to\infty,
\end{equation}
where $\mathcal L_p|_A:=\mathcal{L}_p(\cdot\cap A)$ is the restriction to $A$ of  $\mathcal L_p$.

\end{theorem}
Note that there is no difference between the statement of  Theorem \ref{mainthm} above and that of Theorem \ref{thm:compact_body} in the case $s>p$. As for the case $s=p$, the  modified assumption has no impact for the remaining results of the paper as this case only arises in Theorem \ref{thm:compact_body}.\\

 The proof of Theorem \ref{mainthm} follows exactly like the one of Theorem \ref{thm:compact_body}, except for the following changes:
\begin{itemize}
 \item For the case $s=p$, with the further hypothesis $\mathcal L^p(\partial A)=0$ in Theorem \ref{mainthm}, we can take the sets $G_i:=A\cap B_i$ in the paragraph following (5.12), and the proof holds verbatim.
 \item For the case $s>p$, Proposition \ref{mainprop} below, applied to the sets $A\cap B_i$ and $B_i$ from eq. (5.22) allows to replace eq. (5.23) therein. The new constant $c_{s,p}$ from Proposition \ref{mainprop} below replaces the constant $2$ in eq. (5.23) after which the proof of Theorem \ref{thm:compact_body} follows with no further modifications.
\end{itemize}

\medskip

The new result needed for the case $s>p$ is the following:
\begin{proposition}\label{mainprop}
For $s> p\ge 1$, there exists a constant $c_{s,p}>0$ with the following properties. Let $\epsilon\in(0,1)$ and $B\subset \mathbb R^p$ be a ball and $A\subset B$   a closed set such that $\mathcal L^p(B\setminus A)<\epsilon \mathcal L^p(B)$. Then there holds
\begin{equation}\label{main1}
\bar h^*_{s,p}(A)\le (1-c_{s,p}{\epsilon^{\frac1{p+1}}}) \bar h^*_{s,p}(B).
\end{equation}
\end{proposition}

\medskip

%\section*{Proof of Proposition~\ref{mainprop}}
The proof of Proposition \ref{mainprop} is based on two lemmas. For this section, we consider a closed ball $B\subset \mathbb R^p$ and a closed subset $A\subset B$, such that $\mathcal L^p(B\setminus A)<\epsilon$. Furthermore, let $\mathcal P^*_s(A,N)$ be the optimum $N$-point polarization of set $A$, and let $\omega_N^A$ be an $N$-point configuration such that
\begin{equation}\label{omegana}
 \mathcal U_{\omega_N^A}(y):=\sum_{x \in \omega_N^A}|x-y|^{-s}=\mathcal P^*_s(A,N).
\end{equation}
\begin{lemma}\label{lem1}
 Let $s>0$, ${\delta}\in(0,1)$ and let $N$ be a positive integer. If $y\in\mathbb R^p$ is such that
 \begin{equation}\label{disty}\mathrm{dist}(y,A)<\frac{{\delta}}s\ (\mathcal P_s^*(A,N))^{-\frac1s}
 \end{equation}
 then
 \begin{equation}\label{lbd}
  \mathcal U_{\omega_N^A}(y) \ge (1-{\delta})\mathcal P^*_s(A,N).
 \end{equation}
\end{lemma}
\begin{proof}
First note that if $\mathrm{dist}(y,\omega_N^A)<\left(\mathcal P^*_s(A,N)\right)^{-\frac1s}$ then
\[
 \mathcal U_{\omega_N^A}(y)\ge \mathcal P^*_s(A,N),
\]
and thus \eqref{lbd} holds {\it a fortiori}. Therefore from now on we consider points $y\in \mathbb R^p$ such that \eqref{disty} and $\mathrm{dist}(y,\omega_N^A)\ge \left(\mathcal P^*_s(A,N)\right)^{-\frac1s}$ hold, and our goal is to prove \eqref{lbd} for such $y$.

\medskip
Let $y_1\in A$ be such that $|y_1-y|=\mathrm{dist}(y,A)$ and let
\begin{equation}\label{y2max}
 y_2\in \mathrm{argmax}_{y'\in[y,y_1]}\mathcal U_{\omega_N^A}(y').
\end{equation}
We claim that the following chain of inequalities holds:
\begin{eqnarray}
 \mathcal U_{\omega_N^A}(y)&\ge&\mathcal U_{\omega_N^A}(y_2 ) -\int_0^1 \left|\nabla \mathcal U_{\omega_N^A}(y+t(y_2-y))\cdot (y_2-y)\right|dt\label{uybd0}\\
 &\ge&\mathcal U_{\omega_N^A}(y_2 ) -s|y_2-y|\left[\min_{y'\in[y,y_2]}\mathrm{dist}(y',\omega_N^A)\right]^{-1}\max_{y'\in[y,y_2]}\mathcal U_{\omega_N^A}(y')\label{uybd1}\\
 &\ge&\left(1-{\delta}\right)\mathcal U_{\omega_N^A}(y_2)\ge\left(1-{\delta} \right)\mathcal P_s^*(A,N).\label{uybd2}
\end{eqnarray}
We now prove the above. The bound \eqref{uybd0} follows by Taylor expansion. Inequality \eqref{uybd1} follows by noting that whenever $x\neq y'$ we have $|\nabla_y|x-y'|^{-s}|=s|x-y'|^{-s-1}$, therefore for $y'\in[y,y_2]$ we have
\[
 \left|\nabla \mathcal U_{\omega_N^A}(y')\right|\le s\sum_{x\in \omega_N^A}|x-y'|^{-s-1}\le s\left[\min_{y'\in[y,y_2]}\mathrm{dist}(y',\omega_N^A)\right]^{-1}\mathcal U_{\omega_N^A}(y').
\]
The first inequality in \eqref{uybd2} follows by using definition \eqref{y2max} of $y_2$, and the bounds following from our hypotheses on $y$: $|y_2-y|\le |y_1-y|=\mathrm{dist}(y,A)\le \frac{{\delta}}s(\mathcal P_s^*(A,N))^{-\frac1s}$ and $\mathrm{dist}(y,\omega_N^A)\ge(\mathcal P^*_s(A,N))^{-\frac1s}$. The second inequality in \eqref{uybd2} follows by the fact that $y_1\in A$ and the definition of $y_2$, and of $\omega_N^A$:
\[
\mathcal P^*_s(A,N)=\min_{x\in A} \mathcal U_{\omega_N^A}(x)\le \mathcal U_{\omega_N^A}(y_1)\le \mathcal U_{\omega_N^A}(y_2)
\]
\end{proof}
Recall the notation, for the $r$-neigborhood of a closed set $K\subset \mathbb R^d$, for $r>0$:
\[
 (K)_r:=\{x\in\mathbb R^d:\ \mathrm{dist}(x,K)<r\}.
\]
\begin{lemma}\label{lem2}
 There exists a constant $c_p>0$ with the following properties. Let $B\subset\mathbb R^p$ be a ball and let $A\subset B$ be a closed set, such that for some $\epsilon \in(0,1)$ there holds $\mathcal L^p(B\setminus A)<\epsilon \mathcal L^p(B)$. Then for each $r\in (0, \epsilon\ \mathrm{diam}(B))$ we can cover all of $B\setminus (A)_r$ by at most $c_p \frac{\epsilon \mathcal L^p(B)}{r^p}$ balls of radius $r$.
\end{lemma}

\begin{proof}
 Let $r':=r/\sqrt{p}$ and we show that we may take as the set of ball centers the following:
 \[
  W:=\{r'k:\ k\in \mathbb Z^p,\ (r'k+[-r',r']^p)\cap \left(B\setminus (A)_r\right)\neq \emptyset\}.
 \]
Equivalently, $W$ is formed by those vertices of $(r'\mathbb Z^p)$-grid cubes of the form $r'(k+[0,1]^p)$ which meet $B\setminus (A)_r$.

\medskip

We note that the balls with centers in $W$ and radius $\frac{r'}2$ are disjoint and contained in the $r'$-neighborhood of $B\setminus (A)_r$. Furthermore, we have the inclusion
\[
 (B\setminus (A)_r)_{r'}\subset (B)_{r'}\setminus A=((B)_{r'}\setminus B) \cup (B\setminus A),
\]
from which it follows that, denoting by $\# W$ the cardinality of $W$,
\[
\# W\cdot  \mathcal L^p(B_{\frac{r'}2}) \le \mathcal L^p((B)_{r'}\setminus B) + \mathcal L^p(B\setminus A) \le \left(C_p\frac{r'}{\mathrm{diam}(B)} + \epsilon\right)\mathcal L^p(B)\le (C_p+1)\ \epsilon\ \mathcal L^p(B).
\]
This implies that for $c_p:=2^p(C_p+1)/(p^{\frac{p}2}\beta_p)$, in which $\beta_p$ is the volume of the unit ball in $\mathbb R^p$, there holds
\[
 \# W \le c_p\epsilon \frac{\mathcal L^p(B)}{r^p}.
\]
It remains to show that radius-$r$ balls with centers in $W$ cover $B\setminus (A)_r$. Indeed, note that if the cube $r'k + [-r',r']^p$ with $k\in \mathbb Z^p$ meets $B\setminus (A)_r$, then it is contained in the ball $B(r'k, \sqrt{p}\ r')=B(r'k,r)$, and thus
\[
 B\setminus (A)_r\subset\bigcup_{r'k\in W}(r'k +[-r',r']^p) \subset \bigcup_{r'k\in W}B(r'k,r),
\]
as desired.
\end{proof}

\begin{proof}[Proof of Proposition \ref{mainprop}:]
Observe that by the same proof as in \cite[Thm. 3.4]{2013erdelyisaff}, which applies also without the restriction that the optimum polarization points belong to $A$, with $\mathcal L^p$ used as the measure $\mu$ in the proof, there exists a constant $C_{s,p}>0$ independent of $N,A$, such that $\mathcal P^*_s(A,N)\le C_{s,p}N^{s/p}/(\mathcal L^p(A))^{s/p}$. Now, applying Lemma \ref{lem1} {with $\delta:=\epsilon^{\frac1{p+1}}$} we find that for the optimum configuration $\omega_N^A$ like in \eqref{omegana}, with
\begin{equation}\label{defrn}
 r_N:=\frac{{\epsilon^{\frac1{p+1}}}}{s}\left(\mathcal P_s^*(A,N)\right)^{-\frac1s}\ge \frac{{\epsilon^{\frac1{p+1}}}\,(\mathcal L^p(A))^{\frac1p}}{s(C_{s,p})^{\frac1s}} N^{-\frac1p},
\end{equation}
we have
\begin{equation}\label{1}
\forall y\in (A)_{r_N},\quad \mathcal U_{\omega_N^A}(y)\ge\left(1-{\epsilon^{\frac1{p+1}}}\right) \mathcal P_s^*(A,N).
\end{equation}
Next, for $N$ large enough so that $r_N<\epsilon\ \mathrm{diam}(B)$ we apply Lemma \ref{lem2} with $r_N$ playing the role of $r$, and we find a set of centers $W$ such that, using also \eqref{defrn} in the second inequality below:
\begin{equation}\label{propw}
 B\setminus (A)_{r_N}\subset \bigcup_{x\in W}B(x,r_N),\qquad \# W \le c_p\ \epsilon \frac{\mathcal L^p(B)}{r_N^p}\le \widetilde C_{s,p}\  {\epsilon^{\frac1{p+1}}}\ \frac{\mathcal L^p(B)}{\mathcal L^p(A)}\ N\le \widetilde C_{s,p}\  {\epsilon^{\frac1{p+1}}} (1-\epsilon) N.
\end{equation}
By considering the new configuration $\omega_{\widetilde N}:=\omega_N^A \cup W$ whose cardinality is denoted $\widetilde N$, we find that
\begin{equation}\label{omegatilde}
 \widetilde N\in \left[N,N\left(1+\widetilde C_{s,p} {\epsilon^{\frac1{p+1}}(1-\epsilon)}\right)\right], \quad \mathcal U_{\omega_{\widetilde N}}(y)\ge \left\{\begin{array}{ll}
 \left(1-{\epsilon^{\frac1{p+1}}}\right)\mathcal P_s^*(A,N)&\text{ for }y\in (A)_{r_N},\\[3mm]
 \mathcal P_s^*(A,N)&\text{ for }y\in B\setminus (A)_{r_N},
 \end{array}\right.
\end{equation}
in which for the first inequality we use \eqref{1} and for the second one we use the first part of \eqref{propw} and the fact that for $x\in W, y\in B(x,r_N)$ there holds, since $\epsilon<1\le p<s$,
\[
\mathcal U_{\omega_{\widetilde N}}(y)\ge |x-y|^{-s}\ge r_N^{-s}\ge \left(\frac{{\epsilon^{\frac1{p+1}}}}s\right)^s\mathcal P_s^*(A,N)\ge \mathcal P_s^*(A,N).
\]
We then find that, using also the fact that $\epsilon\in(0,1)$,
\begin{equation}\label{pbntilde}
 \frac{\mathcal P_s^*(B,\widetilde N)}{\widetilde N^{s/p}}\ge{\frac{1-\epsilon^{\frac1{p+1}}}{\left(1+\widetilde C_{s,p}\,\epsilon^{\frac1{p+1}}(1-\epsilon)\right)^{s/p}}}\frac{\mathcal P_s^*(A,N)}{N^{s/p}}\ge {(1-\epsilon^{\frac1{p+1}})\left(1-\frac{s}{p}\,\widetilde C_{s,p}\epsilon^{\frac1{p+1}}\right)}\frac{\mathcal P_s^*(A,N)}{N^{s/p}},
\end{equation}
from which the bound \eqref{main1} follows, with $c_{s,p}=1+\frac{s}{p}\,\widetilde C_{s,p}$.

\end{proof}

\end{document}